\documentclass[12pt]{article}
\usepackage{amsmath}
\usepackage{amsthm,amsfonts,amssymb,epsfig,graphics}
\usepackage{pdfsync}
\usepackage{wrapfig}
\usepackage[utf8]{inputenc}
\usepackage[babel=true,kerning=true]{microtype} 
\usepackage{pgf,tikz,pgfplots}
\usetikzlibrary{arrows}

\usepackage[colorlinks=true,urlcolor=blue]{hyperref}
\usepackage{multicol}

\def\cal{\mathcal}

\newcommand{\bbD}{\mathbb{D}}
\newcommand{\bbH}{\mathbb{H}}

\newcommand{\bbR}{\mathbb{R}}

\newcommand{\bbZ}{\mathbb{Z}}

\def\cA{{\cal A}}    \def\cS{{\cal S}}
  \def\cH{{\cal H}}  
   \def\cO{{\cal O}}

\def\cF{{\cal F}}

\def\t{\widetilde}

\renewcommand{\phi}{\varphi}
\renewcommand{\epsilon}{\varepsilon}

\def\homeo{\mathrm{Homeo}}

\newtheorem{theo}{Theorem}

\newtheorem*{ques*}{Question}
\newtheorem*{prop*}{Proposition}
\newtheorem*{conj*}{Conjecture}
\newtheorem*{theo*}{Theorem}
\newtheorem{coro}{Corollary}[section]

\newtheorem{prop}[coro]{Proposition}
\newtheorem{lemm}[coro]{Lemma}
\newtheorem*{lemm*}{Lemma}
\newtheorem*{coro*}{Corollary}
\def\?{\footnote{?}}

\newcommand{\footnotefred}[1]{}

\newcounter{prblm}
\renewcommand{\theprblm}{\arabic{prblm}}

\newlength{\espaceavantenonce}
\newlength{\espaceapresenonce}
{\vskip \espaceapresenonce}

\newenvironment{enonce3*}[1]{
\vskip\espaceavantenonce \noindent \textbf{\textit{#1.---}} }%
{\vskip \espaceapresenonce}

\newcounter{numexo}
\newcounter{numsubexo}
\newcounter{numsubsubexo}
\begin{document} \sloppy

\author{Fr\'ed\'eric Le Roux}
 \title{An index for Brouwer homeomorphisms and homotopy Brouwer theory}
\maketitle


\begin{abstract}
We use the homotopy Brouwer theory of Handel to define a Poincar\'e index between pairs of orbits for an orientation preserving fixed point free homeomorphism of the plane. Furthermore, we prove that this index is almost additive.
\end{abstract}

\tableofcontents

\newpage
\setcounter{section}{-1}
\section{Description of the results}\label{section.results}

\subsection{The Poincar\'e index for foliations}

Let $X$ be a (smooth) vector field on the plane. If a point $x_{0}$ is an isolated zero of $X$, then one can define the \emph{Poincar\'e-Hopf index} of $X$ at $x_{0}$ as the winding number of the vector $X(x)$ when the point $x$ goes once around the singularity $x_{0}$. This number is a conjugacy invariant: if $\Phi$ is a diffeomorphism, then the Poincar\'e-Hopf index of the image vector field $\Phi_{*}X$ at $\Phi(x_{0})$ is equal to the Poincar\'e-Hopf index of $X$ at $x_{0}$. 
\def\echantillons{100}
\begin{figure}[ht]
\begin{center}
\begin{tikzpicture}[scale=0.3]
\tikzstyle{fleche pointille}=[>=latex,->,dashed]
\draw plot [variable=\t,domain=0:360,samples=\echantillons] ({5*cos(\t)},{5*sin(\t)}) ;
\draw plot [variable=\t,domain=0:360,samples=\echantillons] ({10*cos(\t)},{2.5*sin(\t)}) ;
\draw node (0,0) {\footnotesize $\times$} ;
\draw node (0,0) [right] {$x_{0}$} ;
\def\N{30}
\foreach \t in {0,...,\N}  
{\draw [fleche pointille]   ({5*cos(\t*360/\N)},{5*sin(\t*360/\N)}) --    ({10*cos(\t*360/\N)},{2.5*sin(\t*360/\N)})         ;}
\end{tikzpicture}
\end{center}
\caption{Index of a vector field or a homeomorphism along a curve}
\end{figure}
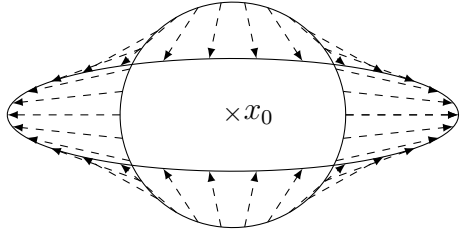
Analogously, if $h$ is a homeomorphism of the plane, and $x_{0}$ is an isolated fixed point of $h$, the \emph{Poincar\'e-Lefschetz index} of $h$ at $x_{0}$ is defined as the winding number of the vector from $x$ to  $h(x)$ when the point $x$ goes once around the singularity $x_{0}$. Again, for every homeomorphism $\Phi$, the Poincar\'e-Lefschetz index of $\Phi h \Phi^{-1}$ at $\Phi(x_{0})$ is equal to the Poincar\'e index of $h$ at $x_{0}$.

Now assume that the planar vector field $X$ has no singularity. Let $\cF$ be the foliation of the plane by trajectories of $X$. According to the Poincar\'e-Bendixson theory, every leaf of $\cF$ is properly embedded in the plane. As a consequence, given two distinct leaves $F_{1}, F_{2}$, one can find an orientation preserving diffeomorphism $\Phi$
that sends $F_{1}, F_{2}$ respectively to the lines $\bbR\times \{1\}, \bbR\times \{2\}$. Consider the vector field $\Phi_{*} X$. On the lines $\Phi(F_{1}), \Phi(F_{2})$ this vector field is horizontal. Thus, given any curve $\gamma$ connecting those two lines, the winding number of $X$ along $\gamma$ is an integer or a half integer. Let us denote this number by $I(\Phi_{*} X,\gamma)$.
An easy connectedness argument shows that $I(\Phi_{*} X,\gamma)$ does not depend on the choice of the curve $\gamma$. We claim that this number does not depend either on the choice of the diffeomorphism $\Phi$. This follows from another connectedness argument; the crucial observation is that the space of changes of coordinates, namely orientation preserving diffeomorphisms that globally preserves both lines $\bbR\times \{1\}, \bbR\times \{2\}$, is connected. In conclusion, the number $I(\Phi_{*} X,\gamma)$ defines an index associated to the foliation $\cF$ and the pair of leaves $(F_{1}, F_{2})$, and this index is again a conjugacy invariant.
This index counts the algebraic number of ``Reeb components'' in the area between the two leaves (see Figure~\ref{fig.reeb-components}).

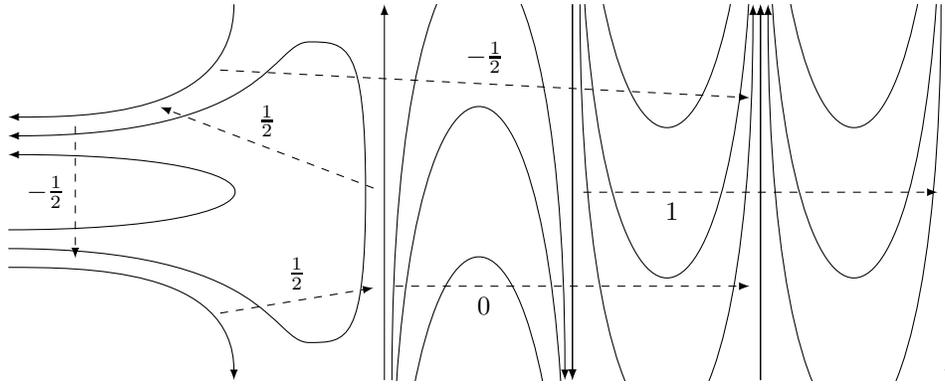
\begin{figure}[h]
\begin{center}
\begin{tikzpicture}[scale=0.5]
\tikzstyle{fleche}=[>=latex,->]
\tikzstyle{fleche pointille}=[>=latex,->,dashed]
\tikzstyle{flecheinverse}=[>=latex,<-]

\begin{scope}
\clip (-10,0) rectangle (0,10) ;
\draw [fleche] (-4,10) .. controls +(down:3) and +(right:2) .. (-10,7)  node [near start,name=debut feuille 1] {} node [midway,name=milieu feuille 1] {} node [near end,name=fin feuille 1] {}  ;
\draw [flecheinverse] (-4,0) .. controls +(up:3) and +(right:2) .. (-10,3)  node [near start,name=fin feuille 2] {}  node [near end,name=debut feuille 2] {}    ;
\draw [fleche] (-10,3.5) .. controls +(right:7)  and +(left:1) .. 
			(-2,1) .. controls +(right:1)  and +(down:4) .. 
			(-0.5,5)  .. controls +(up:4)  and +(right:1) .. 
			(-2,9)   .. controls +(left:1)  and +(right:7) .. 
			(-10,6.5) ; 
\draw [fleche] (-10,4) .. controls +(right:8)  and +(right:8) .. (-10,6) ; 
\end{scope}

\begin{scope}
\clip (-0.5,0) rectangle (5.5,10) ;
\draw [fleche] (0,0) -- (0,10)  node [near end,name=fin feuille 3] {} node [near start,name=debut feuille 3] {}  node [midway,name=milieu feuille 3] {} ; 
\draw [fleche] (5,10) -- (5,0) node [midway,name=milieu feuille 4] {} ;
\def \control {15} 
\foreach \translation in {0, -4, -8}
 	{\draw [fleche] (0.2,\translation) .. controls +(up:\control) and +(up:\control) .. (4.8,\translation) ;}
\end{scope}

\begin{scope}[xshift= 10cm,yshift=10cm,rotate=180]
\clip (-0.5,0) rectangle (5.5,10) ;
\draw [flecheinverse] (0,0) -- (0,10) node [near start,name=fin feuille 5] {} node [near end,name=debut feuille 5] {}  ; 
\draw [flecheinverse] (5,10) -- (5,0) ;
\def \control {15} 
\foreach \translation in {0, -4, -8}
 	{\draw [flecheinverse] (0.2,\translation) .. controls +(up:\control) and +(up:\control) .. (4.8,\translation) ;}
\end{scope}

\begin{scope}[xshift= 15cm,yshift=10cm,rotate=180]
\clip (-0.5,0) rectangle (5.5,10) ;
\draw [fleche] (0,0) -- (0,10)   node [midway,name=milieu feuille 6] {} ;
\draw [fleche] (5,10) -- (5,0) ;
\def \control {15} 
\foreach \translation in {0, -4, -8}
 	{\draw [fleche] (0.2,\translation) .. controls +(up:\control) and +(up:\control) .. (4.8,\translation) ;}
\end{scope}

\draw [fleche pointille]  (debut feuille 1) -- (fin feuille 5)  node [midway,above] {\footnotesize $-\frac{1}{2}$} ;
\draw [fleche pointille]  (fin feuille 1) -- (debut feuille 2)  node [midway,left] {\footnotesize $-\frac{1}{2}$} ;
\draw [fleche pointille]  (fin feuille 2) -- (debut feuille 3)  node [midway,above] {\footnotesize $\frac{1}{2}$} ;
\draw [fleche pointille]  (milieu feuille 3) -- (milieu feuille 1)  node [midway,above] {\footnotesize $\frac{1}{2}$} ;
\draw [fleche pointille]  (debut feuille 3) -- (debut feuille 5)  node [near start,below] {\footnotesize $0$} ;
\draw [fleche pointille]  (milieu feuille 4) -- (milieu feuille 6)  node [near start,below] {\footnotesize $1$} ;

\end{tikzpicture}
\end{center} 
\caption{\label{fig.reeb-components} Examples of indices for pairs of leaves of a particular planar foliation}
\end{figure}

Can we find an analogous definition for homeomorphisms? The discrete counterpart of non vanishing planar vector fields are \emph{Brouwer homeomorphisms}, \emph{i.e.} orientation preserving fixed point free homeomorphims of the plane. There is a profound analogy between the dynamics of Brouwer homeomorphisms and the topology of non singular plane foliations, beginning with the work of Brouwer (\cite{brouwer1912beweis}). In particular, Brouwer proved that every orbit of a Brouwer homeomorphism $h$ is properly embedded, that is, for every $x$, the sequences $(h^{n}(x))_{n \leq 0}$ and $(h^{n}(x))_{n \geq 0}$ tends to infinity\footnote{We call infinity the point added in the Alexandroff one-point compactification of the plane; thus the previous sentence means that every compact subset of the plane contains only finitely many terms of the orbit.}. This is of course reminiscent of the Poincar\'e-Bendixson theorem. The analogy has been developed by several authors, see for example~\cite{homma1953structure,lecalvez2004version,leroux2005structure}.
Note however that, unlike foliations, the dynamics of Brouwer homeomorphisms is rich enough to prevent any attempt for a complete classification (see~\cite{leroux2001classification,beguin2003ensemble}).
The aim of the present paper is to define an index for pairs of orbits of Brouwer homeomorphisms that generalizes the Poincar\'e-Hopf index for pairs of leaves of a plane foliation.
This index is much probably related to the algebraic number of \emph{generalized Reeb components} (as defined in~\cite{leroux2005structure}) separating the two orbits.  But it is unclear to the author how to write a precise definition of the index using generalized Reeb components, whereas homotopy Brouwer theory seems to be the perfect tool.

\subsection{Homotopy Brouwer theory}

In the foliation setting, the key to the definition of the index between pairs of leaves is the existence of a map $\Phi$ that straightens both leaves to euclidean lines, and the connectedness of the space of changes of coordinates. Given a Brouwer homeomorphism $h$ and two full orbits $\cO_{1}, \cO_{2}$ of $h$, it is easy to find a homeomorphism $\Phi$ that sends $\cO_{1}$ to $\bbZ\times \{1\}$ and $\cO_{2}$ to $\bbZ\times \{2\}$. 
For a curve $\gamma$ starting at a point of $\bbZ \times \{1\}$ and ending at a point of $\bbZ \times \{2\}$, we denote by $I(h', \gamma)$ the index of $h'$ along $\gamma$, that is, the winding number of the vector  from $x$ to $h' (x)$ when the point $x$ runs along the curve $\gamma$. Note that $h'$ acts as a horizontal translation on $\bbZ \times \{1\}$ and $\bbZ \times \{2\}$, so that the number $I(h', \gamma)$ is an integer or a half integer. It is tempting to define the index of $h$ between the pair of orbits $\cO_{1}$ and $\cO_{2}$ as the number $I(h', \gamma)$. However, this time the index $I(h', \gamma)$ depends on the choice of  $\Phi$ (and generally also on the choice of $\gamma$, see Figure~\ref{fig.bad-coordinates}). Also note that the corresponding space of changes of coordinates is not connected. Thus one has to give additional conditions on the map $\Phi$ to exclude the ``bad'' maps as the one on Figure~\ref{fig.bad-coordinates}.
These conditions will be provided by Handel's homotopy Brouwer theory. 
\begin{figure}[h]
\begin{center}
\begin{tikzpicture}[scale=1.5]
\tikzstyle{fleche pointille}=[>=latex,->,dashed]

\draw plot [variable=\t,domain=0:360,samples=\echantillons] ({(1.5*cos(\t)+0.5-0.5)*180/(\t+180)+0.5} , {(0.75*sin(\t) + 2 -2 )*180/(\t+180) +2 }) ;
\draw [rotate around={180:(0.5,2)}] plot [variable=\t,domain=0:360,samples=\echantillons] ({(1.5*cos(\t)+0.5-0.5)*180/(\t+180)+0.5} , {(0.75*sin(\t) + 2 -2 )*180/(\t+180) +2 }) ;
\draw (-3,2) -- (-1,2) ;
\draw (0,2) -- (1,2) ;
\draw (2,2) -- (4,2) ;
\draw (-3,1) -- (4,1) ;
\foreach \k in {-3,...,4} 
	{\draw (\k,1) node {\footnotesize $\bullet$} ; \draw (\k,2) node {\footnotesize $\times$} ; }
\draw [fleche pointille]  (0,1) -- (0,2)  node [near start,right] {\footnotesize Index $1$} ;
\draw [fleche pointille]  (2,1) -- (2,2)  node [midway,right] {\footnotesize Index $0$} ;

\end{tikzpicture}
\end{center}
\caption{A bad change of coordinates for the translation $T$: if $\Phi$ is a Dehn twist around the segment $[0,1] \times \{2\}$, the index of $\Phi T \Phi^{-1}$ between $(0,1)$ and  $(0,2)$ is one, whereas the index of $T$ between those same points vanishes}
\label{fig.bad-coordinates}
\end{figure}
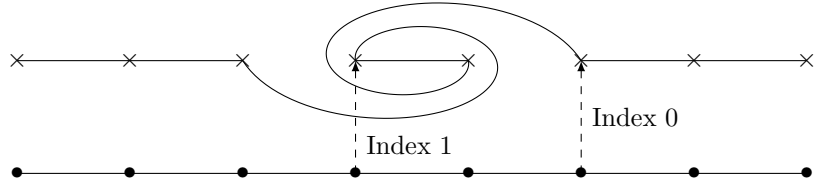

The general setting for homotopy Brouwer theory is the following. 
Let $h$ be an orientation preserving homeomorphism of the plane.
Choose distinct orbits $\cO_{1}, \dots \cO_{r}$ of $h$ and assume that they are proper: for every $x_{i} \in \cO_{i}$,  the sequences $(h^{n}(x_{i}))_{n \leq 0}$ and $(h^{n}(x_{i}))_{n \geq 0}$ tend to infinity (remember that this is automatic if $h$ is a Brouwer homeomorphism). 
Let $\homeo^{+}(\bbR^{2}; \cO_{1}, \dots,  \cO_{r})$ denote the group of orientation preserving homeomorphisms $g$ of the plane that globally preserve each $\cO_{i}$, that is, $g(\cO_{i}) = \cO_{i}$ for each $i$. Let $\homeo^{+}_{0}(\bbR^{2}; \cO_{1}, \dots,  \cO_{r})$ denote the identity component of this group: more explicitly, a homeomorphism $g$ belongs to this subgroup if there exists an isotopy $(g_{t})_{t \in [0,1]}$ such that  $g_{0}$ is the identity and $g_{1} = g$, and for every $t$, $g_{t}$ fixes every point of $\cO = \cup_{i}\cO_{i}$.
We denote the quotient group by
$$
MCG(\bbR^{2}; \cO_{1}, \dots,  \cO_{r}) = \homeo^{+}(\bbR^{2}; \cO_{1}, \dots,  \cO_{r}) \Big{/} \homeo^{+}_{0}(\bbR^{2}; \cO_{1}, \dots,  \cO_{r}).
$$
The homeomorphism $h$ determines an element of this group which is called the \emph{mapping class of $h$ relative to  $\cO$} and denoted by $[h; \cO_{1}, \dots , \cO_{r}]$ (or simply $[h]$ when the set of orbits is clear from the context). Two elements of the same mapping class will be said to be \emph{isotopic relative to $\cO$}.
Two mapping classes $[h; \cO_{1}, \dots , \cO_{r}]$ and $[h'; \cO'_{1}, \dots , \cO'_{r}]$ are \emph{conjugate} if there exists an orientation preserving homeomorphism $\Phi$ such that $\Phi( \cO_{i}) = \cO'_{i}$ for each $i$, and the mapping classes
$[\Phi h \Phi^{-1}; \cO'_{1}, \dots , \cO'_{r}]$ and $[h'; \cO'_{1}, \dots , \cO'_{r}]$ are equal. The mapping class of a Brouwer homeomorphism is  called a \emph{Brouwer mapping class}. One of the purposes of homotopy Brouwer  theory is to give a description of Brouwer mapping classes up to conjugacy. The definition of the index, given in the next subsection, will necessitate the classification relative to \emph{two} orbits, which has been provided by Handel in~\cite{handel1999fixed}. In subsection~\ref{subsection.additivity} we will state a quasi-additivity property for the index.  The classification relative to \emph{three} orbits will be the key to prove this property, we will describe it in subsection~\ref{subsection.three}.

\subsection{Definition of the index}\label{subsection.index}
\begin{figure}[h]
\begin{center}
\begin{tikzpicture}[scale=1.5]
\tikzstyle{fleche}=[>=latex,->]
\draw (-3,1) -- (4,1) ;
\draw (-3,2) -- (4,2) ;
\foreach \k in {-3,...,4} 
	{\draw (\k,1) node {$\bullet$} ; \draw (\k,2) node {$\times$} ; }
\def\depl{0.2}
\draw [fleche,color=blue,dashed] (-2,2+\depl) -- (-1, 2+\depl) node [midway,above] {$T$} ;
\draw [fleche,color=blue,dashed] (-2,1+\depl) -- (-1, 1+\depl) node [midway,above] {$T$} ;

\draw [fleche,color=blue,dashed] (3,2+\depl) -- (2, 2+\depl) node [midway,above] {$R$} ;
\draw [fleche,color=blue,dashed] (2,1+\depl) -- (3, 1+\depl) node [midway,above] {$R$} ;

\end{tikzpicture}
\end{center}
\caption{The orbits $\bbZ \times \{1\}$ and $\bbZ \times \{2\}$, $T$ and $R$}
\label{fig.translation-reeb}
\end{figure}
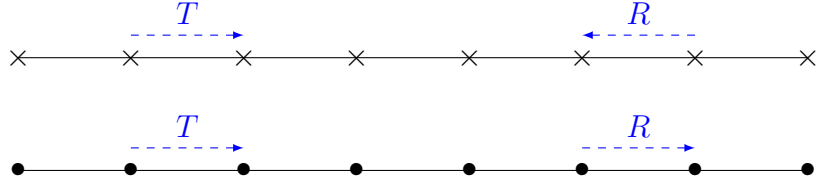

 Let $T$ be the translation by one unit to the right, and $R$ be the ``Reeb homeomorphism'' which is partially described on Figure~\ref{fig.translation-reeb}.
The only relevant features for our purposes are the actions of $T$ and $R$ on the lines $\bbR \times \{1\},\bbR \times \{2\}$. We denote by $[T], [T^{-1}], [R], [R^{-1}]$ the mapping classes relative to the orbits $\bbZ \times \{1\}$ and $\bbZ \times \{2\}$.
We can now state the classification of Brouwer mapping classes relative to two orbits\footnote{Note that the context is slightly different from~\cite{handel1999fixed}, because  (unlike Handel) we restrict ourseves to conjugacy under \emph{orientation preserving} homeomorphisms that globally preserve \emph{each} orbit. The reader may easily deduce the given statement from Handel's one, or refer to section~\ref{subsection.explicit} below. Furthermore, the mapping classes $[T]$ and $[T^{-1}]$ are actually conjugate, as we will see in proposition~\ref{prop.centralisateur2}, but we do not need this fact for the moment.}.
\begin{theo*}[Handel,~\cite{handel1999fixed}, Theorem 2.6]
Every Brouwer mapping class  relative to two orbits is conjugate to $[T]$, $[T^{-1}]$, $[R]$ or $[R^{-1}]$.
\end{theo*}

\begin{figure}[h]
\begin{center}
\begin{tikzpicture}[scale=0.8]
\tikzstyle{fleche}=[>=latex,->]

\begin{scope}[xshift=-120]
\coordinate (A1) at (0,3);
\coordinate (A2) at (-2,-1);
\coordinate (A3) at (0,0);
\coordinate (A4) at (2,0);
\coordinate (A5) at (3,-3);

\coordinate (B1) at (-3,-3);
\coordinate (B2) at (-1,1);
\coordinate (B3) at (1,0);
\coordinate (B4) at (0,-2);
\coordinate (B5) at (2,-4);

\foreach \p in {1,...,5}
{\node at (A\p)  {$\bullet$} ;
\node at (A\p) [below]  {$\p$} ;
\node at (B\p)  {$\times$} ;
\node at (B\p) [above]  {$\p$} ;
}

\end{scope}

\draw (0,-3) -- (0,3) ;

\begin{scope}[xshift=120]
\coordinate (A1) at (0,3);
\coordinate (A2) at (-2,-1);
\coordinate (A3) at (0,0);
\coordinate (A4) at (2,0);
\coordinate (A5) at (3,-3);

\coordinate (B1) at (-3,-3);
\coordinate (B2) at (-1,1);
\coordinate (B3) at (1,0);
\coordinate (B4) at (0,-2);
\coordinate (B5) at (2,-4);

\foreach \p in {1,...,5}
{\node at (A\p)  {$\bullet$} ;
\node at (A\p) [below]  {$\p$} ;
\node at (B\p)  {$\times$} ;
\node at (B\p) [above]  {$\p$} ;
}

\draw (A1) .. controls +(left:2) and +(up:2) .. (A2) ;
\draw (A2)  .. controls +(up:3) and +(up:3) .. (A3) node [pos=0.6,name=C1] {} ;
\draw (A3)  .. controls +(up:1) and +(up:1) .. (A4) ;
\draw (A4)  .. controls +(right:0.5) and +(up:1) ..  (A5) ;

\draw (B1) .. controls +(right:1) and +(down:2) .. (B2) ;
\draw (B2)  .. controls +(down:3) and +(down:2) .. (B3) ;
\draw (B3)  .. controls +(down:1) and +(right:1) .. (B4) ;
\draw (B4)   .. controls +(down:0.5) and +(left:0.5) .. (B5) ;

\draw [dotted,thick] (A3)  .. controls +(1,2) and +(2,2) .. (A4) node [midway,name=C2] {} ;
\draw [color=blue,fleche,thick,dashed]  (C1) .. controls +(1,0.5) and +(-1,0.5) .. (C2) node [near end,above] {$h$} ;

\end{scope}

\end{tikzpicture}
\end{center}

\caption{Given two orbits of a Brouwer homeomorphism, Handel's theorem provides a way to connect the points in each orbit by a \emph{homotopy streamline} $\Phi^{-1}(\bbR \times \{1\}), \Phi^{-1}(\bbR \times \{2\})$ which is isotopic to its image (see section~\ref{subsection.basic} below)}
\end{figure}
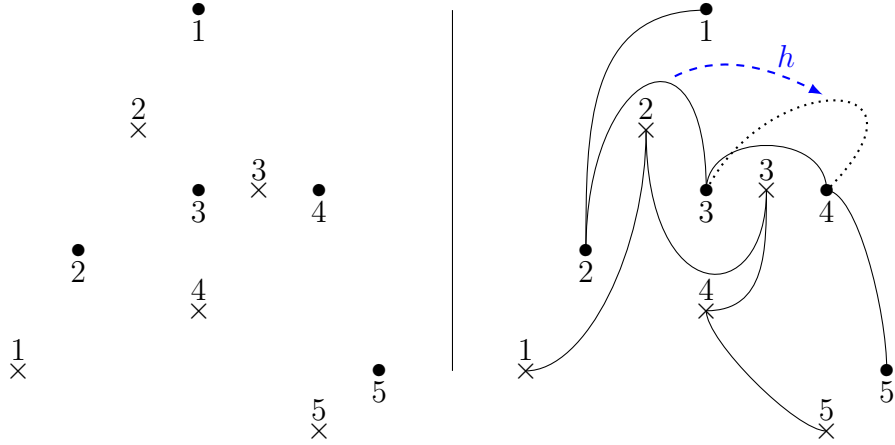

Now let $h$ be a Brouwer homeomorphism, and $\cO_{1}, \cO_{2}$ be two orbits of $h$. Handel's theorem provides an orientation preserving homeomorphism $\Phi$  that sends $\cO_{1}, \cO_{2}$ respectively onto $\bbZ \times \{1\},\bbZ \times \{2\}$, and such that the mapping class of $h' = \Phi h \Phi^{-1}$ relative to these orbits 
is equal either to $[T]$, $[T^{-1}]$, $[R]$ or $[R^{-1}]$. Let  $\gamma$ be some curve starting at a point of $\bbZ \times \{1\}$ and ending at a point of $\bbZ \times \{2\}$.
\begin{theo}\label{theo.1}
The number $I(\Phi h \Phi^{-1},\gamma)$ does not depend on the choice of  the curve $\gamma$ nor of the map $\Phi$.
\end{theo}
Theorem~\ref{theo.1} will be proved in section~\ref{section.index}. We denote this number by  $I(h,\cO_{1}, \cO_{2})$ and call it the \emph{Poincar\'e index of $h$ between the orbits $\cO_{1}$ and $\cO_{2}$}.
As a consequence, this index is a conjugacy invariant: for every orientation preserving homeomorphism $\Psi$,
$$
I(\Psi h\Psi^{-1},\Psi \cO_{1}, \Psi \cO_{2}) = I(h,\cO_{1}, \cO_{2}).
$$
When the homeomorphism $h$ pushes points along a foliation $\cF$, it is not difficult to see that this index coincides with the index of the foliation between the leaves containing the orbits $\cO_{1}$ and $\cO_{2}$. To avoid confusion, we insist that the index is \emph{not} an isotopy invariant. In particular it may take any integer or half-integer value, whereas according to Handel's theorem there are only finitely many isotopy classes. 

\subsection{Quasi-additivity}\label{subsection.additivity}
The index of a non vanishing vector field along a curve is additive, in the sense that the index along the concatenation of two curves is the sum of the indices along each of the curves. We will prove that our index satisfies a slightly weaker property.
\begin{theo}\label{theo.2}
Let $h$ be a Brouwer homeomorphism, and $\cO_{1}, \cO_{2}, \cO_{3}$ be three orbits of $h$. Then the following quasi-additivity relation holds:
$$
\left| I(h,\cO_{1}, \cO_{2}) + I(h,\cO_{2}, \cO_{3}) + I(h,\cO_{3}, \cO_{1}) \right| \leq \frac{1}{2}.
$$ 
\end{theo}
Note that this bound is optimal, in particular additivity does not hold: an example is provided by the sum of the three leftmost indices on Figure~\ref{fig.reeb-components}.
Theorem~\ref{theo.2} will be proved in section~\ref{section.additivity}.

The index for couples of leaves of a foliation satisfies the same property. The proof is much easier, and constitutes an interesting introductory exercise for the proof of Theorem~\ref{theo.2}.

\subsection{Brouwer classes relative to three orbits}\label{subsection.three}
The quasi-additivity property essentially comes from the description of the mapping classes of Brouwer homeomorphisms relative to three orbits, which we explain now.
A \emph{flow} is a continuous family $(h_{t})_{t \in \bbR}$ of homeomorphisms of the plane satisfying the composition law $h_{t+s} = h_{t} \circ h_{s}$ for every $t,s \in \bbR$.
  A homeomorphism $h$ is said to be the \emph{time one map of a flow} if there exists a flow $(h_{t})_{t \in \bbR}$  such that $h_{1}=h$. A mapping class is a \emph{flow class} if it contains the time one map of a flow. If it contains a homeomorphism which is both a time one map and fixed point free, the we will say that it is a \emph{fixed point free flow class}.
  The following is an extension of Handel's theorem.
 \begin{theo}\label{theo.3}
Every Brouwer mapping class relative to one, two or three orbits is a fixed point free flow class. 
\end{theo}
Theorem~\ref{theo.3} will be proved in section~\ref{section.three}. We will provide a more explicit formulation in section~\ref{subsection.explicit}; the formulation in terms of flow classes is just convenient to express Theorem~\ref{theo.3} in a short way.
The number three is optimal: there exist Brouwer homotopy classes relative to four orbits that are not flow classes (see~\cite{handel1999fixed}, example 2.9 for five orbits and~\cite{leroux2012introduction}, exercise 5 for four orbits).

\subsection{Aknowledgements}
This paper relies heavily on the work of Michael Handel, especially~\cite{handel1999fixed}. The author is also grateful to Michael Handel for several illuminating conversations, both electronically and during the conference on Topological Methods in Dynamical Systems held in Campinas, Brasil, in 2011. The author also whishes to thank Lucien Guillou for discussions concerning the classification relative to three orbits, and the anonymous referee for his careful reading and his suggestions that have helped clarifying the paper.

\section{The index}\label{section.index}

\subsection{Proof of the invariance}\label{subsection.proof-invariance}
Theorem~\ref{theo.1} claims that a certain number does not depend on the choice of a curve $\gamma$ nor of a map $\Phi$. For clarity we state the independence on the curve in a separate lemma.
We denote $\bbZ\times \{1\}$ and $\bbZ\times \{2\} $ respectively by $\bbZ_{1}$ and $\bbZ_{2}$.

\begin{lemm}\label{lemma.curve}
Let $h$ be a Brouwer homeomorphism which globally preserves $\bbZ_{1}$ and $\bbZ_{2}$ and is isotopic, relative to these sets, to one of the four maps $T,T^{-1}, R, R^{-1}$. Then the index of $h$ along a curve joining a point of $\bbZ_{1}$ to a point of $\bbZ_{2}$ does not depend on the choice of the curve.
\end{lemm}

Note that Figure~\ref{fig.bad-coordinates} above provides a counter-example if we remove from the hypothesis the condition on the isotopy class of $h$.
Given this lemma, the proof of Theorem~\ref{theo.1} relies on a proposition concerning conjugacy classes and centralizers in the group $MCG(\bbR^{2}; \bbZ_{1}, \bbZ_{2})$.
To state the proposition, we introduce the ``twist maps''  $T_{1}, T_{2}$ defined by $T_{1}(x,y) = (x+2-y,y)$ and $T_{2}(x,y) = (x+y-1,y)$: these maps act like horizontal translations on every horizontal line, $T_{1}$  translates $\bbR \times \{1\}$ from one unit to the right and is the identity on $\bbR \times \{2\}$, and $T_{2}$ does  the converse. Note that the unit translation $T$ is equal to $T_{1} T_{2}$, and the map  $R$ coincides with $T_{1} T_{2}^{-1}$  on the lines $\bbR\times \{1\}$ and  $\bbR\times \{2\}$.

\pagebreak[3]
\begin{prop}~ \label{prop.centralisateur2}
In the group $MCG(\bbR^{2}; \bbZ_{1}, \bbZ_{2})$,
\begin{itemize}
\item the elements $[R]$ and $[R]^{-1}$ are not conjugate, and not conjugate to $[T]$;
\item the elements  $[T]$ and $[T]^{-1}$ are conjugate;
\item the centralizer of $[R]$ is the free abelian group generated by $[T_{1}]$ and $[T_{2}]$;
\item the centralizer of $[T]$ coincides with the projection in $MCG(\bbR^{2}; \bbZ_{1}, \bbZ_{2})$ of the centralizer of $T$ in $\homeo^{+}(\bbR^{2}; \bbZ_{1}, \bbZ_{2})$: in other words, every mapping class that commutes with $[T]$ is the mapping class of a homeomorphism that commutes with $T$.
\end{itemize}
\end{prop}

In the end of this subsection we deduce Theorem~\ref{theo.1} from Lemma~\ref{lemma.curve} and Proposition~\ref{prop.centralisateur2}. In subsection~\ref{subsection.basic} we will introduce tools from homotopy Brouwer theory which will be used in subsections~\ref{subsection.curve} and~\ref{subsection.centralizers} to prove the lemma and the proposition.

The second point of the proposition, namely that the mapping class of $T$ is conjugate to its inverse, will follow from the stronger fact that $T$ is conjugate to its inverse within the group $\homeo^{+}(\bbR^{2}; \bbZ_{1}, \bbZ_{2})$. In the proof of the theorem we will make use of a homeomorphism $\Psi_{0} \in   \homeo^{+}(\bbR^{2}; \bbZ_{1}, \bbZ_{2})$ realizing the conjugacy. Note that a half-turn rotation realizes a conjugacy between $T$ and $T^{-1}$, and there is such a rotation which preserves $\bbZ_{1} \cup \bbZ_{2}$, but this rotation exchanges $\bbZ_{1}$ and $\bbZ_{2}$, thus it does not belong to the group $\homeo^{+}(\bbR^{2}; \bbZ_{1}, \bbZ_{2})$, and the actual construction of $\Psi_{0}$ is less straightforward. Here is one way to do it. Consider the quotient space $\bbR^{2}/T$, which is an infinite cylinder. Take a homeomorphism of this cylinder that is isotopic to the identity and that exchanges the projection of the orbits $\bbZ_{1}$ and $\bbZ_{2}$. Lifting this homeomorphism, we get a homeomorphism $\Psi_{1}$ that commutes with $T$ and exchanges $\bbZ_{1}$ and $\bbZ_{2}$ (more concretely, this map may be obtained as an infinite composition of ``Dehn twists'', pairwise conjugated by powers of $T$). Let $\Psi_{2}$ be the rotation as above, that  exchanges both orbits and satisfy $\Psi_{2} T \Psi_{2}^{-1} = T^{-1}$. The composition $\Psi_{0} = \Psi_{2} \Psi_{1}$ suits our needs, namely we have $\Psi_{0} T \Psi_{0}^{-1} = T^{-1}$, $\Psi_{0} \bbZ_{1} = \bbZ_{1}$ and $\Psi_{0} \bbZ_{2} = \bbZ_{2}$. We will also use the specific form of $\Psi_{0}$ as a composition of a rotation and a map that commutes with $T$.

In the following proof we will also make use of the arcwise connectedness of the space of orientation preserving homeomorphisms that commutes with $T$. This is a consequence of the following facts.  Any element of this space induces an element of the space of homeomorphisms of the infinite cylinder $\bbR^{2}/T$ that preserve the orientation and both ends of the cylinder. This latter space is arcwise connected (see~\cite{hamstrom1966homotopy}), and every isotopy in the infinite cylinder lifts to an isotopy among homeomorphisms of the plane that commutes with $T$.

\begin{proof}[Proof of Theorem~\ref{theo.1}]
Let $h$ be a Brouwer homeomorphism, and $\cO_{1}, \cO_{2}$ be two orbits of $h$. We consider two maps $\Phi_{1}, \Phi_{2}$ as in section~\ref{subsection.index} and Theorem~\ref{theo.1}: 
\begin{itemize}
\item $\Phi_{1},\Phi_{2}$ are orientation preserving homeomorphisms that send $\cO_{1}, \cO_{2}$ respectively to $\bbZ_{1}$, $\bbZ_{2}$, 
\item  the maps $h_{1} = \Phi_{1} h \Phi_{1}^{-1}$ and $h_{2}= \Phi_{2} h \Phi_{2}^{-1}$ are respectively isotopic,   relative to $\bbZ_{1} \cup \bbZ_{2}$, to maps $U_{1}, U_{2}$ which belongs to $\{T,T^{-1}, R, R^{-1} \}$.
\end{itemize}
According to Lemma~\ref{lemma.curve}, the index of a Brouwer homeomorphism $h'$ in $\homeo^{+}(\bbR^{2}; \bbZ_{1},\bbZ_{2})$  along a curve joining  $\bbZ_1$ to $\bbZ_2$ does not depend on the curve, thus we may denote it by $I(h')$.
We want to show that $I(h_{1}) = I(h_{2})$. 
The map $\Psi = \Phi_{2} \Phi_{1}^{-1}$ has the following properties:
\begin{enumerate}
\item $\Psi (\bbZ_1) = \bbZ_1$, $\Psi ( \bbZ_2) = \bbZ_2$,
\item $\Psi U_{1} \Psi^{-1}$ is isotopic  to $U_{2}$ relative to $\bbZ_1 \cup \bbZ_2$,

and

\item  $U_{1}, U_{2}$ belongs to $\{T,T^{-1}, R, R^{-1} \}$.
\end{enumerate}

From now on we work with $\Psi,U_{1},U_{2} \in \homeo^{+}(\bbR^{2}; \bbZ_{1},\bbZ_{2})$ satisfying these three properties, with a Brouwer homeomorphism $h_{1}$ that is isotopic to $U_{1}$ relative to $\bbZ_1 \cup \bbZ_2$, and with $h_{2} = \Psi h_{1} \Psi^{-1}$ (the reader may forget everything about $h, \Phi_{1},\Phi_{2}, \cO_{1}, \cO_{2}$). 
We choose some curve $\gamma$ joining some point in $\bbZ_{1}$ to some point in $\bbZ_{2}$.

We first treat the ``Reeb case'', \emph{i.e.} the case when $U_{1} = R$ or $R^{-1}$. According to the first point of the above Proposition, in this case $U_{1}$ is not conjugate to any of the three other maps in $\{T,T^{-1}, R, R^{-1} \}$. Thus property 2 of the map $\Psi$ implies that $U_{2}=U_{1}$.
Furthermore, the third point of the proposition provides integers $n_{1}, n_{2}$ such that $\Psi$ is isotopic to $\Psi' = T_{1}^{n_{1}} T_{2}^{n_{2}}$ relative to $\bbZ_1 \cup \bbZ_2$. 
We remark that if $(h_{t})$ is an isotopy in $\homeo^{+}(\bbR^{2}; \bbZ_1,\bbZ_2)$, and if every $h_{t}$ is a fixed point free homeomorphism, then the index
$I(h_{t}) = I(h_{t},\gamma)$ is defined for every $t$, and it is an integer or a half integer that depends continuously on $t$, thus it is constant. Using this remark, we first see that
$$
  I(h_{2}) = I(\Psi h_{1} \Psi^{-1}) = I( \Psi' h_{1} {\Psi'}^{-1}).
$$
We conclude this case by the equality $I( \Psi' h_{1} {\Psi'}^{-1}) = I(h_{1})$, which follows from the next lemma since $\Psi'$ commutes with $T$.
\begin{lemm}\label{lemm.index}	
Let $h,\Phi \in \homeo^{+}(\bbR^{2}; \bbZ_1,\bbZ_2)$ and assume that $h$ is fixed point free and $\Phi$ commutes with $T$. Then $I(\Phi h \Phi^{-1}) = I(h)$.
\end{lemm}
\begin{proof}[Proof of Lemma~\ref{lemm.index}]
Since the space of homeomorphisms that commute with $T$ is arcwise connected, we may choose an isotopy $(\Phi_{t})$ from the identity to $\Phi$ every time of which commutes with $T$. Note that for every point $x \in \bbZ_1 \cup \bbZ_2$ we have $h (x) = T^n (x)$ for some $n$, and thus $(\Phi_{t} h \Phi_{t}^{-1}) (\Phi_{t} x) = T^n (\Phi_{t}x)$. We deduce that for each fixed $t$ the quantity
$$
I(\Phi_{t} h \Phi_{t}^{-1},\Phi_{t} \gamma)
$$
is an integer or a half integer. By continuity this quantity does not depend on $t$, and we get the lemma.
\end{proof}

Now let us turn to the ``translation case'', when $U_{1} = T^{\pm1}$. In this case, the Proposition says that $U_{2} = T^{\pm 1}$.
Let us first treat the sub-case when $U_{2}=U_{1}$. We have $[\Psi T \Psi^{-1}] = [T]$, and the fourth point of the proposition says that the map $\Psi$ is isotopic relative to $\bbZ_1 \cup \bbZ_2$ to a homeomorphism $\Psi'$ that commutes with $T$.
As in the Reeb case, we first have the equality $I(h_{2}) = I(\Psi h_{1} \Psi^{-1}) = I(\Psi' h_{1} \Psi'^{-1})$, and then the equality $I(\Psi' h_{1} \Psi'^{-1}) = I(h_{1})$ follows from  Lemma~\ref{lemm.index}.

It remains to consider the sub-case when $U_{2} = U_{1}^{-1}$, that is, when $[\Psi T \Psi^{-1}] = [T^{-1}]$.
Here we will use  the homeomorphism $\Psi_{0}$ whose construction was given before the proof. The mapping class  $[\Psi \Psi_{0}^{-1}]$ commutes with $[T]$. According to the proposition, $\Psi \Psi_{0}^{-1}$ is isotopic relative to $\bbZ_1 \cup \bbZ_2$ to a homeomorphism that commutes with $T$, and thus (composing with $\Psi_{0}$) we see that $\Psi$ is isotopic to a homeomorphism $\Psi'$ satisfying $\Psi' T \Psi'^{-1} = T^{-1}$.
Then we have 
$$
\begin{array}{rcl}
I(h_{2}) & =&  I(\Psi h_{1} {\Psi}^{-1}) \\
& =&  I(\Psi' h_{1} {\Psi'}^{-1}) \\
& =& I(\left(\Psi' \Psi_{0}^{-1} \right)  \Psi_{0} h_{1} \Psi_{0}^{-1}  \left(\Psi' \Psi_{0}^{-1} \right) ^{-1})  \\ 
&=&  I(\Psi_{0} h_{1} \Psi_{0}^{-1}) 
\end{array}
$$
where the last equality follows from Lemma~\ref{lemm.index}.
It remains to prove the equality $I(\Psi_{0} h_{1} \Psi_{0}^{-1}) = I(h_{1})$.  Remember that $\Psi_{0}$ was defined as the composition $\Psi_{0}=\Psi_{2}\Psi_{1}$, where $\Psi_{1}$ commutes with $T$ and $\Psi_{2}$ is a half-turn rotation. Consider an isotopy from the identity to $\Psi_{2}$ among rotations, compose this isotopy on the right with $\Psi_{1}$ to get an isotopy from $\Psi_{1}$ to $\Psi_{2} \Psi_{1} = \Psi_{0}$, and concatenate this first isotopy with a second isotopy from the identity to $\Psi_{1}$ among homeomorphisms that commute with $T$. The resulting isotopy goes from the identity to $\Psi_{0}$, denote it by $(\Psi'_{t})$. Then for every fixed $t$ the map 
$\Psi'_{t} h_{1} {\Psi'_{t}}^{-1}$ acts like a translation (by a non horizontal vector) on $\Psi'_{t} (\bbZ_1 \cup \bbZ_2)$, so that the index of this map along the curve $\Psi'_{t}(\gamma)$ is an integer. We conclude that $I(\Psi_{0} h_{1} \Psi_{0}^{-1}, \Psi_{0}(\gamma)) = I(h_{1}, \gamma)$, which completes the proof of the claim.
\end{proof}

\subsection{Basic tools for homotopy Brouwer theory}\label{subsection.basic}
We review some basic objects of homotopy Brouwer theory, namely \emph{homotopy translation arcs} and \emph{homotopy streamlines}. Everything here comes from~\cite{handel1999fixed}.

Let us fix an orientation preserving homeomorphism $h$, and points $x_{1}, \dots x_{r}$ whose orbits are properly embedded, \emph{i.e.} such that
$$
\lim_{n\to \pm \infty} h^n x_{i} = \infty.
$$
We denote by $\cO_{1}, \dots,  \cO_{r}$ the orbits  of $x_{1}, \dots , x_{r}$, and by $\cO$ their union. We consider the set $\cA_{0}$ of continuous injective curves $\alpha : [0,1] \to \bbR^2$  which are disjoint from $\cO$ except at their end-points which are supposed to be points of $\cO$. Two elements $\alpha,\beta$ of $\cA_{0}$ are said to be \emph{isotopic relative to $\cO$} if there exists an isotopy  $(h_{t})_{t \in [0,1]}$  in the group $\homeo^{+}_{0}(\bbR^{2}; \cO_{1}, \dots , \cO_{r})$, such that $h_{0} = \mathrm{Id}$ and $h_{1}(\alpha) = \beta$.
Two elements $\alpha,\beta$ are \emph{homotopically disjoint} if $\alpha$ is isotopic relative to $\cO$ to an element $\alpha'$ such that $\alpha' \cap \beta \subset \cO$. This is a symmetric relation.
A sequence of curves $(\alpha_{n})_{n \geq 0}$ in $\cA_{0}$ is said to be \emph{homotopically proper} if 
 for every compact subset $K$ of the plane, there exists $n_{0}$ such that for every $n \geq n_{0}$, there exists $\alpha' \in \cA_{0}$ which is isotopic to $\alpha_{n}$ relative to $\cO$ and disjoint from $K$.

Note that the map $h$ acts naturally on $\cA_{0}$. An element $\alpha$ of $\cA_{0}$ is a \emph{homotopy translation arc} for $[h;\cO_{1}, \dots , \cO_{r}]$ if $h (\alpha(0)) = \alpha(1)$ and $\alpha$ is homotopically disjoint from all its iterates under $h$. An element $\alpha$ is \emph{forward proper} if the sequence $(h^n (\alpha))_{n \geq 0}$ is homotopically proper.
Backward properness is defined analogously. We insist that these notions (homotopy translation arcs, properness) depend only on the mapping class of $h$ relative to $\cO$, rather than on $h$ itself.

For example, take $\cO = \bbZ_1 \cup \bbZ_2$. Then the segment $[0,1] \times \{1\}$ is a homotopy translation arc for the Reeb homeomorphism $R$ which is both backward and forward proper. It is not difficult to see that every homotopy translation arc for $R$, relative to 
$\bbZ_1 \cup \bbZ_2$, and joining the points $(0,1)$ and $(1,1)$, is homotopic to this segment (see~\cite{franks2003periodic}, Lemma 8.7 (2)). The same segment is also a homotopy translation arc for the translation $T$, but in this latter case there are infinitely many distinct homotopy classes of translation arcs having the same end-points (see Figure~\ref{fig.hta-for-T}).
These examples will be useful to prove points 1 and 2 of Proposition~\ref{prop.centralisateur2}.

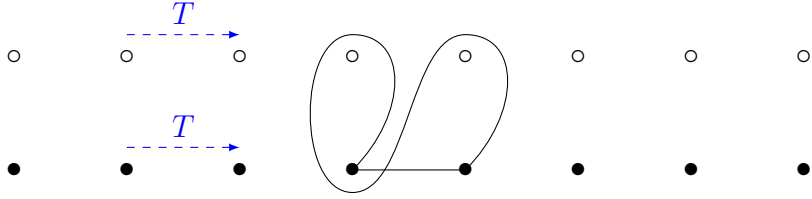
\begin{figure}[ht]
\begin{center}
\begin{tikzpicture}[scale=1.5]
\tikzstyle{fleche}=[>=latex,->]
\foreach \k in {-3,...,4} 
	{\draw (\k,1) node {$\bullet$} ; \draw (\k,2) node {$\circ$} ; }
\def\depl{0.2}
\draw [fleche,color=blue,dashed] (-2,2+\depl) -- (-1, 2+\depl) node [midway,above] {$T$} ;
\draw [fleche,color=blue,dashed] (-2,1+\depl) -- (-1, 1+\depl) node [midway,above] {$T$} ;

\draw (0,1) --(1,1) ;

\draw (0,1) .. controls +(0.5,0.5) and +(0.5,0) ..
	 (0,2.2) .. controls +(-0.5,0) and +(-0.5,0) ..
	  (0,0.8) .. controls +(0.5,0) and +(-0.5,0) .. 
	  (1,2.2) .. controls +(0.5,0) and +(0.5,0.5) ..  (1,1) ;

\end{tikzpicture}
\end{center}
\caption{Two homotopy translation arcs for $[T]$ relative to two orbits}
\label{fig.hta-for-T}
\end{figure}

Still following Handel, we endow the complementary set of $\cO$ in the plane with a (complete) hyperbolic metric: in other words, we consider a discrete subgroup of isometries $G$ of the hyperbolic plane $\bbH^2$ such that the quotient $\bbH^2/G$ is homeomorphic to $\bbR^2 \setminus \cO$. Furthermore, $G$ may be chosen to be of the first kind (see for example~\cite{katok1992fuchsian} for the definitions). The main purpose of this geometrisation of $\bbR^2 \setminus \cO$ is to provide each isotopy class $[\alpha]$ in $\cA_{0}$ with a unique hyperbolic geodesic $\alpha^\sharp$. The family of geodesics have wonderful properties. In particular, they have minimal intersection so, for instance, two elements of $\cA_{0}$ are homotopically disjoint if and only if the associated hyperbolic geodesics have intersection included in $\cO$.
The other essential properties are summed up in the following lemma (see~\cite{handel1999fixed}, lemma 3.5; \cite{matsumoto2000arnold}, beginning of section~2; \cite{leroux2012introduction}, appendix~1).

\begin{lemm}~\label{lemm.straightening}

\begin{enumerate}
\item Let $\{\alpha_{i}\}$ be a locally finite family of curves in $\cA_{0}$  which are pairwise non isotopic relative to $\cO$, and have pairwise intersections included in $\cO$. Then there exists $f \in \mathrm{Homeo}^+_{0}(\bbR^2; \cO_{1}, \dots, \cO_{r})$ such that for every $i$, $f(\alpha_{i}) = \alpha_{i}^\sharp$.

\item More generally, let $\{\alpha_{i}\}$ be as in the previous item, and 
let $\{\beta_{j}\}$ having the same properties.  Assume that for every $i,j$, the curve $\alpha_{i}$ is  not isotopic to $\beta_{j}$ relative to $\cO$ and in minimal position with $\beta_{j}$. Then there exists $f \in \mathrm{Homeo}^+_{0}(\bbR^2; \cO_{1}, \dots, \cO_{r})$ such that for every $i,j$, $f(\alpha_{i}) = \alpha_{i}^\sharp$ and  $f(\beta_{j}) = \beta_{j}^\sharp$.

\item Let  $(\alpha_{n})_{n \geq 0}$ be a sequence  in $\cA_{0}$.
 If  $(\alpha_{n})_{n \geq 0}$ is homotopically proper then  $(\alpha_{n}^\sharp)_{n \geq0}$ is proper.
\end{enumerate}
\end{lemm}

 Since pairs of geodesics are in minimal position, if $\alpha$ is a homotopy translation arc then $(h^n \alpha)^\sharp \cap \alpha^\sharp \subset \cO$ for every $n \neq 0$. Thus the concatenation
$$
\alpha^\sharp \star \cdots \star (h^n \alpha)^\sharp \star \cdots
$$
is the image of $[0,+\infty)$ under an injective continuous map $A^{+}$. Furthermore, point 3 of the lemma entails that the homotopy translation arc $\alpha$ is forward proper if and only if 
the map $A^{+}$ is proper\footnote{Remember that a map is \emph{proper} if the inverse image of every compact subset of the target is compact; in our context, the map $A^{+}$ is proper if and only if $\lim_{t \to +\infty} A^+(t) = \infty$.}. 
The embedding $A^+$ will be called the \emph{proper forward (geodesic) homotopy streamline} generated by $\alpha$.
\emph{Proper backward (geodesic) homotopy streamlines} are defined analogously. The union of a proper backward geodesic homotopy streamline and a proper forward geodesic homotopy streamline generated by the same homotopy translation arc will be called a \emph{proper (geodesic) homotopy streamline}; it is a \emph{topological line}, that is, the image of a proper injective continuous map from the real line to the plane. The word ``geodesic'' will generally be omitted. 

Assume the homotopy translation arc  $\alpha$ is both forward and backward proper.
Then the set $\{ (h^{n} \alpha)^{\sharp} , i=1, \dots, r, n \in \bbZ \}$
is a locally finite family of geodesics having pairwise intersections included in $\cO$ (and whose union is the proper homotopy streamline $A$). The image of this family under $h$ is again a locally finite family of elements of $\cA_{0}$ having pairwise intersections included in $\cO$. 
Note furthermore that for each $n$ and $i$, $h((h^{n} \alpha)^{\sharp})$ is isotopic to the geodesic $(h^{n+1} \alpha)^{\sharp}$ relative to $\cO$.
The first point of the above lemma applies: there exists an element $f$ of the group $\homeo^{+}_{0}(\bbR^{2}; \cO_{1}, \dots , \cO_{r})$ which sends 
$h((h^{n} \alpha)^\sharp)$ to  $(h^{n+1} \alpha)^{\sharp}$ for every $n$. Then $h':=fh$ belongs to the mapping class of $h$ and satisfies ${h'}^n(\alpha^\sharp) = (h^{n} \alpha)^{\sharp}$ for every integer $n$: in other words, the proper homotopy streamline $A$ is the concatenation of all the iterates of $\alpha^\sharp$ under $h'$ (the arc $\alpha^\sharp$ is called a \emph{translation arc} for $h'$ in classical Brouwer theory, see for example~\cite{guillou1994brouwer}). The construction is easily generalized to give the following corollary.
\begin{coro}\label{coro.straightening}
If $A_{1}, \dots, A_{r}$ is a family of pairwise disjoint proper homotopy streamlines for $h$, then there exists a homeomorphism $h'$ isotopic to $h$ relative to $O$ such that $h'(A_{i}) = A_{i}$ for every $i= 1, \dots, r$.
\end{coro}
(This kind of argument is very common in homotopy Brouwer theory; more generally, we will call ``straightening principle'' every use of the first or the second point of the above lemma).

The beginning of section 3 in~\cite{handel1999fixed} provides a construction of a hyperbolic structure $\cH$ in $\bbR^{2} \setminus \bbZ$ for which the translation $T$ is an isometry. Now let $\cO_{1}, \dots, \cO_{r}$ be distinct orbits of $T$. By considering the projections of these orbits in the annulus $\bbR^{2}/T$ on the one hand, and the projection of $\bbZ$ in the annulus $\bbR^{2}/T^{r}$ on the other hand, one can find a map $\Psi$ such that $\Psi T \Psi^{-1} = T^{r}$ and $\Psi(\cO_{1} \cup \cdots \cup \cO_{r}) = \bbZ$. Then the inverse image of $\cH$ under $\Psi$ is a hyperbolic structure on the complement of $\cO_{1} \cup \cdots \cup \cO_{r}$ for which $T$ is an isometry.
Here is one application of this construction. Let $\alpha$ be a homotopy translation arc for the translation $T$ relative to the orbits $\cO_{1}, \dots, \cO_{r}$. Since $T$ is an isometry, the image of a geodesic is a geodesic, and in particular we get the relations $(T^{n} \alpha)^{\sharp} = T^{n} (\alpha^{\sharp})$ for every $n$.
Thus the concatenation of all the $T^{n} (\alpha^{\sharp})$'s is a proper geodesic homotopy streamline which is invariant under $T$. In particular we see that every homotopy translation arc for $T$ is homotopic to a translation arc for $T$. We will take advantage of this remark in the proof of Proposition~\ref{prop.centralisateur2}.

We end this subsection by a lemma which gives the essential property of flow classes, and which will be useful in section~\ref{section.additivity}. 
\begin{lemm}\label{lem.flow-streamlines}
The Brouwer mapping class $[h;\cO_{1}, \dots , \cO_{r}]$ is a fixed point free flow class if and only if it admits a family of pairwise disjoint proper geodesic homotopy streamlines whose union contains all the $\cO_{i}$'s.
\end{lemm}
\begin{proof}
Let us assume that $[h;\cO_{1}, \dots , \cO_{r}]$ is a Brouwer mapping class which admits a family of pairwise disjoint proper geodesic homotopy streamlines $A_{1}, \dots, A_{r}$ with  $\cO_{i}$ included in $A_{i}$ for every $i$. 
The above corollary provides an element $h'$ in the mapping class of $h$ such that $h'(A_{i}) = A_{i}$ for every $i$.
Since $h'$ coincides with $h$ on $\cO_{i}$, it has no fixed point on $A_{i}$.
The fact that $h$ is a flow class now comes from Lemmas~\ref{lemma.flow} and~\ref{lemma.Alexander} below. The first Lemma provides a (fixed point free) time one map of a flow  $h''$ that coincides with $h'$ on each $A_{i}$. The second Lemma implies that $h''$ belongs to the mapping class of $h'$, and thus also to the mapping class of $h$.

Now let us assume  that the Brouwer mapping class $[h;\cO_{1}, \dots , \cO_{r}]$ is a flow class: there exists a flow $(h_{t})_{t \in \bbR}$  such that $h_{1}$ is isotopic to $h$ relative to the union of the $\cO_{i}$'s. Consider the trajectory under the flow of a point $x \in \cO_{i}$ for some $i$: it is the image of the real line under the map $t \mapsto h_{t}(x)$. Since $h_{1}$ is isotopic to $h$ relative to the union of the $\cO_{i}$'s, we have $h_{n}(x) = h^n(x)$ for every integer $n$, and since the sequence $(h^n(x))_{n \in \bbZ}$ tends to infinity when $n$ tends to $\pm \infty$, this implies that the trajectory of $x$ under the flow is a topological line.

First assume that the $\cO_{i}$'s belong to distinct trajectories of the flow.
Choose for each $i$ a point $x_{i}$ on $\cO_{i}$, and define a curve $\alpha_{i}$ by letting $\alpha_{i}(t) = h_{t}(x_{i})$ for $t \in [0,1]$.
Then the $\alpha_{i}$'s are proper homotopy translation arcs for  $[h;\cO_{1}, \dots , \cO_{r}]$, and they generate pairwise disjoint proper homotopy streamlines.

\begin{figure}[h]
\begin{center}
\begin{tikzpicture}[scale=0.7] \footnotesize
\def\hauteur{7.5}
\def\bas{0.5}
\begin{scope}[xshift=-5cm,xscale=2]
\clip (-2,\bas) rectangle (2,\hauteur) ;

\draw  (0,0) -- (0,10) ;
\def \control {20} 
\foreach \translation in {0, -2, -4 , -6}
 	{\draw  (0.1,\translation) .. controls +(up:\control) and +(up:\control) .. (4.8,\translation) ;}
\foreach \translation in {0, -2, -4 , -6}
 	{\draw  (-0.1,\translation) .. controls +(up:\control) and +(up:\control) .. (-4.8,\translation) ;}
\foreach \p in {1,4,7}
	{\node at (0,\p) {$\bullet$} ; }
\foreach \p in {2,5}
	{\node at (0,\p) {$\circ$} ; }
\foreach \p in {3,6}
	{\node at (0,\p) {$\times$} ; }
\end{scope}

\begin{scope}[xshift=5cm,xscale=2]
\clip (-2,\bas) rectangle (2,\hauteur) ;
\draw[very thin,dashed]  (0,0) -- (0,10) ;
\def \control {20} 
\def\courbure{0.3}
\def\petitecourbure{0.1}
\foreach \translation in {0, -2, -4 , -6}
 	{\draw[very thin,dashed]  (0.1,\translation) .. controls +(up:\control) and +(up:\control) .. (4.8,\translation) ;}
\foreach \translation in {0, -2, -4 , -6}
 	{\draw[very thin,dashed]  (-0.1,\translation) .. controls +(up:\control) and +(up:\control) .. (-4.8,\translation) ;}
\foreach \p in {-2,1,4,7}
	{\node at (0,\p) {$\bullet$} ; 
	\draw[thick] (0,\p) .. controls +(-\courbure,1) and +(-\courbure,-1) .. (0,\p+3);
	}
\foreach \p in {-1,2,5,8}
	{\node at (0,\p) {$\circ$} ; 
	\draw[thick] (0,\p) .. controls +(\courbure,1) and +(\courbure,-1) .. (0,\p+3);
	}
\foreach \p in {0,3,6,9}
	{\node at (0,\p) {$\times$} ; 
		\draw[thick] (0,\p) .. controls +(\petitecourbure,0.5) and +(\petitecourbure,-0.5) .. (0,\p+1.5);
		\draw[thick] (0,\p+1.5) .. controls +(-\petitecourbure,0.5) and +(-\petitecourbure,-0.5) .. (0,\p+3);
	}
\end{scope}

\end{tikzpicture}
\end{center}
\caption{\label{figure.modifier-flot} Constructing proper homotopy streamlines in the bad case}
\end{figure}
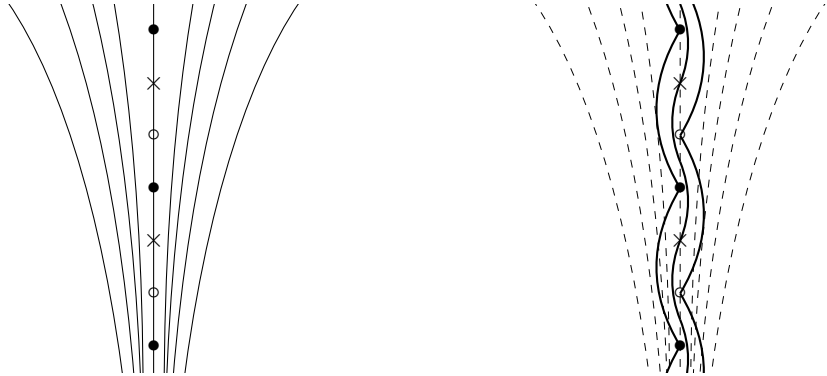

In the bad case when several $\cO_{i}$'s belong to the same trajectory $\Gamma$ of the conjugated flow, one can proceed as follows. 
The Schoenflies theorem provides a homeomorphism of the plane that sends $\Gamma$ to a vertical straight line.
Thus, up to a change of coordinates, we may assume that $\Gamma$ is a vertical straight line. Then Figure~\ref{figure.modifier-flot} indicates how to construct a homotopy translation arc for each of the $\cO_{i}$'s contained $\Gamma$, in such a way that they generate pairwise disjoint homotopy streamlines. Up to isotopy relative to the union of the $\cO_{i}$'s, the whole construction takes place in an arbitrarily small neighborhood of $\Gamma$. This entails that our homotopy translation arcs are backward and forward proper. Thus the homotopy streamlines are proper. Furthermore, the homotopy translation arcs associated to $\cO_{i}$'s included in distinct trajectories of the flow will still be pairwise disjoint.
\end{proof}

\begin{lemm}\label{lemma.flow}
Let $\cF$ be a finite family of pairwise disjoint topological lines in the plane. Let $F$ be the union of the elements of $\cF$. Let $h':F \to F$ be a fixed point free homeomorphism that preserves each $F$ in $\cF$.
Then there exists a fixed point free homeomorphism $h''$ of the plane which extends $h'$ and which is the time one map of a flow.
\end{lemm}

The proof will use the following construction. Under the hypotheses of the lemma, let $U$ be a connected component of the complementary set of $F$ in the plane. The end compactification of the closure of $U$ is homeomorphic to the unit disk (see the corollary in appendix~\ref{sec.schoenflies-homma}): there exist a finite set $E$ in the boundary of the closed unit disk $\bbD^2$, and a homeomorphism $\Phi$ between  $\bbD^2 \setminus E$ and the closure of $U$ (each interval of the complement of $E$ in the boundary of $\bbD^2$ is sent by $\Phi$ onto an element of $\cF$ that is part of the boundary of $U$).

We will also need the following ``model flows'' on the closed unit disk $\bbD^2$. Let $E_{1}, E_{2}, E_{3}$ be subsets of the boundary $\partial \bbD^2$ containing respectively one, two and three points.
\begin{itemize}
\item There exists a fixed point free flow on $\bbD^2 \setminus E_{1}$.
\item There exist two fixed point free flows on $\bbD^2 \setminus E_{2}$, such that for the first one the points on the two components of $\partial \bbD^2 \setminus E_{2}$ flow in the same direction, whereas for the second one they flow in opposite directions.
\item There exist two fixed point free flows on $\bbD^2 \setminus E_{3}$, such that for the first one the points on the three components of $\partial \bbD^2 \setminus E_{2}$ flow in the same direction, whereas for the second one the directions are the same on two components and opposite in the last component.
\end{itemize}
The  flow on $\bbD^2 \setminus E_{1}$ is conjugate to a translation flow on a closed half-plane. The flows on $\bbD^2 \setminus E_{2}$ are conjugate respectively to the translation flow and to the Reeb flow on a closed strip. The details of the constructions are left to the reader.


\begin{proof}[Proof of Lemma~\ref{lemma.flow}]
Let $U$ be any connected component of the complement of $F$. Consider the set $E \subset \partial \bbD^2$ as above. If $E$ contains more than three points, we choose some point $x_{0}$ in $E$  and join $x_{0}$ to every other point of $E$ by a chord of the disk. The images under the homeomorphism $\Phi$ of these chords are topological lines which are disjoint from the elements of the family $\cF$. We add these lines to  $\cF$ and extend $h'$ on each of the new lines in an arbitrary fixed point free way. We make this construction for every connected component $U$ of $\bbR^2 \setminus F$, and still denote by $\cF$ the extended family and by $h'$ the extended homeomorphism. Note that every connected component $U$ of the complement of (the new) $F$ have at most three boundary components. 

Since every fixed point free homeomorphism of the real line is conjugate to a translation, we can find a flow $(h''_{t})$, defined on $F$ and preserving each component of $F$, such that $h''_{1} = h'$. It remains to construct an extension of this flow to each connected component $U$ of the complement of $F$. We note that the extension can be made independently on each such $U$.

Now let $\Phi$ be a homeomorphism between the closure of $U$ and $\bbD^2 \setminus E$ provided by the end-compactification.
We note that the cardinal of $E$ coincides with the number of boundary components of $U$, which equals one, two or three.
Up to conjugating by a homeomorphism of the disk that permutes the points of $E$ (and which may reverse the orientation), we can assume that on each component of $\partial \bbD^2 \setminus E$ the flow $(\Phi h''_{t \mid \partial U} \Phi^{-1})$ flows in the same direction as one of the above ``model flows'' $(m_{t})$. Since all the fixed point free flows of an open interval are conjugate, and each homeomorphism of $\partial \bbD^2$ extends to the disk, a further conjugacy ensures that $(\Phi h''_{t \mid \partial U} \Phi^{-1}) = (m_{t})$ on $\partial \bbD^2$. Now we may extend the flow $(h''_{t})$ on $U$ by the formula $(\Phi^{-1} m_{t} \Phi)$.
%
%
\end{proof}

The next Lemma is a variation on Alexander's trick (\cite{alexander1923deformation}).

\begin{lemm}\label{lemma.Alexander} 
Let $\cF$ be a finite family of pairwise disjoint topological lines in the plane.
Let $h_{0}$ be an orientation preserving homeomorphism of the plane. 
Let $H_{h_{0}}$ be the space of orientation preserving homeomorphisms that coincide with $h_{0}$ on the union of the elements of $\cF$. Then $H_{h_{0}}$ is arcwise connected.
\end{lemm}
\begin{proof}
The map $h \mapsto h_{0}^{-1}h$ is a homeomorphism between $H_{h_{0}}$ and $H_{\mathrm{Id}}$, thus it suffices to prove the Lemma when $h_{0}$ is the identity.
Denote by $F$ the union of the topological lines that belong to $\cF$, and let $U$ be a connected component of the complementary set of $F$ in the plane. Let $\Phi$ be a homeomorphism between  $\bbD^2 \setminus E$ and the closure of $U$, where $E$ is a finite subset of $\partial \bbD^2$ as above.
If $h$ is an element of $H_{\mathrm{Id}}$ then $\Phi^{-1} h \Phi$ extends to a homeomorphism $\bar{h}$ of the disk which is the identity on the boundary (in other words, we are using the naturality of Freudenthal's end compactification of the closure of $U$). The map $h \mapsto \bar{h}$ is a homeomorphism between the space of homeomorphisms of the closure of $U$  that are the identity on the boundary of $U$ and the space of homeomorphisms of the disk that are the identity on the boundary of the disk. The latter is arcwise connected (and even contractible, see~\cite{alexander1923deformation}).
Thus we may find an isotopy in $H_{\mathrm{Id}}$ from $h$ to a homeomorphism that is the identity on $U$. We proceed independently on each connected component of the complement of $F$ to get an isotopy from $h$ to the identity.
 (Note that the proof actually shows the the space $H_{h_{0}}$ is contractible).
\end{proof}
We note that both lemmas still hold, with the same proofs, when the family $\cF$ is only supposed to be locally finite. We will not make use of this remark.

\subsection{Change of curve}
\label{subsection.curve}
In this subsection we prove Lemma~\ref{lemma.curve}.

\begin{proof}
We assume the hypotheses of the Lemma.
The space of continuous curves in the plane joining two given points is easily seen to be connected. Since the index of $h$ along a curve joining a point of $\bbZ_1$
to a point of $\bbZ_2$ is half an integer, the usual continuity argument yields that two curves having the same end-points give the same index. 

It remains to see that different end-points still give the same index. Let $\gamma_{0}$ be any curve from $(0,1)$ to $(0,2)$. Let $n_{0}, n_{1} \in \bbZ$, and choose some curve $\alpha$ from $(n_{0},1)$ to $\gamma_{0}(0)$ and some curve $\beta$ from $\gamma_{0}(1)$  to $(n_{1},2)$. We are left to check that the index along the concatenation $\alpha \star \gamma_{0} \star \beta$ is equal to the index along $\gamma_{0}$. By additivity of the index, it is enough to prove that the indices along $\alpha$ and $\beta$ vanish. Let us take care of $\alpha$ (the case of $\beta$ is obviously symmetric). Since $h$ is isotopic relative to $\bbZ_{1}\cup \bbZ_{2}$ to $T,T^{-1}, R$ or $R^{-1}$, it is also isotopic to one of these maps relative to $\bbZ_{1}$.
Furthermore $R=T$ on the line $\bbR\times \{1\}$, thus $R$ and $T$ are isotopic relative to this line (Alexander's trick, see Lemma~\ref{lemma.Alexander}). Finally $h$ is isotopic relative to $\bbZ_{1}$ to $T$ or $T^{-1}$. Thus the argument reduces to the following Lemma. 

\begin{lemm}\label{lemma.index-same-orbit}
Let $h \in \homeo^{+}(\bbR^{2}; \bbZ_1)$ be a Brouwer homeomorphism whose mapping class relative to $\bbZ_1$ is equal to $[T]^{\pm 1}$.
Then the index  of $h$ along any curve joining two points of the orbit $\bbZ_1$ vanishes. 
\end{lemm}

\end{proof}

\begin{proof}[Proof of Lemma~\ref{lemma.index-same-orbit}]
We again refer to the additivity to reduce the Lemma to the case of a curve joining some point $(n,1)$ to the point $(n+1,1)$. The problem is invariant under translation, thus everything boils down to a curve joining $(0,1)$ to $(1,1)$. Note that it is enough to do this for some specific curve, since again by connectedness any two curves joining these two points give the same index.

According to classical Brouwer theory, we can find a \emph{translation arc}, that is, an injective curve $\gamma$ joining $(0,1)$ to $(1,1)$ and satisfying $\gamma \cap h(\gamma) = \{(1,1)\}$ (see~\cite{guillou1994brouwer}, Lemme 3.2 or~\cite{handel1999fixed}, Lemma 4.1). Furthermore, every translation arc is a homotopy translation arc (Th\'eor\`eme 3.3 in the former).
According to~\cite{handel1999fixed}, Corollary~6.3, there is a unique isotopy class, relative to $\bbZ_{1}$,  of homotopy translation arcs from $(0,1)$ to $(1,1)$
(this is also a special case of Lemma~\ref{lemma.uniqueness-hta} below). Since $h$ is isotopic relative to $\bbZ_{1}$ to $T$ or $T^{-1}$ the segment  $[0,1] \times \{1\}$ is a homotopy translation arc, and thus $\gamma$ is isotopic to this segment relative to $\bbZ_{1}$. The same argument shows that $h(\gamma)$ is isotopic to the segment $[1,2] \times \{1\}$.
We now apply the``straightening principle'' (point 1 of Lemma~\ref{lemm.straightening}) to get an isotopy $(\Psi_{t})_{t \in [0,1]}$ of elements in $\homeo^{+}_{0}(\bbR^{2}; \bbZ_{1})$ such that $\Psi_{0}$ is the identity while $\Psi_{1}$ satisfies $\Psi_{1}(\gamma) =  [0,1] \times \{1\}$ and $\Psi_{1}(h(\gamma)) =  [1,2] \times \{1\}$. By continuity we get $I(h,\gamma) = I(\Psi_{1} h \Psi_{1}^{-1}, [0,1] \times \{1\})$. This last number is  equal to zero, since $\Psi_{1} h \Psi_{1}^{-1}$ sends the segment $[0,1] \times \{1\}$ to the segment $[1,2] \times \{1\}$.
This completes the proof of the Lemma.
\end{proof}

\subsection{Centralizers of $[T]$ and $[R]$}
\label{subsection.centralizers}

\begin{proof}[Proof of proposition~\ref{prop.centralisateur2}]

The translation $T$ admits infinitely many isotopy classes of homotopy translation arcs sharing the same endpoints, while $R$ admits only one (see section~\ref{subsection.basic}). This property distinguishes the conjugacy classes of $[T]$ and $[R]$ in the mapping class group. 

To see that $[R]$ and $[R]^{-1}$ are not conjugate we will argue by contradiction. 
The proof is a ``homotopic version'' of the easier fact that there is no conjugacy between $R$ and $R^{-1}$ by a homeomorphism $\Psi$ that preserves orientation and globally fixes each line $\bbR \times \{1\}$ and $\bbR \times \{2\}$, which may be proved by contradiction as follows: such a homeomorphism would reverse the orientation on the line $\bbR \times \{1\}$, and since it preserves the orientation of the plane it would have to exchange the two half-planes delimited by this line, in contradiction with the preservation of $\bbR \times \{2\}$.

 Let $\alpha_{1} = [0,1] \times \{1\}, \alpha_{2} = [0,1] \times \{2\}$. These segments are homotopy translation arcs for $R$. Let $A_{1} = \bbR \times \{1\}, A_{2}= \bbR \times \{2\}$ be the proper homotopy streamlines generated by $\alpha_{1}, \alpha_{2}$. These  lines are oriented by the action of $R$, and each one is on  the left hand side of the other one with respect to this orientation. 
Assume by contradiction that the relation $[\Psi] [R] [\Psi]^{-1} = [R]^{-1}$ holds in the mapping class group, with $\Psi \in \homeo^{+}(\bbR^{2}; \bbZ_1,\bbZ_{2})$. Then $\Psi(\alpha_{1})$ is a homotopy translation arc for $R^{-1}$. Since $\Psi$ globally preserves $\bbZ_{1}$, the end-points of this arc are on $\bbZ_{1}$. Remember that there is a unique homotopy class of homotopy translation arcs for $R$ or $R^{-1}$, relative to $\bbZ_{1} \cup \bbZ_{2}$, joining a given point of $\bbZ_{1}$ to its image (\cite{franks2003periodic}, Lemma 8.7 (2)).
Thus there exists some $n_{0}$ such that $\Psi(\alpha_{1})$ is isotopic to $R^{n_{0}} (\alpha_{1}) = [n_{0},n_{0}+1]\times\{1\}$ with the opposite orientation. Up to composing $\Psi$ with $T^{-n_{0}}$ (which commutes with $R$), we may assume that $n_{0} = 0$.
Then from the relation $[\Psi] [R] [\Psi]^{-1} = [R]^{-1}$ we get that for every integer $n$, the arc $\Psi(R^n \alpha_{1})$ is isotopic to the arc $R^{-n} \alpha_{1}$ endowed with the opposite orientation. Up to composing $\Psi$ by a homeomorphism provided by the ``straightening principle'' (point 1 of lemma~\ref{lemm.straightening}), we may further assume that $\Psi(R^n \alpha_{1}) = R^{-n} \alpha_{1}$ for every $n$.
Thus $\Psi$ globally preserves the line $A_{1}$ while reversing its orientation. Since $\Psi$ preserves the orientation of the plane, it must send the half-plane on the left hand side of $A_{1}$ to its right hand side. But this contradicts the fact that $\Psi$ globally preserves $\bbZ_{2}$. The proof of the first point of Proposition~\ref{prop.centralisateur2} is complete.

The argument for the second point, namely that the mapping class of $T$ is conjugate to its inverse, has already been given (see the construction of $\Psi_{0}$ after the statement of the Proposition). 
We turn to the third point of the Proposition: we will prove that the centralizer of $[R]$ in the mapping class group is generated by $[T_{1}]$ and $[T_{2}]$.
The homeomorphism $T_{1} R T_{1}^{-1}$ coincides with $R$ on the union of the lines $\bbR \times \{1\}$ and $\bbR \times \{2\}$. 
Alexander's trick (Lemma~\ref{lemma.Alexander}) provides an isotopy from  $T_{1} R T_{1}^{-1}$ to $R$ relative to $\bbR \times \{1,2\}$, and in particular this is an isotopy relative to $\bbZ_1 \cup \bbZ_2$. In other words, the mapping class $[T_{1}]$ commutes with $[R]$. The same argument shows that $[T_{2}]$ commutes with $[R]$. It is easy to check that the subgroup generated by those two elements is free abelian.

Conversely, let $\Psi \in \homeo^{+}(\bbR^{2}; \bbZ_{1}, \bbZ_{2})$ such that $\Psi R \Psi^{-1}$ is isotopic to $R$, and let us show that $[\Psi]$ belongs to the subgroup generated by $[T_{1}]$ and $[T_{2}]$.
By uniqueness of the homotopy translation arcs, there exists $n_{1}, n_{2}$ such that, for every integer $n$, the arc $\Psi( R^n \alpha_{1})$ is isotopic to $R^{n+n_{1}} \alpha_{1}$, and 
$\Psi( R^n \alpha_{2})$ is isotopic to $R^{n+n_{2}} \alpha_{2}$. The ``straightening principle'' (point 1 of lemma~\ref{lemm.straightening}) provides a first isotopy relative to $\bbZ_{1} \cup \bbZ_{2}$ from $\Psi$ to a homeomorphism that coincides with $R^{n_{1}}$ on $A_{1}$ and with $R^{n_{2}}$ on $A_{2}$. Since 
$T_{1}^{n_{1}} T_{2}^{n_{2}}$ also coincides with $R^{n_{1}}$ on $A_{1}$ and with $R^{n_{2}}$ on $A_{2}$, Alexander's trick (Lemma~\ref{lemma.Alexander}) provides a second isotopy to $T_{1}^{n_{1}} T_{2}^{n_{2}}$, which proves that $[\Psi] = [T_{1}^{n_{1}} T_{2}^{n_{2}}]$, as wanted.

Finally, let us suppose that $T'=\Psi T \Psi^{-1}$ is isotopic to $T$ relative to $\bbZ_{1} \cup \bbZ_{2}$, and let us find a homeomorphism that is isotopic to $\Psi$ and that commutes with $T$. 
 Let $\alpha'_{1} = \Psi(\alpha_{1}), \alpha'_{2} = \Psi(\alpha_{2})$, and consider the family of curves
 $$
 \{ {T'}^{n}(\alpha'_{1}), {T'}^{n}(\alpha'_{2}) \mid n \in \bbZ\}.
 $$
By conjugacy, this is a locally finite family of pairwise disjoint curves. The ``straightening principle'' (point 1 of lemma~\ref{lemm.straightening}) provides a homeomorphism $\Psi_{1}$, isotopic to the identity relative to $\bbZ_{1} \cup \bbZ_{2}$, sending each curve in our family to its isotopic geodesic. On the other hand, since $T'$ is isotopic to $T$, the curve ${T'}^{n}(\alpha'_{1})$ is isotopic to ${T}^{n}(\alpha'_{1})$. We may have chosen the geodesic structure so that $T$ is an isometry (see section~\ref{subsection.basic}), and then the corresponding geodesic is $({T}^{n}(\alpha'_{1}))^{\sharp} = {T}^{n}({\alpha'_{1}}^{\sharp})$.
 Likewise, the geodesic in the isotopy class of ${T'}^{n}(\alpha'_{2})$ is  ${T}^{n}({\alpha'_{2}}^{\sharp})$.
Letting $\Psi' = \Psi_{1} \Psi$ we get,
for every integer $n$, $\Psi'(T^{n}(\alpha_{1})) = T^{n}({\alpha'_{1}}^{\sharp})$ and $\Psi'(T^{n}(\alpha_{2})) = T^{n}({\alpha'_{2}}^{\sharp})$. 
Thus the concatenation of the curves $T^{n}({\alpha'_{1}}^{\sharp})$ is a proper homotopy streamline which is invariant under $T$, the same holds for the concatenation of the curves  $T^{n}({\alpha'_{2}}^{\sharp})$, both homotopy streamlines are disjoint, and the first one is on the right hand side of the second one endowed with the orientation induced by the action of $T$.
Consider the quotient map $P : \bbR^2 \to \bbR^{2}/T$. The projections of ${\alpha'_{1}}^{\sharp}$ and ${\alpha'_{2}}^{\sharp}$ in this quotient are two disjoint simple closed curves, and $P {\alpha'_{1}}^{\sharp}$ is on the right hand side of $P {\alpha'_{2}}^{\sharp}$, where $P {\alpha'_{2}}^{\sharp}$ is oriented from $\Psi(0,2)$ to $\Psi(1,2)$. Of course, the projections of the curves $\alpha_{1}$ and $\alpha_{2}$ share the same properties. The classification of surfaces provides a homeomorphism of the quotient annulus that is isotopic to the identity and sends $P \alpha_{1}$ on $P {\alpha'_{1}}^{\sharp}$ and $P \alpha_{2}$ on $P {\alpha'_{2}}^{\sharp}$;  we may further demand that the image of the projection of a point $p \in \alpha_{1}$ is precisely the projection of the point $\Psi'(p)$, and likewise for $\alpha_{2}$. Lifting this homeomorphism to the plane, we get a homeomorphism $\Psi''$ which commutes with $T$ and agrees with $\Psi'$ on the lines $A_{1}$ and $A_{2}$. A last use of Alexander's tricks provides an isotopy (relative to $A_{1} \cup A_{2}$) from $\Psi'$ to $\Psi''$. Finally $[\Psi'']=[\Psi']=[\Psi]$ and $\Psi''$ commutes with $T$, which completes the proof.
\end{proof}

\section{Quasi-additivity}\label{section.additivity}
In this section we deduce the quasi-additivity Theorem~\ref{theo.2} from the classification in homotopy Brouwer theory provided by Theorem~\ref{theo.3}. We will use the characterization of flow classes in terms of proper homotopy streamlines (Lemma~\ref{lem.flow-streamlines}).

Let $h$ be a Brouwer homeomorphism and let us choose two orbits $\cO_{1}, \cO_{2}$ of $h$. Assume there exist two disjoint proper homotopy streamlines  $\Gamma_{1}, \Gamma_{2}$ for
 $[h ;  \cO_{1}, \cO_{2}]$ containing respectively $\cO_{1}$ and $\cO_{2}$. According to the ``straightening principle'' (Corollary~\ref{coro.straightening}), there exists $h' \in [h ;  \cO_{1}, \cO_{2}]$ which globally fixes $\Gamma_{1}$ and  $\Gamma_{2}$. 
The Schoenflies-Homma theorem (appendix~\ref{sec.schoenflies-homma}) provides an orientation preserving homeomorphism $\Phi$ that sends $\cO_{1}$ onto $\bbZ_{1}$, $\cO_{2}$ onto $\bbZ_{2}$, $\Gamma_{1}$ onto $\bbR \times \{1\}$ and $\Gamma_{2}$ onto $\bbR \times \{2\}$; we may further demand that $\Phi h' \Phi^{-1}$ coincides either with $T$ or $T^{-1}$ on each line $\bbR \times \{1\}$ and $\bbR \times \{2\}$. Then, by Alexander's trick (Lemma~\ref{lemma.Alexander}), $\Phi h' \Phi^{-1}$ is isotopic to $T,T^{-1},R$ or $R^{-1}$ relative to $\bbZ_{1} \cup \bbZ_{2}$, and so is $\Phi h \Phi^{-1}$. Thus
the index  $I(h, \cO_{1}, \cO_{2})$ is defined as the index of $\Phi h \Phi^{-1}$ along any curve going from a point of $\Phi(\cO_{1})$ to a point of $\Phi(\cO_{2})$ (section~\ref{subsection.index}).

We would like to be able to evaluate this index directly with $h$, $\Gamma_{1}$ and $\Gamma_{2}$. For this, we consider some curve $\alpha$ going from a point of $\cO_{1}$ to a point of $\cO_{2}$, and an isotopy $(\Phi_{t})$ from the identity to $\Phi$. The number
$$
I(\Phi_{t} h \Phi_{t} ^{-1}, \Phi_{t}(\alpha))
$$
varies continuously from $I(h,\alpha)$  to $I(\Phi h \Phi^{-1}, \Phi(\alpha)) = I(h, \cO_{1}, \cO_{2})$. Furthermore, the total variation $I(h, \cO_{1}, \cO_{2}) - I(h,\alpha)$ 
of this number may be expressed as follows. Let $a_{s}$ denote the angular variation, when $t$ goes from $0$ to $1$, of the vector joining the points $\Phi_{t} (\alpha(s))$ and $\Phi_{t} h \Phi_{t}^{-1} (\Phi_{t} (\alpha(s)))$. Then 
$$
I(h, \cO_{1}, \cO_{2}) - I(h,\alpha) = a_{1} - a_{0}.
$$
Indeed, consider the map
$$
\begin{array}{rcl}
\Delta : [0,1]^2 & \to & \bbR^2 \setminus \{0\} \\
(s,t) & \mapsto & \Phi_{t} h \Phi_{t}^{-1} (\Phi_{t} (\alpha(s))) - \Phi_{t}(\alpha(s)).
\end{array}
$$
Since the unit square is contractible, the winding number of $\Delta$ along the boundary of the unit square is zero. On the other hand this number is the difference between the two terms of the above equality.

 As an example, consider the situation depicted by Figure~\ref{figure.trois-trajectoires}.
 One can find a map $\Phi$ that sends the curves $\Gamma_{1}$, $\Gamma_{2}$ to parallel straight lines, and an isotopy from the identity to $\Phi$ during which the vector  joining $\alpha_{1}(0)$ to $\Phi_{t} h \Phi_{t}^{-1} (\alpha_{1}(0))$ is constant, while the vector joining  $\alpha_{1}(1)$ to $ \Phi_{t} h \Phi_{t}^{-1} (\alpha_{1}(1))$ undergoes an angular variation of a sixth of a full turn in the positive direction. Thus $I(h, \cO_{1}, \cO_{2}) = I(h,\alpha_{1}) +1/6$.

\begin{figure}[h!]
\begin{center}
\begin{tikzpicture}[scale=0.7]
\tikzstyle{fleche pointille}=[>=latex,->,dashed]

\begin{scope}
\path (-4,-5) .. controls +(45:1) and  +(left:1.7) ..
 (-1,-3) node[inner sep=0mm,pos=0.25,name=P1] {$\bullet$} node[inner sep=0mm,pos=0.5,name=P2] {$\bullet$} node[inner sep=0mm,pos=0.75,name=P3] {$\bullet$} node[inner sep=0mm,pos=1,name=P4] {$\bullet$}
-- (0,-3) node [inner sep=0mm,name=A1] {$\bullet$} 
-- (1,-3) node [inner sep=0mm,name=B1] {$\bullet$} 
.. controls +(right:1.7) and +(135:1) ..
(4, -5)  node[inner sep=0mm,pos=0.25,name=Q1] {$\bullet$} node[inner sep=0mm,pos=0.5,name=Q2] {$\bullet$} node[inner sep=0mm,pos=0.75,name=Q3] {$\bullet$}
node [very near end,right,inner sep=5mm] {$\Gamma_{1}$} ;

\draw (-4,-5) -- (P1)  -- ({P2})  -- ({P3}) -- (P4) -- (A1) -- (B1) -- ({Q1})  -- ({Q2})  -- ({Q3}) -- (4,-5) ;

\end{scope}

\begin{scope}[rotate=120]
\path (-4,-5) .. controls +(45:1) and  +(left:1.7) ..
 (-1,-3) node[inner sep=0mm,pos=0.25,name=P1] {$\times$} node[inner sep=0mm,pos=0.5,name=P2] {$\times$} node[inner sep=0mm,pos=0.75,name=P3] {$\times$} node[inner sep=0mm,pos=1,name=P4] {$\times$}
-- (0,-3) node [inner sep=0mm,name=A2] {$\times$} 
-- (1,-3) node [inner sep=0mm,name=B1] {$\times$} 
.. controls +(right:1.7) and +(135:1) ..
(4, -5)  node[inner sep=0mm,pos=0.25,name=Q1] {$\times$} node[inner sep=0mm,pos=0.5,name=Q2] {$\times$} node[inner sep=0mm,pos=0.75,name=Q3] {$\times$}
node [very near end,left,inner sep=5mm] {$\Gamma_{2}$} ;

\draw (-4,-5) -- (P1)  -- ({P2})  -- ({P3}) -- (P4) -- (A2) -- (B1) -- ({Q1})  -- ({Q2})  -- ({Q3}) -- (4,-5) ;
\end{scope}

\begin{scope}[rotate=240]
\path (-4,-5) .. controls +(45:1) and  +(left:1.7) ..
 (-1,-3) node[inner sep=0mm,pos=0.25,name=P1] {$\circ$} node[inner sep=0mm,pos=0.5,name=P2] {$\circ$} node[inner sep=0mm,pos=0.75,name=P3] {$\circ$} node[inner sep=0mm,pos=1,name=P4] {$\circ$}
-- (0,-3) node [inner sep=0mm,name=A3] {$\circ$} 
-- (1,-3) node [inner sep=0mm,name=B1] {$\circ$} 
.. controls +(right:1.7) and +(135:1) ..
(4, -5)  node[inner sep=0mm,pos=0.25,name=Q1] {$\circ$} node[inner sep=0mm,pos=0.5,name=Q2] {$\circ$} node[inner sep=0mm,pos=0.75,name=Q3] {$\circ$}
node [very near end,below,inner sep=5mm] {$\Gamma_{3}$} ;

\draw (-4,-5) -- (P1)  -- ({P2})  -- ({P3}) -- (P4) -- (A3) -- (B1) -- ({Q1})  -- ({Q2})  -- ({Q3}) -- (4,-5) ;
\end{scope}

\draw [fleche pointille] (A1) -- (A2) node [near start,above,inner sep=4mm] {$\alpha_{1}$} ;
\draw [fleche pointille] (A2) -- (A3) node [midway,above] {$\alpha_{2}$} ;

\end{tikzpicture}
\end{center}
\caption{Three streamlines in the non-separating case}
\label{figure.trois-trajectoires}
\end{figure}
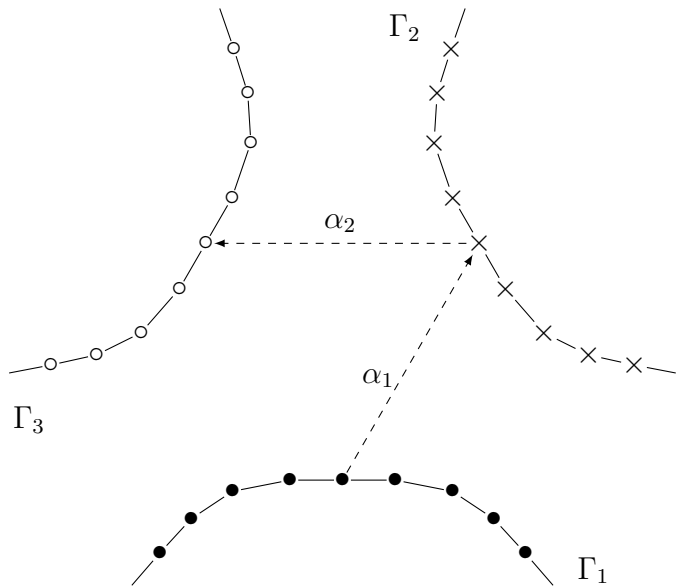

Now consider three orbits $\cO_{1}, \cO_{2}, \cO_{3}$ of $h$. According to the classification of Theorem~\ref{theo.3}, and its practical version given by Lemma~\ref{lem.flow-streamlines}, we may find pairwise disjoint proper homotopy streamlines $\Gamma_{1}, \Gamma_{2}, \Gamma_{3}$  for
 $[h ;  \cO_{1}, \cO_{2}, \cO_{3}]$ containing respectively $\cO_{1},\cO_{2}, \cO_{3}$. The topological lines $\Gamma_{1}, \Gamma_{2}$ are also proper homotopy streamlines with respect to  $[h ;  \cO_{1}, \cO_{2}]$, and thus the above discussion applies. Now there are two cases. The first case happens when one of the three lines separates\footnote{A subset $A$ of a topological space $X$ is said to \emph{separate} two other subsets $B,C$ if $A$ is disjoint from $B$ and $C$ and no connected component of $X \setminus A$ meets both $B$ and $C$.}
  the other two, and the situation is topologically equivalent to three parallel straight lines. Since our index is a conjugacy invariant, we may assume that the streamlines are parallel straight lines. In this situation we choose any curve $\alpha_{1}$ from $\cO_{1}$ to $\cO_{2}$, and any curve $\alpha_{2}$ from $\cO_{2}$ to $\cO_{3}$, with $\alpha_{1}(1) = \alpha_{2}(0)$. We have
 $$
I(h, \cO_{1}, \cO_{2}) =  I(h,\alpha_{1}), \ \ I(h, \cO_{2}, \cO_{3}) =  I(h,\alpha_{2}), 
 $$
 $$
 I(h, \cO_{1}, \cO_{3}) =   I(h,\alpha_{1} \star \alpha_{2}) = I(h,\alpha_{1}) + I(h,\alpha_{2}).
 $$
Since $I(h, \cO_{3}, \cO_{1}) = -I(h, \cO_{1}, \cO_{3})$,  in this case we get perfect additivity, namely 
$$
I(h, \cO_{1}, \cO_{2}) + I(h, \cO_{2}, \cO_{3}) + I(h, \cO_{3}, \cO_{1}) = 0.
$$

In the second case, none of the three lines separates the other two, as on figure~\ref{figure.trois-trajectoires}.
There exists a homeomorphism of the plane that sends the general situation to the particular situation depicted on the Figure. To begin with let us assume that this homeomorphism preserves the orientation of the plane. Then we may again assume the three lines are those depicted on the Figure. We pick two curves $\alpha_{1}, \alpha_{2}$ as above.  
According to the above discussion, this time we get
$$
I(h, \cO_{1}, \cO_{2}) =  I(h,\alpha_{1})+1/6, \ \ I(h, \cO_{2}, \cO_{3}) =  I(h,\alpha_{2}) + \frac{1}{6}, 
$$
$$
 I(h, \cO_{3}, \cO_{1}) = -  I(h,\alpha_{1} \star \alpha_{2}) +\frac{1}{6} 
 $$
and thus 
$$
I(h, \cO_{1}, \cO_{2}) + I(h, \cO_{2}, \cO_{3}) + I(h, \cO_{3}, \cO_{1}) = \frac{1}{2}.
$$
It remains to consider the case when the situation of the Figure happens only up to an  orientation reversing homeomorphism. This case  yields an entirely similar computation, where the '$+1/6$' are replaced by '$-1/6$', and thus the final '$+1/2$' becomes a '$-1/2$'. In any case the quasi-additivity relation holds.

\section{Brouwer classes relative to three orbits}\label{section.three}

In this section we prove Theorem~\ref{theo.3}.

\subsection{Advanced tools for homotopy Brouwer theory}
We need to go deeper into homotopy Brouwer theory, borrowing from~\cite{handel1999fixed} and~\cite{franks2010entropy}. The statements that are suitable for our needs are not explicitly stated in the quoted papers. In Appendix~A we will explain how to get these statements (Proposition~\ref{prop.half-proper-streamlines} and~\ref{prop.reducing-line} below) from those papers.

We adopt the notations of section~\ref{subsection.basic}, and we generalize the notion of homotopy translation arc to encompass the case of an arc meeting several $\cO_{i}$'s. Consider an arc $\alpha$ which is the concatenation $\alpha_{1} \star \dots  \star  \alpha_{k}$ of
some elements $\alpha_{i}$ in $\cA_{0}$. Assume that $\alpha(1) = h(\alpha(0))$, and that for every $n,i,j$, the arcs $h^{n}(\alpha_{i})$ and $\alpha_{j}$ are homotopically disjoint. Then $\alpha$  is called a \emph{(generalized) homotopy translation arc} for $[h;\cO_{1}, \dots , \cO_{r}]$. A typical example is shown on figure~\ref{figure.generalized-hta-ex}.

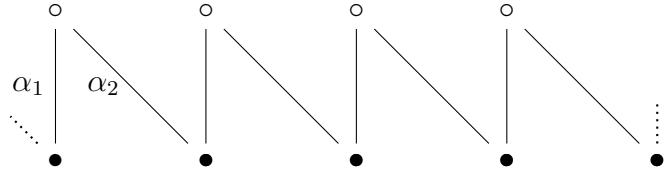
\begin{figure}[h!]
\begin{center}
\begin{tikzpicture}[scale=2]
\foreach \k in {0,...,3} 	{\node [name=A\k] at (\k,1)  {$\bullet$} ; \node [name=B\k] at  (\k,2) {$\circ$} ;\node [name=C\k] at  (\k+1,1) {$\circ$} ;}
\node [name=A4] at (4,1)  {$\bullet$} ;
\foreach \k in {0,...,3} {\draw (A\k)  --  (B\k) --(C\k) ;}
\draw [thick,dotted] (A0) -- (-0.3,1.3) ;
\draw [thick,dotted] (A4) -- (4,1.4) ;
\node at (0,1.5) [left] {$\alpha_{1}$} ;
\node at (0.5,1.5) [left] {$\alpha_{2}$} ;
\end{tikzpicture} 
\end{center}
\caption{A generalized homotopy translation arc containing two orbits}
\label{figure.generalized-hta-ex}
\end{figure}

The (generalized) homotopy translation arc  will be called \emph{forward proper} if $\alpha_{1},  \cdots ,  \alpha_{k}$ are forward proper. Equivalently, the half-infinite concatenation
$$
A^{+} =  \alpha_{1}^\sharp \star \cdots \star \alpha_{k}^\sharp   \star  (h\alpha_{1})^\sharp \star \cdots \star (h\alpha_{k})^\sharp  \star \cdots
$$ 
is the image of $[0,+\infty)$ under a proper injective continuous map. This set will be called a \emph{(generalized) proper forward homotopy streamline}.
A \emph{(generalized) proper backward homotopy streamline} is defined symmetrically. We will need the existence of a nice family of (generalized) proper backward and forward homotopy streamlines.

\begin{prop}[Handel]
\label{prop.half-proper-streamlines}
Let $h$ be a Brouwer homeomorphism, and $x_{1}, \dots , x_{r}$ points belonging to distinct orbits $\cO_{1}, \dots,  \cO_{r}$ of $h$.
Then there exist a number $1 \leq r' \leq r$ and a family of $2r'$ pairwise disjoint (generalized) proper backward or forward  homotopy streamlines, whose union contains  all but a finite number of points of the union of the $\cO_{i}$'s. Furthermore, the backward and forward streamlines alternate in the cyclic order at infinity.
\end{prop}
The backward homotopy streamlines provided by this Proposition will be denoted by $A^{-}_{1}, \dots , A^{-}_{r'}$, and the forward homotopy streamlines by
$A^{+}_{1}, \dots , A^{+}_{r'}$. The condition on the cyclic order at infinity means that we may choose the numbering so that there exists an arbitrarily small neighbourhood of $\infty$, whose boundary is an (oriented) Jordan curve that meets each streamline exactly once, and that meets $A^+_{1}, A^-_{1}, A^+_{2}, A^-_{2}, \dots , A^-_{r'}$ in that order.
A typical example is displayed on figure~\ref{figure.streamlines}.

\begin{figure}[h]
\begin{center}
\begin{tikzpicture}[scale=1]
\tikzstyle{fleche}=[>=latex,->]
\tikzstyle{fleche passee}=[>=latex,>-]
\draw [fleche,dotted] (0,0) -- (10,0) ; 
\draw [fleche,dotted] (10,4) -- (0,4) ;

\foreach \k in {1,2,...,9} 
	{\draw (\k,0) node [name=A\k,,inner sep=0mm] {$\bullet$} ; \draw (\k,4) node  [name=B\k,,inner sep=0mm] {\footnotesize$\times$} ; }
\draw [fleche,dotted] (0,0.5) .. controls +(right:5) and +(right:5) .. (0,3.5) 
	\foreach \p in {1,2,...,9} {node[pos=\p/10,name=C\p,minimum size=1.5mm,inner sep=1mm] {}} ;
 
\foreach \p in {1,2,...,9} { \node at (C\p) [inner sep=0mm] {$\circ$} ;} 

\draw  [very thick, dotted] (0.5,0.25) -- (A1) ;
\draw [very thick,fleche] (A1) --(C1) --(A2) -- (C2) -- (A3) --(C3) --(A4)  ;
\node at (A3)  [below] {$A_{1}^-$};

\draw [very thick,fleche passee] (A6) --(A7) ; \draw [thick] (A7) -- (A8) -- (A9) ;
\draw [very thick, dotted] (A9) -- (9.5,0) ;
\node at (A8) [below] {$A_{2}^+$};

\draw [very thick, dotted] (0.5,3.75) -- (B1) ;
\draw [very thick] (B1)-- (C9) --(B2) -- (C8) -- (B3) --(C7)  ; \draw [very thick,fleche passee] (B4) --(C7) ; 
\node at (B3)  [above] {$A_{1}^+$};

\draw [very thick,fleche] (B9) --(B8) -- (B7) -- (B6) ;
\draw [very thick, dotted] (B9) -- (9.5,4) ;
\node at (B8) [above] {$A_{2}^-$};

\end{tikzpicture} 
\end{center}
\caption{Alternating forward and backward proper homotopy streamlines for a Brouwer class: here $r=3$ and $r'=2$}
\label{figure.streamlines}
\end{figure}
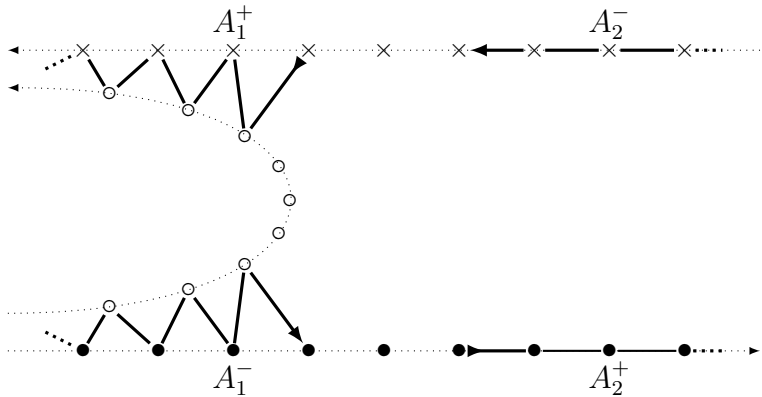

A \emph{reducing line} for $[h; \cO_{1}, \dots, \cO_{r}]$ is a topological line, \emph{i. e.} the image of a proper continuous injective map $\Delta : \bbR \to \bbR^{2}$, which is disjoint from $\cO = \cO_{1} \cup \dots \cup \cO_{r}$, such that none of the two topological half-planes delimited by $\Delta$ contains all the $\cO_{i}$'s, 
and such that $\Delta$ is \emph{properly isotopic} to  $h(\Delta)$ relative to $\cO$: there exists a continuous proper map $H : \bbR \times [0,1] \to \bbR^{2}$ such that $H(.,0) = \Delta$, $H(.,1) = h \Delta$, and for every $t$ the map $H(.,t)$ is  injective.

Let $\cA_{1}$ denote the set of topological lines. Exactly as for $\cA_{0}$, each element in $\cA_{1}$ is properly isotopic relative to $\cO$ to a unique geodesic for the hyperbolic metric, and the ``straightening principle'' still holds
(see~\cite{handel1999fixed} for details).
\begin{lemm}\label{lemm.straightening2}
Lemma~\ref{lemm.straightening} holds when $\cA_{0}$ is replaced by $\cA_{0} \cup \cA_{1}$.
\end{lemm}
Thus if $\Delta$ is a reducing line we may find an element $h'$ of the mapping class $[h; \cO_{1}, \dots, \cO_{r}]$ such that $h'(\Delta) = \Delta$. In particular each orbit $\cO_{i}$ has to be entirely on one side of $\Delta$.

We will say that $[h; \cO_{1}, \dots, \cO_{r}]$ is a \emph{translation class} if it contains a homeomorphism which is conjugate to a translation. (Equivalently, 
$[h; \cO_{1}, \dots, \cO_{r}]$ is a translation class if and only if it is conjugate to $[T; \bbZ \times\{1\}, \dots, \bbZ \times \{r\}]$.)
\begin{prop}[Handel]
\label{prop.reducing-line}
We assume the notation of the preceding proposition.

\begin{enumerate}
\item If $r'=1$, then $[h;\cO_{1}, \dots , \cO_{r}]$ is a translation class.
\item If $r'=r$, then there exists a reducing line which is disjoint from the $2r$ homotopy backward and forward streamlines of the preceding proposition.
\item More generally, as soon as $r \geq2$,  there exists a reducing line which is disjoint from the $r'$ backward proper homotopy streamlines, and from all the forward proper homotopy streamlines that meet only one of the $\cO_{i}$'s.
\end{enumerate}
\end{prop}
We call the first case the \emph{translation case}, the second case the \emph{alternating case}, and the remaining case the \emph{remaining case}.
Note that in the alternating case each homotopy translation arc provided by the first above proposition meets only one orbit; in other words, these are homotopy translation arcs in the non generalized meaning.  The alternating case generalizes the \emph{multiple Reeb class} case (see Figure~\ref{figure.alternating-case}) in which the cyclic order at infinity on the set of proper homotopy streamlines given by Proposition~\ref{prop.half-proper-streamlines} is  $A_{1}^{-}, A_{2}^{+} , \dots ,  A_{r}^{\varepsilon}, A_{r}^{-\varepsilon}, \dots,A_{2}^{-}, A_{1}^{+}$, where $A_{i}^-$ and $A_{i}^+ $ are the backward and forward homotopy streamlines meeting the same orbit $\cO_{i}$, and $\varepsilon=\pm 1$ according to the parity of $r$. The multiple Reeb class case is treated by Franks and Handel (Lemma 8.9 in~\cite{franks2003periodic}).

\begin{figure}[h!]
\begin{center}
\begin{tikzpicture}[scale=1,thick]
\tikzstyle{fleche}=[>=latex,->]
\tikzstyle{fleche passee}=[>=latex,>-]

\draw [fleche] (4,0) node [right] {$A_{1}^{-}$} -- (3,0) ;
\draw [dotted] (3,0) -- (1,0) ;
\draw [fleche passee] (1,0)  -- (0,0) node [left] {$A_{1}^{+}$} ;

\draw [fleche] (0,1) node [left] {$A_{2}^{-}$} -- (1,1) ;
\draw [dotted] (1,1) -- (3,1) ;
\draw [fleche passee] (3,1)  -- (4,1) node [right] {$A_{2}^{+}$} ;

\draw [fleche] (4,2) node [right] {$A_{3}^{-}$} -- (3,2) ;
\draw [dotted] (3,2) -- (1,2) ;
\draw [fleche passee] (1,2)  -- (0,2) node [left] {$A_{3}^{+}$} ;

\draw [fleche] (0,3) node [left] {$A_{4}^{-}$} -- (1,3) ;
\draw [dotted] (1,3) -- (3,3) ;
\draw [fleche passee] (3,3)  -- (4,3) node [right] {$A_{4}^{+}$} ;

\begin{scope}[xshift=8cm]
\draw [fleche] (4,0) node [right] {$A_{1}^{-}$} -- (3,0) ;
\draw [dotted] (3,0) -- (1,0) ;
\draw [fleche passee] (1,0)  -- (0,0) node [left] {$A_{1}^{+}$} ;

\draw [fleche] (0,1) node [left] {$A_{2}^{-}$} -- (1,1) ;
\draw [dotted] (1,1) -- (3,1) ;
\draw [fleche passee] (3,1)  -- (4,1) node [right] {$A_{2}^{+}$} ;

\draw [fleche] (4,2) node [right] {$A_{3}^{-}$} -- (3,2) ;
\draw [dotted] (3,2) ..controls +(-0.5,0) and +(0,-0.5)..    (2.5,2.5) ;
\draw [fleche passee] (2.5,2.5)  -- (2.5,3.5) node [above] {$A_{3}^{+}$} ;

\draw [fleche] (1.5,3.5) node [above] {$A_{4}^{-}$} -- (1.5,2.5) ;
\draw [dotted] (1.5,2.5) ..controls +(0,-0.5) and +(-0.5,0)..    (1,2) ;
\draw [fleche passee] (1,2)  -- (0,2) node [left] {$A_{4}^{+}$} ;
\end{scope}

\end{tikzpicture} 
\end{center}
\caption{A multiple Reeb class (left) and another alternating case (right)}
\label{figure.alternating-case}
\end{figure}
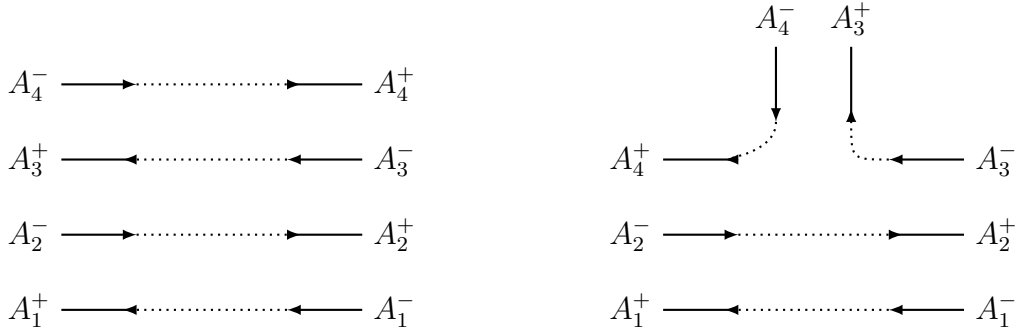
 
 Lemma~\ref{lem.flow-streamlines} above provides a caracterization of flow classes in terms of homotopy streamlines. The proof may be easily adapted to the case of  generalized  homotopy streamlines, and we leave this task to the reader. For future reference we state the corresponding result.
\begin{lemm} \label{lem.flow-streamlines2}
 Lemma~\ref{lem.flow-streamlines} still holds when we replace homotopy streamlines with generalized homotopy streamlines.
\end{lemm}

The easiest way to construct a reducing line is to consider a proper homotopy streamline and to ``push it on the side'', as explicited by the following lemma.
\begin{lemm}\label{lemm.reducing-line}
Let $A$ be a proper homotopy streamline for some $[h;\cO_{1}, \dots , \cO_{r}]$. 
Let $U_{1}, U_{2}$ be the two connected components of $\bbR^2 \setminus A$, and assume that $U_{1}$ contains some points of $\cO$. Then there exists a reducing line $\Delta$ that separates $\cO' := U_{1} \cap \cO$ from  $\cO \setminus \cO'$. Furthermore $\Delta$ may be chosen to be included in any given open set containing $A$.
\end{lemm}
\begin{proof}
By the straightening principle (Corollary~\ref{coro.straightening}), there is some $h' \in [h;\cO_{1}, \dots , \cO_{r}]$ such that $A$ is invariant under $h'$.
By applying a conjugacy we may assume that $A$ is a straight vertical line on which the restriction of $h'$ is a translation. A further conjugacy ensures the existence of a vertical strip $B$ containing $A$ such that $B \cap \cO = A \cap \cO$.  A variation on Alexander's trick provides some $h''\in [h;\cO_{1}, \dots , \cO_{r}]$ such that $h'' = h'$ outside $B$ and $h''$ is a translation in some smaller vertical strip $B'$ still containing $A$. Now it is clear that one of the two boundary components of $B'$ is a reducing line for $[h'';\cO_{1}, \dots , \cO_{r}] = [h;\cO_{1}, \dots , \cO_{r}]$. Given an open set $V$ containing $A$ one can push $\Delta$ inside $V$ by an isotopy relative to $\cO$, which proves the last sentence of the lemma.
\end{proof}

We will now turn to the proof of Theorem~\ref{theo.3}. 
In the translation case the mapping class of $h$ contains a translation, which is certainly the time one map of a flow, so there is nothing to prove in this case.
We will give a separate argument in the alternating case and in the remaining case. According to Lemma~\ref{lem.flow-streamlines}, it suffices to find a family of pairwise disjoint (generalized) proper homotopy streamlines. For this we will consider the (generalized) homotopy translation arcs $\alpha^{-}_{1}, \dots , \alpha^{-}_{r'}$ that generate the backward proper homotopy streamlines $A^{-}_{1}, \dots , A^{-}_{r'}$. In the alternating case, we will prove that the $\alpha_{j}$'s are also forward proper, thus they will generate the wanted family of homotopy streamlines. Life will be a little more complicated in the remaining case, especially in one sub-case where one homotopy streamline will be provided by forward iterating one of the $A^{-}_{j}$'s,  another one by backward iterating one of the $A^{+}_{j}$'s, and the last one by a separate argument.

\subsection{The alternating case}
In this case we prove a stronger result, namely that  the mapping class of $h$ is a flow class, not only when $r=3$ as is required by the Theorem, but for every $r \geq 1$.
Note that in the present section all homotopy translation arcs will be non generalized ones.
We choose once and for all a Brouwer homeomorphism $h$, and we introduce the following definitions.
A family  $\cO_{1}, \dots, \cO_{r}$ is  \emph{alternating} if there exists a family of  $2r$ pairwise disjoint (non generalized) proper backward and forward  homotopy streamlines as in the conclusion of Proposition~\ref{prop.half-proper-streamlines}. This family \emph{satisfies the uniqueness of homotopy translation arcs} if for every point $x \in \cO_{1} \cup \cdots \cup \cO_{r}$ there exists a unique homotopy class, relative to these orbits, of homotopy translation arcs joining $x$ and $h(x)$.

\begin{lemm}\label{lemm.uniqueness}
Every family of  alternating orbits satisfies the uniqueness of homotopy translation arcs.
\end{lemm}
Given the lemma, we consider a family of alternating orbits, and $2r$ backward and forward homotopy streamlines generated by homotopy translation arcs $\alpha_{i}^{\pm}$. By uniqueness of homotopy translation arcs, for every $i$, some iterate of $\alpha_{i}^{-}$ is homotopic to $\alpha_{i}^{+}$. In particular, $\alpha_{i}^{-}$ is both backward and forward proper. Furthermore, the proper streamlines generated by the $\alpha_{i}^{-}$'s are pairwise disjoint, since they arise from the iteration of pairwise disjoint backward homotopy streamlines. We apply Lemma~\ref{lem.flow-streamlines} to conclude that $[h;\cO_{1}, \dots , \cO_{r}]$ is a flow class, which completes the proof that the mapping class of $h$ relative to any alternating family of orbits is a flow class.

\bigskip	

Now let us prove the Lemma. The proof is by induction on the number $r$ of orbits.
For $r=1$ we have a translation class, and the uniqueness of homotopy translation arcs is Corollary 6.3 of~\cite{handel1999fixed} (see also Corollary 2.1 in~\cite{leroux2012introduction}).
Let $r \geq 2$ and let  $\cO_{1}, \dots, \cO_{r}$ be  a family of $r$ alternating orbits. We consider the $2r$ pairwise disjoint proper backward and forward  homotopy streamlines given by the definition. Let $\Delta$ be a reducing line provided by point 2 of Proposition~\ref{prop.reducing-line}.
By the ``straightening principle'' (Lemma~\ref{lemm.straightening2}) we may find some $h'$ in the mapping class $[h;\cO_{1}, \dots, \cO_{r}]$ such that $h'(\Delta) = \Delta$.
The topological line $\Delta$ splits the set $\cO= \cO_{1} \cup \cdots \cup \cO_{r}$ into two subsets, say $\cO = \cO' \sqcup \cO''$, each one consisting in less than $r$ orbits of $h$. Note that $h$ and $h'$ still have the same mapping class relative to $\cO'$.
Furthermore, since the reducing line $\Delta$ is disjoint from the backward and forward homotopy streamlines, it also splits the family $\cS$ of streamlines into two sub-families that we denote by $\cS'$ and $\cS''$. The key point is that $\cS'$ is an interval of $\cS$ in the cyclic order at infinity, and since
the backward and forward streamlines in $\cS$ alternate, this is still true in $\cS'$.
In particular, the family of orbits making up $\cO'$ is alternating.
 Choose some orbit $\cO_i$ in $\cO'$ and denote by $\alpha^-, \alpha^{+}$ the backward and forward proper homotopy translation arcs inducing the elements of $\cS'$ that meet $\cO_{i}$. By the induction hypothesis, the family of orbits in $\cO'$ satisfies the uniqueness of homotopy translation arcs, and thus  there is an iterate $h^{N}\alpha^{-}$ which is homotopic to $\alpha^{+}$ with respect to $\cO'$.
Obviously  $h'^{N}\alpha^{-}$ is also homotopic to $\alpha^{+}$  relative to $\cO'$. Since both arcs are disjoint from $\Delta$, this homotopy may be chosen to be disjoint from $\Delta$ (compose the homotopy with a retraction from the plane to the half plane delimited by $\Delta$ and containing both arcs). Then it is a homotopy  relative to $\cO$. Thus $h^{N} \alpha^{-}$ is homotopic to $\alpha^{+}$ relative  to $\cO$. In particular, the backward proper homotopy translation arc $\alpha^{-}$  is also forward proper. The same argument applies to $\cS''$. Finally the $r$ proper backward homotopy translation arcs provided by Proposition~\ref{prop.half-proper-streamlines} are also forward proper. Since the corresponding proper backward homotopy streamlines are pairwise disjoint, the proper homotopy streamlines are also pairwise disjoint. Note that the argument also shows that the union of the proper homotopy streamlines contain all the $A^{-}_{i}$'s and the $A^{+}_{i}$'s. We are now in a position to apply the following lemma, which shows that $\cO_{1}, \dots, \cO_{r}$ satisfies the uniqueness of homotopy translation arcs. This concludes the induction.

\begin{lemm}\label{lemma.uniqueness-hta}
Let $\cO_{1}, \dots, \cO_{r}$ be distinct orbits of $h$. Assume that there exists a proper homotopy translation arc $\alpha_{i}$ for each orbit $\cO_{i}$, such that the associated proper homotopy streamlines $A_{1}, \dots A_{r}$ are pairwise disjoint. Also assume that the positive and negative ends of these streamlines alternates in the cyclic order at infinity. 
Then  every (non generalized) homotopy translation arc $\alpha$ for $[h;\cO_{1}, \dots, \cO_{r}]$ is isotopic relative to $\cO_{1} \cup \dots \cup  \cO_{r}$ to some iterate of one of the $\alpha_{i}$'s.
\end{lemm}

The proof of this Lemma is virtually the same as the proof of Lemma 8.7 (2) in~\cite{franks2003periodic}, details are left to the reader.
Note that the hypotheses of the lemma is satisfied by the mapping class of a translation relative to a single orbit, and by the mapping class $[R]$ that occurred in section~\ref{section.index}.

\subsection{The remaining case}
Note that when $r=1$ we are in the translation case. If $r=2$ we are either in the translation or in the alternating case, according to whether $r' = 1$ or $r'=2$. We are left with the case when $r=3$ and $r'=2$, so that we have two backward homotopy streamlines $A^-,B^{-}$ and  two forward homotopy streamlines $A^+,B^{+}$. We also denote by $\alpha^\pm, \beta^\pm$ the corresponding generalized homotopy translation arcs.
 We choose the notation so that $A^{-}$ and $A^{+}$ are the ones that meet two of the orbits $\cO_{i}$, whereas $B^{-}$ and $B^{+}$ each meet a single orbit. There are two cases to consider (see Figure~\ref{figure.subcases}): either $B^{-}$ and $B^{+}$ meet the same orbit, or not.
In any case, since the backward and forward streamlines alternate at infinity, up to conjugating by an orientation reversing homeomorphism, we may assume that the cyclic order is  $A^{+},A^{-}, B^{+}, B^{-}$. 

\begin{figure}[h]
\begin{center}
\begin{tikzpicture}[scale=0.19]
\tikzstyle{fleche}=[>=latex,->]
\tikzstyle{fleche passee}=[>=latex,>-]

\begin{scope}[xshift=-20cm]
\tiny
\path  (14,10) -- (3,3) 	node[pos=0,name=A0,above right] {\small $B^{-}$} 
					node[inner sep=0mm,minimum size=1.5mm,pos=0.125,name=A1] {$\times$} 
					node[inner sep=0mm,minimum size=1.5mm,pos=0.375,name=A2] {$\times$} 
					node[inner sep=0mm,minimum size=1.5mm,pos=0.625,name=A3] {$\times$} 
					node[inner sep=0mm,minimum size=1.5mm,pos=0.875,name=A4]  {$\times$} ;
\draw [dotted] (A0) -- (A1) ;
\draw [fleche] (A1) -- (A2) -- (A3) -- (A4) ;

\path  (-3,3) -- (-14,10) 	node[pos=0.125,inner sep=0mm,minimum size=1mm,name=B1] {$\bullet$}  
					node[pos=0.25,inner sep=0mm,minimum size=1mm,name=B2] {$\circ$} 
					node[pos=0.375,inner sep=0mm,minimum size=1mm,name=B3] {$\bullet$} 
					node[pos=0.5,inner sep=0mm,minimum size=1mm,name=B4] {$\circ$} 
					node[pos=0.625,inner sep=0mm,minimum size=1mm,name=B5] {$\bullet$} 
					node[pos=0.75,inner sep=0mm,minimum size=1mm,name=B6] {$\circ$} 
					node[pos=0.875,inner sep=0mm,minimum size=1mm,name=B7] {$\bullet$} 
					node[pos=1,name=B8,above left] {\small $A^{+}$}  ;
\draw [fleche passee] (B1) -- (B2) ;
\draw (B2) -- (B3) -- (B4) -- (B5) -- (B6) -- (B7) ;
\draw [dotted] (B8) -- (B7) ;

\path  (-3,-3) -- (-14,-10) 	node[pos=0.125,inner sep=0mm,minimum size=1mm,name=B1] {$\bullet$}  
					node[pos=0.25,inner sep=0mm,minimum size=1mm,name=B2] {$\circ$} 
					node[pos=0.375,inner sep=0mm,minimum size=1mm,name=B3] {$\bullet$} 
					node[pos=0.5,inner sep=0mm,minimum size=1mm,name=B4] {$\circ$} 
					node[pos=0.625,inner sep=0mm,minimum size=1mm,name=B5] {$\bullet$} 
					node[pos=0.75,inner sep=0mm,minimum size=1mm,name=B6] {$\circ$} 
					node[pos=0.875,inner sep=0mm,minimum size=1mm,name=B7] {$\bullet$} 
					node[pos=1,name=B8,below left] {\small $A^-$} ;
\draw [fleche] (B2) -- (B1) ;
\draw (B2) -- (B3) -- (B4) -- (B5) -- (B6) -- (B7) ;
\draw [dotted] (B8) -- (B7) ;

\path  (14,-10) -- (3,-3) 	node[pos=0,name=A0,below right] {\small $B^{+}$} 
					node[inner sep=0mm,minimum size=1.5mm,pos=0.125,name=A1] {$\times$} 
					node[inner sep=0mm,minimum size=1.5mm,pos=0.375,name=A2] {$\times$} 
					node[inner sep=0mm,minimum size=1.5mm,pos=0.625,name=A3] {$\times$} 
					node[inner sep=0mm,minimum size=1.5mm,pos=0.875,name=A4]  {$\times$} ;
\draw [dotted] (A0) -- (A1) ;
\draw (A1) -- (A2) -- (A3)  ;
\draw [fleche passee] (A4)--(A3) ;

\end{scope}


\begin{scope}[xshift=20cm]
\tiny
\path  (14,10) -- (3,3) 	node[pos=0,name=A0,above right] {\small $B^{-}$}
					node[inner sep=0mm,minimum size=1.5mm,pos=0.125,name=A1] {$\times$} 
					node[inner sep=0mm,minimum size=1.5mm,pos=0.375,name=A2] {$\times$} 
					node[inner sep=0mm,minimum size=1.5mm,pos=0.625,name=A3] {$\times$} 
					node[inner sep=0mm,minimum size=1.5mm,pos=0.875,name=A4]  {$\times$} ;
\draw [dotted] (A0) -- (A1) ;
\draw [fleche] (A1) -- (A2) -- (A3) -- (A4) ;

\path  (-3,3) -- (-14,10) 	node[pos=0.125,inner sep=0mm,minimum size=1mm,name=B1] {$\times$}  
					node[pos=0.25,inner sep=0mm,minimum size=1mm,name=B2] {$\circ$} 
					node[pos=0.375,inner sep=0mm,minimum size=1mm,name=B3] {$\times$} 
					node[pos=0.5,inner sep=0mm,minimum size=1mm,name=B4] {$\circ$} 
					node[pos=0.625,inner sep=0mm,minimum size=1mm,name=B5] {$\times$} 
					node[pos=0.75,inner sep=0mm,minimum size=1mm,name=B6] {$\circ$} 
					node[pos=0.875,inner sep=0mm,minimum size=1mm,name=B7] {$\times$} 
					node[pos=1,name=B8,above left] {\small $A^{+}$}  ;
\draw [fleche passee] (B1) -- (B2) ;
\draw (B2) -- (B3) -- (B4) -- (B5) -- (B6) -- (B7) ;
\draw [dotted] (B8) -- (B7) ;

\path  (-3,-3) -- (-14,-10) 	node[pos=0.125,inner sep=0mm,minimum size=1mm,name=B1] {$\bullet$}  
					node[pos=0.25,inner sep=0mm,minimum size=1mm,name=B2] {$\circ$} 
					node[pos=0.375,inner sep=0mm,minimum size=1mm,name=B3] {$\bullet$} 
					node[pos=0.5,inner sep=0mm,minimum size=1mm,name=B4] {$\circ$} 
					node[pos=0.625,inner sep=0mm,minimum size=1mm,name=B5] {$\bullet$} 
					node[pos=0.75,inner sep=0mm,minimum size=1mm,name=B6] {$\circ$} 
					node[pos=0.875,inner sep=0mm,minimum size=1mm,name=B7] {$\bullet$} 
					node[pos=1,name=B8,below left] {\small $A^-$} ;
\draw [fleche] (B2) -- (B1) ;
\draw (B2) -- (B3) -- (B4) -- (B5) -- (B6) -- (B7) ;
\draw [dotted] (B8) -- (B7) ;

\path  (14,-10) -- (3,-3) 	node[pos=0,name=A0,below right] {\small $B^{+}$} 
					node[inner sep=0mm,minimum size=1.5mm,pos=0.125,name=A1] {$\bullet$} 
					node[inner sep=0mm,minimum size=1.5mm,pos=0.375,name=A2] {$\bullet$} 
					node[inner sep=0mm,minimum size=1.5mm,pos=0.625,name=A3] {$\bullet$} 
					node[inner sep=0mm,minimum size=1.5mm,pos=0.875,name=A4]  {$\bullet$} ;
\draw [dotted] (A0) -- (A1) ;
\draw (A1) -- (A2) -- (A3)  ;
\draw [fleche passee] (A4)--(A3) ;

\end{scope}
\end{tikzpicture}
\end{center}
\caption{The two sub-cases of the remaining case}
\label{figure.subcases}
\end{figure}
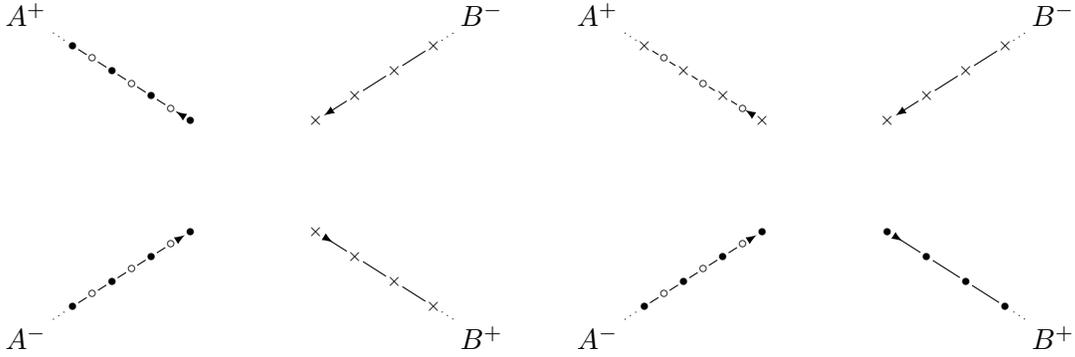

Consider the first case, and choose the numbering of the $\cO_{i}$'s so that $B^{-}$ and $B^{+}$ meet $\cO_{1}$. Let $\Delta$ be a reducing line provided by   Proposition~\ref{prop.reducing-line}, and which is disjoint from $A^{-}, B^{-}$ and $B^{+}$ (note though that $\Delta$ may intersect $A^{+}$). By definition, the $\cO_{i}$'s are not in the same half-plane delimited by $\Delta$; since $A^-$ is disjoint from $\Delta$ and meets $\cO_{2}$ and $\cO_{3}$, we conclude that $\Delta$ separates $A^-$ from $B^-$ and $B^+$. In particular $\alpha^-$ is separated from $B^+$, and $\beta^+$ is separated from $A^-$.
Lemma~\ref{lem.proper} below implies that $\alpha^{-}$ is forward proper, and, when applied symmetrically to $h^{-1}$, that $\beta^{+}$ is backward proper. Denote the corresponding proper homotopy streamlines by $A$ and $B$. Note that $B$ contains $\cO_{1}$ and $A$ contains $\cO_{2}$ and $\cO_{3}$.
Since the reducing line $\Delta$ separates $\alpha^-$ from $\beta^+$, $A$ and $B$ are disjoint. We conclude that $[h;\cO_{1}, \cO_{2}, \cO_{3}]$ is a flow class by applying Lemma~\ref{lem.flow-streamlines2}, which completes the proof in this sub-case. 

\begin{lemm}\label{lem.proper}
Let $\alpha$ be a generalized homotopy translation arc.
Assume that there exist reducing lines $\Delta_{1}, \dots , \Delta_{k}$ which are pairwise disjoint and whose union separates $\alpha$ from all but one of the $A^{+}_{j}$'s provided by Proposition~\ref{prop.half-proper-streamlines}.
Then $\alpha$ is forward proper.
\end{lemm}
The proof of the Lemma is provided in the Appendix.

We turn to the second sub-case. We apply Proposition~\ref{prop.reducing-line} to get a reducing line $\Delta$ which is disjoint from $A^-, B^{-}$ and $B^{+}$. 
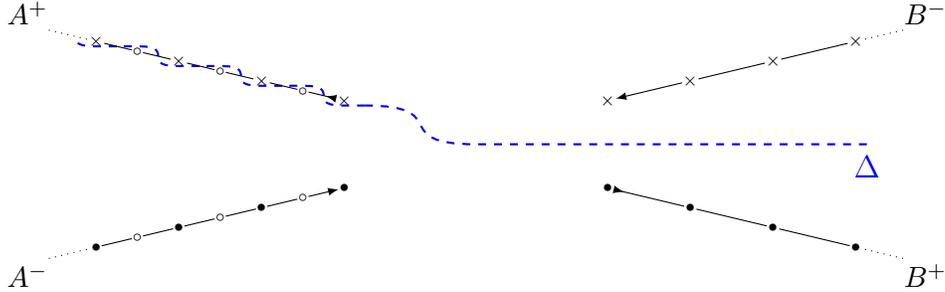
\begin{figure}[h]
\begin{center}
\begin{tikzpicture}[xscale=0.4, yscale=0.15]
\tikzstyle{fleche}=[>=latex,->]
\tikzstyle{fleche passee}=[>=latex,>-]

\tiny

\begin{scope}[color=blue]
\path   (-3,3) -- (-14,10) 	\foreach \q in {0.1875,0.4375,0.6875,0.9375} 
					{node[pos=\q,name=A\q] {} }  
					\foreach \q in {0.0625,0.3125,0.5625,0.8125} 
					{node[pos=\q,name=B\q] {} }  ;
					
\foreach \p / \q in {0.0625/0.1875,0.3125/0.4375,0.5625/0.6875,0.8125/0.9375} 
\draw [thick,dashed] ({B\p}.center) .. controls +(left:1) and +(down:1) ..  ({A\q}.center) ;

\foreach \q / \p in {0.1875/0.3125,0.4375/0.5625,0.6875/0.8125} 
\draw [thick,dashed] ({A\q}.center) .. controls +(up:1) and +(right:1) ..  ({B\p}.center) ;

\normalsize
\draw [thick,dashed] ({B0.0625}.center) .. controls +(right:3) and +(left:3) .. (0,0) -- (13,0) node [below] {$\Delta$} ;
\end{scope}

\path  (14,10) -- (3,3) 	node[pos=0,name=A0,above right] {\small $B^{-}$} 
					node[inner sep=0mm,minimum size=1.5mm,pos=0.125,name=A1] {$\times$} 
					node[inner sep=0mm,minimum size=1.5mm,pos=0.375,name=A2] {$\times$} 
					node[inner sep=0mm,minimum size=1.5mm,pos=0.625,name=A3] {$\times$} 
					node[inner sep=0mm,minimum size=1.5mm,pos=0.875,name=A4]  {$\times$} ;
\draw [dotted] (A1)+(11/7,1) -- (A1) ;
\draw [fleche] (A1) -- (A2) -- (A3) -- (A4) ;

\path  (-3,3) -- (-14,10) 	node[pos=0.125,inner sep=0mm,minimum size=1mm,name=B1] {$\times$}  
					node[pos=0.25,inner sep=0mm,minimum size=1mm,name=B2] {$\circ$} 
					node[pos=0.375,inner sep=0mm,minimum size=1mm,name=B3] {$\times$} 
					node[pos=0.5,inner sep=0mm,minimum size=1mm,name=B4] {$\circ$} 
					node[pos=0.625,inner sep=0mm,minimum size=1mm,name=B5] {$\times$} 
					node[pos=0.75,inner sep=0mm,minimum size=1mm,name=B6] {$\circ$} 
					node[pos=0.875,inner sep=0mm,minimum size=1mm,name=B7] {$\times$} 
					node[pos=1,name=B8,above left] {\small $A^{+}$}  ;
\draw  [fleche passee](B1) -- (B2) ;
\draw (B2) -- (B3) -- (B4) -- (B5) -- (B6) -- (B7) ;
\draw [dotted] (B7)+(-11/7,1) -- (B7) ;

\path  (-3,-3) -- (-14,-10) 	node[pos=0.125,inner sep=0mm,minimum size=1mm,name=B1] {$\bullet$}  
					node[pos=0.25,inner sep=0mm,minimum size=1mm,name=B2] {$\circ$} 
					node[pos=0.375,inner sep=0mm,minimum size=1mm,name=B3] {$\bullet$} 
					node[pos=0.5,inner sep=0mm,minimum size=1mm,name=B4] {$\circ$} 
					node[pos=0.625,inner sep=0mm,minimum size=1mm,name=B5] {$\bullet$} 
					node[pos=0.75,inner sep=0mm,minimum size=1mm,name=B6] {$\circ$} 
					node[pos=0.875,inner sep=0mm,minimum size=1mm,name=B7] {$\bullet$} 
					node[pos=1,name=B8,below left] {\small $A^-$} ;
\draw [fleche] (B2) -- (B1) ;
\draw (B2) -- (B3) -- (B4) -- (B5) -- (B6) -- (B7) ;
\draw [dotted] (B7)+(-11/7,-1) -- (B7) ;

\path  (14,-10) -- (3,-3) 	node[pos=0,name=A0,below right] {\small $B^{+}$} 
					node[inner sep=0mm,minimum size=1.5mm,pos=0.125,name=A1] {$\bullet$} 
					node[inner sep=0mm,minimum size=1.5mm,pos=0.375,name=A2] {$\bullet$} 
					node[inner sep=0mm,minimum size=1.5mm,pos=0.625,name=A3] {$\bullet$} 
					node[inner sep=0mm,minimum size=1.5mm,pos=0.875,name=A4]  {$\bullet$} ;
\draw [dotted] (A1)+(11/7,-1) -- (A1) ;
\draw (A1) -- (A2) -- (A3)  ;
\draw [fleche passee] (A4)--(A3) ;

\end{tikzpicture}
\end{center}
\caption{The situation on the second sub-case of the remaining case \label{figure.remaining}}
\end{figure}
Now $\Delta$ must separate $B^{-}$ from $B^{+}$ and $A^{-}$ (see Figure~\ref{figure.remaining}).
Lemma~\ref{lem.proper} implies that $\beta^{-}$ is forward proper, and that $\beta^{+}$ is backward proper. We denote by $B_{1}$ and $B_{2}$ the corresponding proper homotopy streamlines, and we observe that they are disjoint.
Up to renumbering, we may assume that $\cO_{3}$ does not belong to $B_{1}$ nor to $B_{2}$.  We denote by $U$  the unique connected component of the complement of $B_{1} \cup B_{2}$ 
which contains $\cO_{3}$.
By the straightening principle (Corollary~\ref{coro.straightening}), we may find a homeomorphism $h'$ in the class $[h;\cO_{1}, \cO_{2}, \cO_{3}]$ that fixes $B_{1}$ and $B_{2}$. The open set $U$ is homeomorphic to the plane, and $h'(U)=U$.
The mapping class $[h'_{\mid U}, \cO_{3}]$ is conjugate to $[h,\cO_{3}]$ (see~\cite{handel1999fixed}, p238) which is a Brouwer mapping class, thus it admits a homotopy translation arc. In other words, the mapping class $[h', \cO_{3}]$ admits a homotopy translation arc  $\alpha$ which is included in $U$. Thus it is also a homotopy translation arc for $[h';\cO_{1}, \cO_{2}, \cO_{3}] = [h;\cO_{1}, \cO_{2}, \cO_{3}]$.
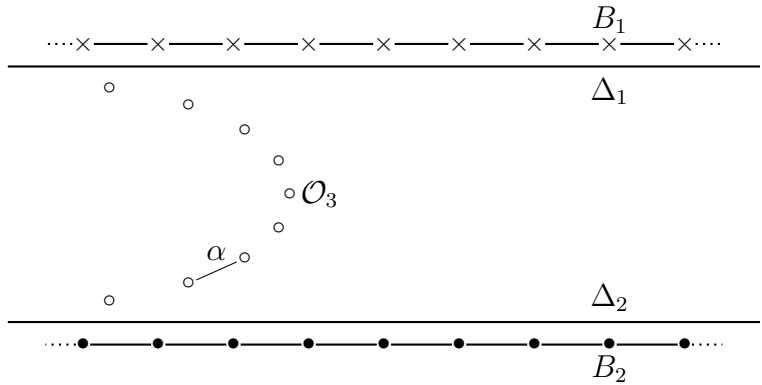
\begin{figure}[h!]
\begin{center}
\begin{tikzpicture}[scale=1]
\tikzstyle{fleche}=[>=latex,->]
\tikzstyle{fleche passee}=[>=latex,>-]
\foreach \k in {1,2,...,9} 
	{\draw (\k,0) node [name=A\k,,inner sep=0mm] {\footnotesize $\bullet$} ; \draw (\k,4) node  [name=B\k,,inner sep=0mm] {\footnotesize$\times$} ; }
\path (0,0.5) .. controls +(right:5) and +(right:5) .. (0,3.5) 
	\foreach \p in {1,2,...,9} {node[pos=\p/10,name=C\p,minimum size=1.5mm,inner sep=1mm] {}} ;
 \foreach \p in {1,2,...,9} { \node at (C\p) [inner sep=0mm] {\footnotesize $\circ$} ;} 
\node at (C5) [right] {$\cO_{3}$};
\draw (C2) -- (C3) node [midway, above] {$\alpha$} ;
\draw [thick] (A1) -- (A2) -- (A3) -- (A4) -- (A5) -- (A6) --(A7) -- (A8) -- (A9) ;
\draw [thick, dotted] (A9) -- (9.5,0) ;
\draw [thick, dotted] (A1) -- (0.5,0) ;
\node at (A8) [below] {$B_{2}$};
\draw [thick] (0,0.3) -- (10,0.3) node [pos=0.8,above] {$\Delta_{2}$} ;
\draw [thick] (B9) --(B8) -- (B7) -- (B6) --(B5) -- (B4) -- (B3) --(B2) -- (B1) ;
\draw [thick, dotted] (B9) -- (9.5,4) ;
\draw [thick, dotted] (B1) -- (0.5,4) ;
\node at (B8) [above] {$B_{1}$};
\draw [thick] (0,3.7) -- (10,3.7) node [pos=0.8,below] {$\Delta_{1}$} ;
\end{tikzpicture} 
\end{center}
\caption{Situation in the second sub-case when $U$ is the middle component}
\label{figure.last-remaining}
\end{figure}
We consider first the case when $U$ is the ``middle'' component, which means that its boundary is $B_{1} \cup B_{2}$. \footnote{Actually it may be proved that this is always the case, so the second case never appears, but it is easier to prove that the second case would also yield a flow class.}
 Let $\Delta_{1},\Delta_{2}$ be two reducing lines provided by Lemma~\ref{lemm.reducing-line}, the first one separates $\cO_{1}$ from $\cO_{2} \cup \cO_{3}$ and the second one separates $\cO_{2}$ from $\cO_{1} \cup \cO_{3}$ (see figure~\ref{figure.last-remaining}). The set $\Delta_{1} \cup \Delta_{2}$ separates $\alpha$ from $B_{1}$ and $B_{2}$, which respectively contain $B^-$ and $B^+$. Two applications of Lemma~\ref{lem.proper} entail that $\alpha$ is both forward and backward proper. Let $A$ be the proper homotopy streamline generated by $\alpha$. This streamline is disjoint from $B_{1}$ and $B_{2}$, and we conclude from Lemma~\ref{lem.flow-streamlines} that $[h;\cO_{1}, \cO_{2}, \cO_{3}]$  is a flow class.

In the case when $U$ is not the middle component,  we may assume that its boundary is $B_{1}$. Likewise, we consider a reducing line $\Delta$ which is contained in a small tubular neighborhood of $B_{1}$ and which separates $B_{1}$ and $B_{2}$ from $\cO_{3}$. The same reasoning as in the previous case provides a proper homotopy streamline $A$ containing $\cO_{3}$ and disjoint from $B_{1}$ and $B_{2}$, as wanted.

\subsection{Explicit description}
\label{subsection.explicit}
We provide a more explicit formulation of Theorem~\ref{theo.3}.
\begin{coro} Up to a renumbering of the orbits and a change of orientation of the plane, 
\begin{itemize}
\item every Brouwer mapping class $[h;\cO_{1}, \cO_{2}]$  is conjugate to the mapping class of $T$ or $R$ relative to $\bbZ \times \{1,2\}$,

\item every Brouwer mapping class $[h;\cO_{1}, \cO_{2}, \cO_{3}]$  is conjugate to one of the five mapping classes described on figure~\ref{figure.explicit}.
\end{itemize}
\end{coro}

The proof is left to the reader.

\begin{figure}[h]
\begin{center}
\begin{tikzpicture}[scale=0.2]
\tikzstyle{fleche}=[>=latex,->]
\tikzstyle{fleche inverse}=[>=latex,<-]

\begin{scope}[xshift=-30cm]\footnotesize
\draw [fleche] (-5,-4) -- (5,-4) \foreach \p in {0.2,0.5,0.8} {node [pos=\p] {$\bullet$}  }  ;
\draw [fleche] (-5,0) -- (5,0) \foreach \p in {0.2,0.5,0.8} {node [pos=\p] {$\times$}  } ;
\draw [fleche] (-5,4) -- (5,4) \foreach \p in {0.2,0.5,0.8} {node [pos=\p] {$\circ$}  } ;
\end{scope}

\begin{scope}[xshift=-15cm]\footnotesize
\draw [fleche] (-5,-4) -- (5,-4) \foreach \p in {0.2,0.5,0.8} {node [pos=\p] {$\bullet$}  }  ;
\draw [fleche] (-5,0) -- (5,0) \foreach \p in {0.2,0.5,0.8} {node [pos=\p] {$\times$}  }  ;
\draw [fleche] (5,4) -- (-5,4) \foreach \p in {0.2,0.5,0.8} {node [pos=\p] {$\circ$}  }  ;
\end{scope}

\begin{scope}\footnotesize
\draw [fleche] (-5,-4) -- (5,-4) \foreach \p in {0.2,0.5,0.8} {node [pos=\p] {$\bullet$}  }  ;
\draw [fleche] (5,0) -- (-5,0) \foreach \p in {0.2,0.5,0.8} {node [pos=\p] {$\times$}  } ;
\draw [fleche] (-5,4) -- (5,4) \foreach \p in {0.2,0.5,0.8} {node [pos=\p] {$\circ$}  } ;
\end{scope}

\begin{scope}[xshift=15cm]\footnotesize
\draw [fleche] (-4,-5) .. controls +(45:1) and  +(left:1.7) .. (-1,-3) node [midway] {$\bullet$} -- (0,-3) node {$\bullet$} -- (1,-3)  .. controls +(right:1.7) and +(135:1) .. (4, -5) node [midway] {$\bullet$} ;
\draw [rotate=120,fleche]  (-4,-5)  .. controls +(45:1) and  +(left:1.7) .. (-1,-3) node [midway] {$\times$} -- (0,-3) node {$\times$} -- (1,-3)  .. controls +(right:1.7) and +(135:1) ..(4, -5) node [midway] {$\times$} ;  
\draw [rotate=240,fleche]  (-4,-5) .. controls +(45:1) and  +(left:1.7) .. (-1,-3) node [midway] {$\circ$} -- (0,-3) node {$\circ$} -- (1,-3) .. controls +(right:1.7) and +(135:1) ..(4, -5) node [midway] {$\circ$} ;
\end{scope}


\begin{scope}[xshift=30cm]\footnotesize
\draw [fleche inverse] (-4,-5) .. controls +(45:1) and  +(left:1.7) .. (-1,-3) node [midway] {$\bullet$} -- (0,-3) node {$\bullet$} -- (1,-3)  .. controls +(right:1.7) and +(135:1) .. (4, -5) node [midway] {$\bullet$} ;
\draw [rotate=120,fleche]  (-4,-5)  .. controls +(45:1) and  +(left:1.7) .. (-1,-3) node [midway] {$\times$} -- (0,-3) node {$\times$} -- (1,-3)  .. controls +(right:1.7) and +(135:1) ..(4, -5) node [midway] {$\times$} ;  
\draw [rotate=240,fleche]  (-4,-5) .. controls +(45:1) and  +(left:1.7) .. (-1,-3) node [midway] {$\circ$} -- (0,-3) node {$\circ$} -- (1,-3) .. controls +(right:1.7) and +(135:1) ..(4, -5) node [midway] {$\circ$} ;
\end{scope}

\end{tikzpicture}
\end{center}
\caption{The five Brouwer mapping classes relative to three orbits}
\label{figure.explicit}
\end{figure}
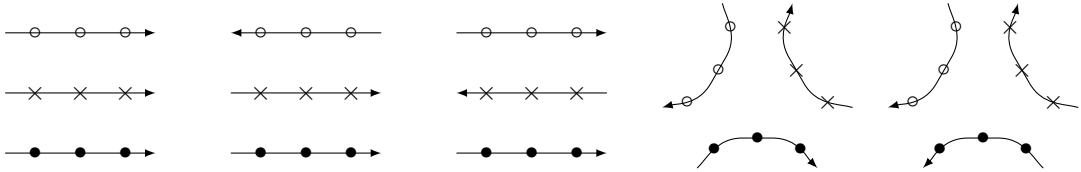

\appendix

\section{Existence of homotopy streamlines and reducing lines}

\begin{proof}[Proof of Proposition~\ref{prop.half-proper-streamlines}]
Choose one point $x_{i}$ in each orbit $\cO_{i}$.
According to Proposition~6.6 in~\cite{handel1999fixed}, each $x_{i}$ is included in a backward proper  homotopy streamline.
Furthermore, if two (generalized or not) backward proper  homotopy streamlines $A^{-}$, ${A'}^{-}$ are not disjoint then Lemma~4.6 in the same paper shows that there is a (generalized) backward proper  homotopy streamline that contains all the $x_{i}$'s that are contained in  $A^{-}$ or ${A'}^{-}$. We may apply this Lemma repeatedly until we get a family $A_{1}^{-}, \dots , A_{r'}^{-}$ of pairwise disjoint (generalized) backward proper  homotopy streamlines whose union contains all the $x_{i}$'s.
Likewise we find a family $A_{1}^{+}, \dots , A_{r''}^{+}$ of pairwise disjoint (generalized) forward proper  homotopy streamlines with the same property. Denote by $\alpha_{i}^{-}, \alpha_{j}^{+}$ the (generalized) homotopy translation arcs that generate $A_{i}^{-}, A_{j}^{+}$.
For every $i,j$ and every positive $k,k'$, if $h^{-k}\alpha_{i}^-$ is not homotopically disjoint from $h^{k'}\alpha_{j}^+$, then $h^{k'+k}(\alpha_{j}^+)$ is not homotopically disjoint from $\alpha_{i}^-$. By properness this cannot happen for arbitrarily large values of $k'$. In other words, the union of the forward streamlines intersects the union of the backward streamlines in  a bounded subset of the plane. Thus, up to replacing each streamline of the second family by some sufficiently large positive iterate, we may assume that the two families are disjoint. Now consider the family $\{ A_{1}^{-}, \dots , A_{r'}^{-}, A_{1}^{+}, \dots , A_{r''}^{+} \}$ whose elements are pairwise disjoint.
Assume that there are two forward proper streamlines, say $A_{1}^+$ and $A_{2}^+$, that are  adjacent for the cyclic order at infinity in this family.
 Using Lemma 3.5 of~\cite{handel1999fixed} we get a homeomorphism $h'$, isotopic to $h$ relative to $\cO = \cO_{1} \cup \cdots \cup \cO_{r}$, for which $A_{1}^+, A_{2}^+$ are positively invariant, \emph{i.e.} $h'(A_{1}^{+}) \subset A_{1}^{+}$ and $h'(A_{2}^{+}) \subset A_{2}^{+}$. 
 Then the adjacency implies the existence of an arc $\gamma$ joining  the end-points of $A_{1}^+$ and $A_{2}^+$, otherwise disjoint from $A_{1}^+, A_{2}^+$,  and such that the ``half-strip'' $U$ bounded by $\gamma \cup A_{1}^+ \cup A_{2}^+$ is disjoint from the $\cO_{i}$'s  (see Figure~\ref{figure.generalized-hta}). Then $h(\gamma)$ is isotopic, relative to $\cO$, to an arc included in $U$. As a consequence, $\gamma$ is forward proper. It is now easy to construct a (generalized) homotopy translation arc $\alpha$ that generates a forward proper homotopy streamline $A^{+}$  containing $\cO \cap ( A_{1}^+ \cup A_{2}^+) $ and that meets no other streamline of the family (see the figure). Then we consider the new family obtained from the first one  by removing $A_{1}^+, A_{2}^+$ and adding $A^{+}$. We repeat this process until we get a family with no adjacent forward proper streamlines. A symmetric process applies to get a family with no adjacent backward proper streamlines either. This completes the proof of the Proposition. 
\end{proof}

\begin{figure}[h]
\begin{center}
\begin{tikzpicture}[xscale=0.6,yscale=0.2]
\tikzstyle{fleche}=[>=latex,->]
\tikzstyle{fleche passee}=[>=latex,>-]

\footnotesize
\path  (3,3) -- (14,7) 	node[pos=0,name=B0] {} 
					node[pos=0.125,inner sep=0mm,minimum size=1mm,name=B1] {$\times$}  
					node[pos=0.25,inner sep=0mm,minimum size=1mm,name=B2] {$\circ$} 
					node[pos=0.375,inner sep=0mm,minimum size=1mm,name=B3] {$\times$} 
					node[pos=0.5,inner sep=0mm,minimum size=1mm,name=B4] {$\circ$} 
					node[pos=0.625,inner sep=0mm,minimum size=1mm,name=B5] {$\times$} 
					node[pos=0.75,inner sep=0mm,minimum size=1mm,name=B6] {$\circ$} 
					node[pos=0.875,inner sep=0mm,minimum size=1mm,name=B7] {$\times$} 
					node[pos=1,name=B8] {} node [pos=1,above right] {$A_{2}^{+}$} ;
\draw [fleche passee] (B1) -- (B2) ;
\draw (B2) -- (B3) -- (B4) -- (B5) -- (B6) -- (B7) ;
\draw [dotted] (B8) -- (B7) ;										

\path  (14,-7) -- (3,-3) node[pos=0,name=A0] {} node[pos=0,below right] {$A_{1}^{+}$} 
					node[pos=0.125,inner sep=0mm,minimum size=1mm,name=A1] {$\bullet$}  
					node[pos=0.25,inner sep=0mm,minimum size=1mm,name=A2] {$+$} 
					node[pos=0.375,inner sep=0mm,minimum size=1mm,name=A3] {$\bullet$} 
					node[pos=0.5,inner sep=0mm,minimum size=1mm,name=A4] {$+$} 
					node[pos=0.625,inner sep=0mm,minimum size=1mm,name=A5] {$\bullet$} 
					node[pos=0.75,inner sep=0mm,minimum size=1mm,name=A6] {$+$} 
					node[pos=0.875,inner sep=0mm,minimum size=1mm,name=A7] {$\bullet$} 
					node[pos=1,name=A8] {}  ;
%
					
\draw [dotted] (A0) -- (A1) ;
\draw (A1) -- (A2) -- (A3)  -- (A4) -- (A5) -- (A6) ;
\draw [fleche passee] (A7)--(A6) ;

\foreach \p in {\bullet,\times,\circ,\bullet,\times,\circ,+,+}
	{\draw (1+rand*2,rand*6) node {$\p$} ;}

\small
\draw (A7) .. controls +(right:3) and +(right:3) .. (B1) node[near start,above left] {$\gamma$} ; 
\draw (A5) .. controls +(75:3) and +(-135:3) .. (B3) node[near start,right] {$h'(\gamma)$} ;

\draw [very thick, dashed]  (A3) -- (A2) -- (B5) -- (B6) -- (A1) node [midway,right] {$h^{2}(\alpha)$} ;

\end{tikzpicture}
\end{center}

\caption{Construction of a generalized homotopy translation arc}
\label{figure.generalized-hta}
\end{figure}
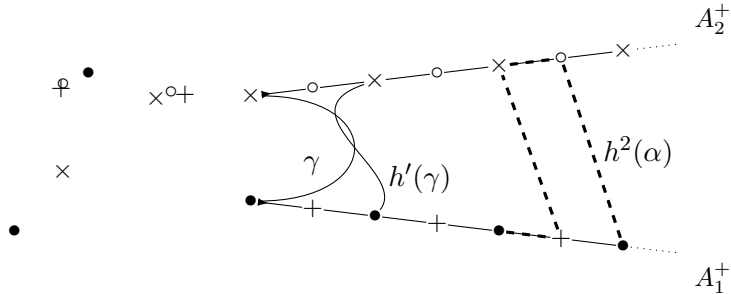

\begin{proof}[Proof of Proposition~\ref{prop.reducing-line}]
All the needed arguments are included in~\cite{handel1999fixed} and~\cite{franks2010entropy}, more precisely in the proofs of Lemma~6.4 of the former and Lemma~8.13 of the latter.
 The main tools are the fitted families of section 5 of~\cite{handel1999fixed}.
 By hypothesis there is a family of disjoint backward and forward (generalized) homotopy streamlines  $A_{1}^-, A_{1}^+, \dots ,A_{r'}^-, A_{r'}^+ $. The corresponding (generalized)  homotopy translation arcs  are denoted by $\alpha_{1}^-, \alpha_{1}^+, \dots ,\alpha_{r'}^-, \alpha_{r'}^+ $.
 For each $i$, there exists a unique geodesic topological line $L_{i}^-$ which is the boundary of a neighbourhood of $A_{i}^-$ whose intersection with $\cO$ equals $A_{i}^- \cap \cO$ (see Figure~\ref{figure.t} below). We define $L_{i}^+$ symmetrically. The closed subset of the plane whose boundary is the union of the $L_{i}^-$'s and the $L_{i}^+$'s is called a \emph{Brouwer subsurface} and denoted by $W$ (see~\cite{handel1999fixed}, section 5). 

The translation case will be treated at the end of the proof, so for the time being let us assume that $r'$ is bigger than $1$. 
 We first consider the  case when one of the backward proper homotopy translation arcs, say  $\alpha_{1}^-$, is also forward proper. In this case, $\alpha_{1}^{-}$ generates a proper homotopy streamline $A$. Assume that the notation has been chosen so that the end-points of $\alpha_{1}^-$ belong to the orbit $\cO_{1}$, and so that this orbit also meets $A_{1}^+$.
For every large enough positive  $n$, the arc $(h^n\alpha_{1}^-)^\sharp$ has its end-points on $A_{1}^{+}$ and is disjoint from all the $A_{i}^-$'s. Since the $A_{i}^+$'s and $A_{i}^-$ alternate in the cyclic order at infinity, by properness, there exists $n_{0}$ such that for every  $n \geq n_{0}$ this arc is also disjoint from  all the $A_{i}^+$'s except $A_{1}^+$. Denote by $V_{1}$  the connected component of the complement of $W$ that contains $A_{1}^{+}$ ; since $V_{1}$ has geodesic boundary, the minimality property of geodesics implies that for every  $n \geq n_{0}$, $(h^n\alpha_{1}^-)^\sharp$ is included in $V_{1}$. In particular, by iterating backward we find that for every $n \leq n_{0}$, $(h^n\alpha_{1}^-)^\sharp$ is  also disjoint from  all the $A_{i}^+$'s except $A_{1}^+$, and we see that $A$ is disjoint from all the $A_{i}^-$'s and all the $A_{i}^+$'s except $A_{1}^+$.

Let $U$ be a connected component of the complement of $A$ in the plane that contains some of the $\cO_{i}$'s, and let $\Delta$ be a reducing line provided by Lemma~\ref{lemm.reducing-line} that separates $U \cap \cO$ from $A \cap \cO$. The reducing line
$\Delta$ may be chosen to be disjoint from all the $A_{i}^-$'s and all the $A_{i}^+$'s except maybe $A_{1}^+$. If $A_{1}^+$ meets more than one of the $\cO_{i}$'s then there nothing left to prove. Let us assume that $\cO_{1}$ is the only one that meets $A_{1}^+$. Note that in this case $\cO_{1}$ is also the only orbit in $\cO$ that meets $\alpha_{1}^{-}$; in other words, $\alpha_{1}^{-}$ is an \emph{ungeneralized} homotopy translation arc.
Let $n>0$ such that $(h^n\alpha_{1}^-)^\sharp$ is included in $V_{1}$. By uniqueness of the homotopy translation arcs relative to one orbit (Lemma~\ref{lemm.uniqueness}), $h^n\alpha_{1}^-$ is isotopic relative to $\cO_{1}$ to some iterate of $\alpha_{1}^+$. Since both arcs belongs to  $V_{1}$ the homotopy may be performed in $V_{1}$, and thus it is also a homotopy relative to the union of the $\cO_{i}$'s. Uniqueness of geodesics in their homotopy class now implies that $A$ contains $A_{1}^+$. Thus, in the case when $A_{1}^+$ meet only one of the $\cO_{i}$'s, the reducing line $\Delta$ is also disjoint from $A_{1}^+$. This completes the proof in the case when one of the backward homotopy translation arc is forward proper.

 Now we have to face the opposite case, when none of the backward proper homotopy translation arc is forward proper. Then the forward iterates of one of the arcs, say $\alpha_{1}^{-}$, give raise to a ``fitted family'' which does not ``disappear under iteration'' (see~\cite{handel1999fixed}, section 5 and the proof of Lemma 6.4, especially the second paragraph).
In order to reformulate Theorem~5.5 of~\cite{handel1999fixed} in our context, we first apply the ``straightening principle'' (Lemma~\ref{lemm.straightening2}, arguing as in the proof of Corollary~\ref{coro.straightening}) to get a homeomorphism $h'$ which is isotopic to $h$ relative to $\cO$ and such that, for every $i=1, \dots, r'$, for every positive $n$, ${h'}^{n}(L_{i}^{+})$ is a geodesic; in particular, these topological lines are pairwise disjoint.  The following is an adapted version of Theorem~5.5 of~\cite{handel1999fixed}.

 \begin{figure}[h!]
\begin{center}
\begin{tikzpicture}[scale=0.3]
\tikzstyle{fleche}=[>=latex,->]
\tikzstyle{fleche passee}=[>=latex,>-]
\footnotesize

\colorlet{lightgray}{black!5}
\fill [thick,even odd rule,color=lightgray] (-13,-10) rectangle (13,10) 
(11,-10) .. controls +(130:0) and +(-135:2) .. (3,-3) .. controls +(45:2) and + (130:0)  .. (13,-7) -- (13,-10) -- cycle
(-11,10) .. controls +(310:0) and +(45:2) ..  (-3,3) .. controls +(-135:2) and + (310:0)  .. (-13,7) -- (-13,10) -- cycle
 [yscale=-1] (11,-10) .. controls +(130:0) and +(-135:2) .. (3,-3) .. controls +(45:2) and + (130:0)  .. (13,-7) -- (13,-10) -- cycle 
   (-11,10) .. controls +(310:0) and +(45:2) .. (-3,3) .. controls +(-135:2) and + (310:0)  .. (-13,7) -- (-13,10) -- cycle ;

\begin{scope}[color=blue]
\path   (-3,3) -- (-14,10) 	\foreach \q in {0.1875,0.4375,0.6875,0.9375} 
					{node[pos=\q,name=A\q] {} }  
					\foreach \q in {0.0625,0.3125,0.5625,0.8125} 
					{node[pos=\q,name=B\q] {} }  ;
					
%
%
\end{scope}

\path  (14,10) -- (3,3) 	node[pos=0,name=A0] {} 
					node[inner sep=0mm,minimum size=1.5mm,pos=0.125,name=A1] {$\times$} 
					node[inner sep=0mm,minimum size=1.5mm,pos=0.375,name=A2] {$\times$} 
					node[inner sep=0mm,minimum size=1.5mm,pos=0.625,name=A3] {$\times$} 
					node[inner sep=0mm,minimum size=1.5mm,pos=0.875,name=A4]  {$\times$} ;
\draw [dotted] (A0) -- (A1) ;
\draw [fleche] (A1) -- (A2) -- (A3) -- (A4) ;

\path  (-3,3) -- (-14,10) 	node[pos=0.125,inner sep=0mm,minimum size=1mm,name=B1] {$\times$}  
					node[pos=0.25,inner sep=0mm,minimum size=1mm,name=B2] {$\circ$} 
					node[pos=0.375,inner sep=0mm,minimum size=1mm,name=B3] {$\times$} 
					node[pos=0.5,inner sep=0mm,minimum size=1mm,name=B4] {$\circ$} 
					node[pos=0.625,inner sep=0mm,minimum size=1mm,name=B5] {$\times$} 
					node[pos=0.75,inner sep=0mm,minimum size=1mm,name=B6] {$\circ$} 
					node[pos=0.875,inner sep=0mm,minimum size=1mm,name=B7] {$\times$} 
					node[pos=1,name=B8] {} ;
					node [pos=0.19,name=B] {} ;
\draw  (B1) -- (B2) ;
\draw (B2) -- (B3) -- (B4) -- (B5) -- (B6) -- (B7) ;
\draw [dotted] (B8) -- (B7) ;

\path  (-3,-3) -- (-14,-10) 	node[pos=0.125,inner sep=0mm,minimum size=1mm,name=B1] {$\bullet$}  
					node[pos=0.25,inner sep=0mm,minimum size=1mm,name=B2] {$\circ$} 
					node[pos=0.375,inner sep=0mm,minimum size=1mm,name=B3] {$\bullet$} 
					node[pos=0.5,inner sep=0mm,minimum size=1mm,name=B4] {$\circ$} 
					node[pos=0.625,inner sep=0mm,minimum size=1mm,name=B5] {$\bullet$} 
					node[pos=0.75,inner sep=0mm,minimum size=1mm,name=B6] {$\circ$} 
					node[pos=0.875,inner sep=0mm,minimum size=1mm,name=B7] {$\bullet$} 
					node[pos=1,name=B8] {} ;
\draw [fleche] (B2) -- (B1) ;
\draw (B2) -- (B3) -- (B4) -- (B5) -- (B6) -- (B7) ;
\draw [dotted] (B8) -- (B7) ;

\path  (14,-10) -- (3,-3) 	node[pos=0,name=A0] {} 
					node[inner sep=0mm,minimum size=1.5mm,pos=0.125,name=A1] {$\bullet$} 
					node[inner sep=0mm,minimum size=1.5mm,pos=0.375,name=A2] {$\bullet$} 
					node[inner sep=0mm,minimum size=1.5mm,pos=0.625,name=A3] {$\bullet$} 
					node[inner sep=0mm,minimum size=1.5mm,pos=0.875,name=A4]  {$\bullet$}
					node[pos=1,name=A5] {} ;
\draw [dotted] (A0) -- (A1) ;
\draw (A1) -- (A2) -- (A3)  ;
\draw [fleche passee] (A4)--(A3) ;

\node at (-4,1) {$\circ$} ;
\node at (1,3) {$\times$} ;
\node at (-5,8) {$\circ$} ;
\node at (-3,-2) {$\bullet$} ;

\normalsize
\draw (11,-10) .. controls +(130:0) and +(-135:2) .. (3,-3) .. controls +(45:2) and + (130:0)  .. (13,-7) node [right] {$L_{i_{1}}^{+}$} ; 
\draw (13,-10) .. controls +(130:0) and +(-135:2) .. (6,-5)   .. controls +(45:2) and + (130:0)  .. (14,-9) node [right] {$h(L_{i_{1}}^{+})$} node [pos=0.2,name=A] {} ; 

\draw (-11,10) .. controls +(310:0) and +(45:2) .. (-3,3) .. controls +(-135:2) and + (310:0)  .. (-13,7) node [left] {$L_{i_{0}}^{+}$} ; 
\draw (-13,10) .. controls +(310:0) and +(45:2) .. (-6.3,4.9) node [pos=0.8,name=C] {}  .. controls +(-135:2) and + (310:0)  .. (-14,9) node [left] {$h(L_{i_{0}}^{+})$} ;

\draw [yscale=-1] (11,-10) .. controls +(130:0) and +(-135:2) .. (3,-3) .. controls +(45:2) and + (130:0)  .. (13,-7) ;
\draw [yscale=-1]  (-11,10) .. controls +(310:0) and +(45:2) .. (-3,3) .. controls +(-135:2) and + (310:0)  .. (-13,7) ;

\draw [dashed, thick,blue] (3,-3) -- (-3,3) node [near start,below] {$t$} ;
\draw [dashed, thick,blue] (A.center) .. controls +(90:4) and +(down:4) .. ({A0.1875}.center)
			.. controls +(up:2) and +(-30:2) .. (-4,9)
			.. controls +(150:2) and +(up:1) .. (C.center) node [pos=0,above] {$h'(t)$} ;

\end{tikzpicture}
\end{center}
\caption{The Brouwer subsurface $W$ and the curve $t$}
 \label{figure.t}
\end{figure}
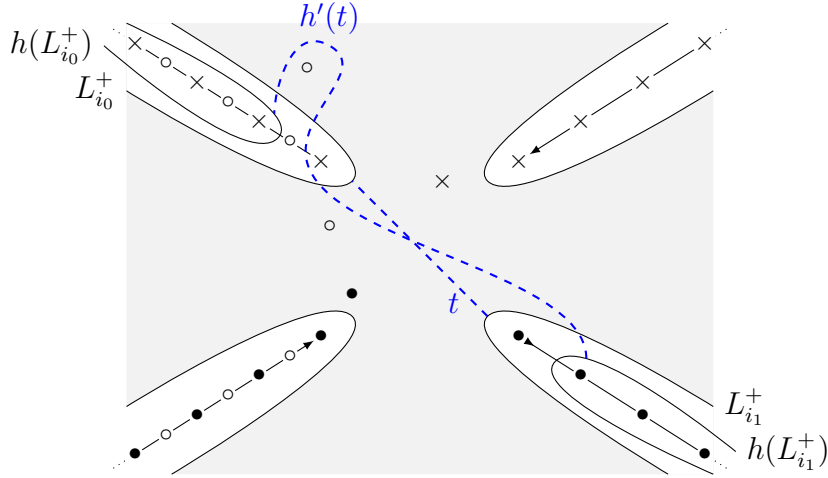

\pagebreak[3]
\begin{theo}[Handel]
\label{theo.fitted}
 Assume that one of the $\alpha_{i}^-$'s is not forward proper. Then there exists a curve $t: [0,1] \to \bbR^2 \setminus \cO$ with the following properties: 
 
 \begin{enumerate}
 \item the end-points $t(0),t(1)$ respectively belong to some $L_{i_{0}}^+$,  $L_{i_{1}}^+$ with $i_{0} \neq i_{1}$, and $t$ is included in $W$; 
 \item $h'(t)$ is homotopic, relative to $\cO$ and its end-points, to a curve $t'$ such that 
 \begin{enumerate}
 \item $t'$ is disjoint from every $L_{i}^-$ and has minimal intersection with every $L_{i}^+$ (in its homotopy class relative to end-points and $\cO$),
 \item $t$ is a connected component of $t' \cap W$,
 \item for every other connected component $s$ of $t' \cap W$, there exists $n>0$ such that ${h'}^n(s)$ is homotopic (relative to its end-points and $\cO$) to a curve $s'$ which is disjoint from $W$. 
\end{enumerate}
 \end{enumerate}
 
\end{theo}

The curves $s$ are said to ``disappear under iteration''. Note in particular that point 2.(c) in the conclusion of the Theorem prevents $t'$ from meeting the $L_{i}^+$'s except $L_{i_{0}}^{+}$ and $L_{i_{1}}^{+}$.

 \begin{figure}[h!]
\begin{center}
\begin{tikzpicture} [scale=0.4]

\begin{scope}
\clip [draw] (0,0) circle (10) ;

\def\rayon{0.5}
\node [name=A] at (77:10) {};
\node [name=B] at (85:10) {};
\node [name=C] at (69:10) {};

\node [name=D] at (66:10) {};
\node [name=E] at (72:10) {};
\node [name=F] at (82:10) {};
\node [name=G] at (88:10) {};

\node [name=AA] at (-77:10) {};
\node [name=BB] at (-85:10) {};
\node [name=CC] at (-69:10) {};

\node [name=DD] at (-66:10) {};
\node [name=EE] at (-72:10) {};
\node [name=FF] at (-82:10) {};
\node [name=GG] at (-88:10) {};

\def\rayonn{0.4}
\node [name=a] at (125:10) {};
\node [name=b] at (119:10) {};
\node [name=c] at (131:10) {};
\node [name=aa] at (-125:10) {};
\node [name=bb] at (-119:10) {};
\node [name=cc] at (-131:10) {};

\colorlet{lightgray}{black!5}
\fill [even odd rule,color=lightgray] (10,-10) rectangle (0,10)  (20,0) circle (18) (A) circle (\rayon) (B) circle (\rayon/2) (C) circle (\rayon/2) (AA) circle (\rayon) (BB) circle (\rayon/2) (CC) circle (\rayon/2) (D) circle (\rayon/4) (E) circle (\rayon/4) (F) circle (\rayon/4) (G) circle (\rayon/4) (DD) circle (\rayon/4) (EE) circle (\rayon/4) (FF) circle (\rayon/4) (GG) circle (\rayon/4) ;
\fill [even odd rule,color=lightgray] (-20,0) circle (18) (-14,0) circle (10)  (a) circle (\rayonn) (b) circle (\rayonn/2) (c) circle (\rayonn/2) (aa) circle (\rayonn) (bb) circle (\rayonn/2) (cc) circle (\rayonn/2) ;


\draw (A) circle (\rayon) ;
\draw (B) circle (\rayon/2) ;
\draw (C) circle (\rayon/2) ;
\draw (D) circle (\rayon/4) ;
\draw (E) circle (\rayon/4) ;
\draw (F) circle (\rayon/4) ;
\draw (G) circle (\rayon/4) ;

\draw (AA) circle (\rayon) ;
\draw (BB) circle (\rayon/2) ;
\draw (CC) circle (\rayon/2) ;
\draw (DD) circle (\rayon/4) ;
\draw (EE) circle (\rayon/4) ;
\draw (FF) circle (\rayon/4) ;
\draw (GG) circle (\rayon/4) ;

\draw (a) circle (\rayonn) ;
\draw (b) circle (\rayonn/2) ;
\draw (c) circle (\rayonn/2) ;

\draw (aa) circle (\rayonn) ;
\draw (bb) circle (\rayonn/2) ;
\draw (cc) circle (\rayonn/2) ;

\draw (0,-10) -- (0,10) ;
\draw  (20,0) circle (18) ;
\draw (14,0) circle (10) ;
\draw  (-20,0) circle (18) ;
\draw (-14,0) circle (10) ;
\draw (-11,0) circle (5) ;

\draw [dashed, thick] (0,0)  -- (2,0) node [near start,above] {$\tilde t$} ;
\draw [dashed, thick] (-6,0) .. controls +(2,-1) and +(-2,-1) .. (1,0)  node [very near start,below] {\small $\tilde{h'}(\tilde t)$}  .. controls +(1,0.5) and +(-0.25,0.5) ..  (4,0) ;
\end{scope}

\coordinate  [xshift=-8cm] (SA) at (14:18)  {} ;
\coordinate  [xshift=-5.6cm] (SB) at (20:10)  {} ;
\draw [dashed, thick] (SA) -- (SB) node [midway,above] {$\tilde s$} ;

\draw [very thick,color=blue] (0,10)  arc (90:270:10) ;
\draw [very thick,color=blue] (10,0)  arc (0:64:10) ;
\draw [very thick,color=blue] (10,0)  arc (0:-64:10) ;

\node at (-10,0) [left] {\textcolor{blue}{$I_{0}$}};
\node at (10,0) [right] {\textcolor{blue}{$I_{1}$}};
\node at (0,-10) [below] {$\tilde L_{i_{0}}^{+}$} ;
\node at (5,-9.8) {$\tilde L_{i_{1}}^{+}$} ;
\node at (8.4,-8) {$\tilde{h'} (\tilde L_{i_{1}}^{+})$} ;
\node at (-11,-5) {$\tilde{h'} (\tilde L_{i_{0}}^{+})$} ;

\end{tikzpicture}
\end{center}
 \caption{The same situation in the universal cover ($\tilde W$ is in grey)}
 \label{figure.universal}
 \end{figure}
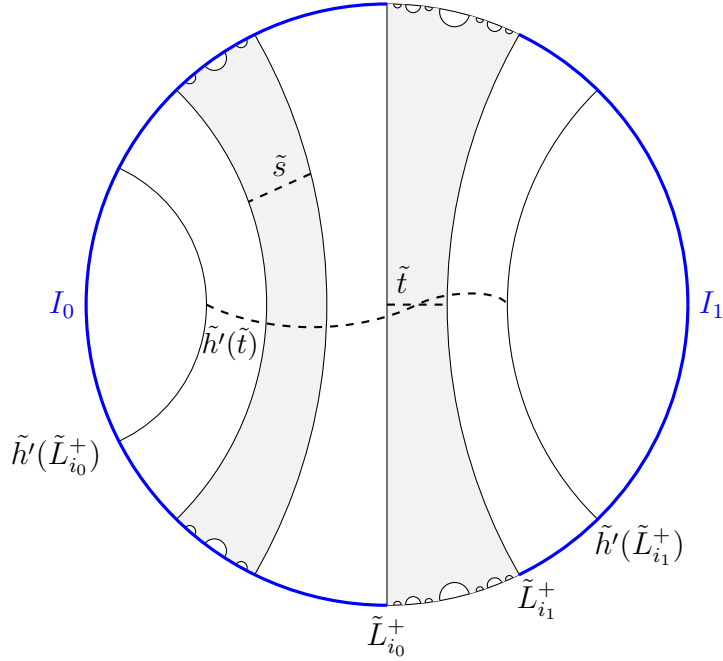

Here is the translation in terms of the universal cover $\pi : \bbH^{2} \to \bbR^{2} \setminus \cO$. 
 Assume again that some $\alpha_{i}^-$ is not forward proper, and let $i_{0}, i_{1}$ be as in the previous Theorem.
Then there exists a lift $\tilde h': \bbH^{2} \to \bbH^{2}$ of $h'$,   and lifts $\tilde L_{i_{0}}^{+}, \tilde L_{i_{1}}^{+}$ of $L_{i_{0}}^+$,  $L_{i_{1}}^+$ with the following properties (see figure~\ref{figure.universal}):

\begin{itemize}
\item $\tilde h' \tilde L_{i_{0}}^{+}$ is separated by  $\tilde L_{i_{0}}^{+}$ from $\tilde L_{i_{1}}^{+}$, and $\tilde L_{i_{0}}^{+}$ is separated by  $\tilde L_{i_{1}}^{+}$ from $\tilde h' \tilde L_{i_{1}}^{+}$ (you don' really need to read that sentence, just look at the picture), 
\item For every connected component $\mathbf{s}$ of $\pi^{-1} W$ that separates $\tilde h \tilde L_{i_{0}}^{+}$ from $\tilde h \tilde L_{i_{1}}^{+}$, let $\tilde L, \tilde L'$ be the two boundary components of $\mathbf{s}$ that separate $\tilde L_{i_{0}}^{+}$ from $\tilde L_{i_{1}}^{+}$ (both are lifts of either $L_{i_{0}}^{+}$ or 
$L_{i_{1}}^{+}$). Then there exists some $n>0$ such that $\pi^{-1} W$ does not separate ${\tilde {h'}}^{n} \tilde L$ from ${\tilde {h'}}^{n} \tilde L'$.
\end{itemize}
The lift $\tilde h'$ is obtained as follows: first choose a lift $\tilde t$ of $t$, then denote by $\tilde t'$ the unique lift of $t'$ which contains $\tilde t$, consider the lift $\tilde t''$ of $h(t)$ which has the same end-points as $\tilde t'$, and define $\tilde h'$ as the unique lift of $h'$ for which $\tilde h'(\tilde t) = \tilde t''$.

Now we follow the construction of a reducing line  in the proof of Lemma 6.4 in~\cite{handel1999fixed}. We work in the universal cover, and consider its circle boundary $\partial \bbH^{2}$. The map $\tilde h'$ extends to a homeomorphism $\tilde h' : \bbH^2 \cup \partial \bbH^{2} \to \bbH^2 \cup \partial \bbH^{2}$ (see~\cite{handel1999fixed}, Proposition 3.1).
Since $A_{i_{0}}^{+}$ and $A_{i_{1}}^{+}$ are separated by the backward homotopy streamlines in the cyclic order at infinity,  the geodesics 
$\tilde L_{i_{0}}^{+}$ and $\tilde L_{i_{1}}^{+}$ have distinct end-points in  $\partial \bbH^{2}$.
Let $I_{0}, I_{1}$ be the disjoint closed intervals in $\partial \bbH^{2}$ whose end-points are respectively the end-points of $\tilde L_{i_{0}}^{+}$ and $\tilde L_{i_{1}}^{+}$ (see the picture). The properties of $\tilde h'$ entail that $\tilde h' (I_{0}) \subset I_{0}$ and thus $\tilde h'$
 has at least one  fixed point in the interval $I_{0}$. Furthermore this fixed point is actually unique. Indeed, if the opposite would  hold then the sequence of geodesics $\tilde {h'}^{n}(\tilde L_{i_{0}}^{+})$ for $n \geq 0$ would accumulate in $\bbH^{2}$, and thus the sequence ${h'}^{n}(L_{i_{0}}^{+})$ would accumulate in the plane. But since the forward homotopy streamline $A_{i_{0}}^{+}$ is a proper, the sequence ${h'}^{n}(L_{i_{0}}^{+})$ is homotopically proper, and this would contradicts point 3 of Lemma~\ref{lemm.straightening} (or more precisely the version for topological lines stated in Lemma~\ref{lemm.straightening2}).

 Likewise $\tilde h'$ has a unique fixed point in $I_{1}$.
  Let $\tilde \lambda$ be a geodesic joining those two fixed points, and $\lambda$ be its projection in $\bbR^{2}\setminus \cO$.
Using property (2.c) of Theorem~\ref{theo.fitted}, one can prove that $\lambda$ is properly embedded in the plane (see~\cite{handel1999fixed}, last paragraph of the proof of Lemma~6.4), and since the end-points of $\tilde \lambda$ are fixed by $\tilde h'$, $\lambda$ is a reducing line for $h'$, and thus also for $h$. 

It remains to check that $\lambda$ is well positioned relative to our family of streamlines. By construction no component of some $\pi^{-1}( L_{i}^{-})$ separates the end-points of $\tilde \lambda$, thus $\lambda$ is disjoint from all the $A_{i}^{-}$'s. Since the $s$'s disappear under iteration, $\lambda$ does not meet the $L_{i}^+$'s except maybe $L_{i_{0}}^{+}$ and $L_{i_{1}}^{+}$. Thus $\lambda$ is also disjoint from all the $A_{i}^{+}$'s except maybe $A_{i_{0}}^{+}$ and $A_{i_{1}}^{+}$.

Now assume that the forward proper homotopy streamline $A_{i_{0}}^{+}$ meets a single orbit $\cO_{j}$, and let us prove that $\lambda$ is disjoint from $A_{i_{0}}^{+}$. 
 (Since  $A_{i_{0}}^{+}$ and $A_{i_{1}}^{+}$ plays symmetric roles, this will complete the proof of point 3 in Proposition~\ref{prop.reducing-line}.)
We follow the proof of Lemma 8.13 in~\cite{franks2003periodic}. 


\begin{lemm}
In this situation, no connected component of $t' \cap W$ has both end-points on $L_{i_{0}}^{+}$; in other words, all the $s$'s in Theorem~\ref{theo.fitted} have both end-points on $L_{i_{1}}^{+}$.
\end{lemm}
\begin{proof}
We argue by contradiction. According to  Theorem~\ref{theo.fitted} we may write $t'$ as the concatenation 
$$
t' = \mu_{0} \star s_{1} \star \cdots \star s_{k} \star \mu_{k} \star t \star \mu'
$$
 where $k \geq 1$ and each $\mu_{i}$ is disjoint from $W$ except from its end-point, and the $s_{i}$'s are the component of $t' \cap W$ whose end-points are on $L_{i_{0}}^{+}$. Let $B_{1}$ be the closed strip bounded by $L_{i_{0}}^{+}$ and $h'(L_{i_{0}}^{+})$. 
 Note that $\mu_{1}$ is included in $B_{1}$, and has both end-points on $L_{i_{0}}^{+}$. 
 By point (2.c) of Theorem~\ref{theo.fitted}, there exists $n_{1}>0$ such that ${h'}^{n_{1}} (s_{1})$ is homotopic to a curve $s'$ which is disjoint from $L_{i_{0}}^{+}$, let us choose $n_{1}$ to be the smallest positive integer having this property. Note that $s'$ is included in the connected component $V$ of $\bbR^{2} \setminus L_{i_{0}}^{+}$ that contains $A_{i_{0}}^{+}$.

We will draw two pictures in $\bbH^{2}$, these pictures will turn out to be incompatible, which will provide the desired contradiction.
For the first picture, we let $\tilde t'$ be the lift of $t'$ which is homotopic to ${\tilde h}'(\tilde t)$, and $\tilde \mu_{1}$ be the lift of $\mu_{1}$ included in $\tilde t'$ (see  the left-hand side of figure~\ref{fig.contradiction}). The end-points of $\tilde \mu_{1}$ belong to two lifts of $L_{i_{0}}^{+}$, and since $t'$ has minimal intersection with $L_{i_{0}}$, each of these lifts separates the end-points of $\tilde{t'}$, and thus it separates also  ${\tilde h}' \tilde L_{i_{0}}^{+}$ from ${\tilde h}' \tilde L_{i_{1}}^{+}$ since these geodesics contain the end-points of $\tilde t'$. Since the intervals $I_{0}, I_{1}$ in $\partial \bbH^{2}$ are positively invariant under $\tilde h$, the geodesics $\tilde {h'}^{n_{1}+1} \tilde L_{i_{0}}^{+}$ and $\tilde {h'}^{n_{1}+1} \tilde L_{i_{1}}^{+}$ are separated by all the previous lines, as on the picture. Let $\tilde \tau$ be the geodesic arc joining both end-points of  $\tilde {h'}^{n_{1}} (\tilde t')$. From this first picture we draw the conclusion that $\tilde \tau$ contains an arc $\tilde \tau_{1}$ included in the lift $\tilde B_{1}$ of $B_{1}$, with end-points on the lifts of $L_{i_{0}}^{+}$ which contain the end-points of $\tilde \mu_{1}$.

\begin{figure}
\begin{center}
\begin{tikzpicture} [scale=0.28]

\begin{scope}[xshift=-13cm] 
\clip [draw] (0,0) circle (10) ;

\colorlet{lightgray}{black!5}
\fill [even odd rule,color=lightgray] (-20,0) circle (18) (-14,0) circle (10) ;

\draw (0,-10) -- (0,10) ;
\draw  (20,0) circle (18) ;
\draw (14,0) circle (10) ;
\draw  (-20,0) circle (18) ;
\draw (-14,0) circle (10) ;
\draw (-11,0) circle (5) ;
\draw (-10,0) circle (2) ;
\draw (10,0) circle (2) ;

\draw [dashed, thick] (-6,0) -- (4,0) node [above right] {$\tilde{t'}$} ;
\draw [thick] (0,0) -- (2,0) node [midway,above] {$\tilde{t}$} ;
\draw [thick] (-4,0) -- (-2,0) node [midway,above] {$\tilde{\mu_{1}}$} ;
\draw [thick] (-4,-0.5) -- (-2,-0.5) node [midway,below] {$\tilde{\tau_{1}}$} ;
\draw [dashed, thick] (-8,-0.5) -- (8,-0.5) node [very near end,below] {$\tilde{\tau}$} ;

\end{scope} 
\begin{scope}[xshift=-13cm]
\footnotesize
\node at (0,-10) [below] {$\tilde L_{i_{0}}^{+}$} ;
\node at (5,-10) {$\tilde L_{i_{1}}^{+}$} ;
\node at (8.8,-8.2) {$\tilde{h'} (\tilde L_{i_{1}}^{+})$} ;
\node at (-11,-5) {$\tilde{h'} (\tilde L_{i_{0}}^{+})$} ;
\node at (-13,-2) {$\tilde{h'}^{n_{1}+1} (\tilde L_{i_{0}}^{+})$} ;
\node at (10,-2) [below] {$\tilde{h'}^{n_{1}+1} (\tilde L_{i_{1}}^{+})$} ;
\end{scope}

\begin{scope}[xshift=13cm] 
\clip [draw] (0,0) circle (10) ;

\colorlet{lightgray}{black!5}
\fill [even odd rule,color=lightgray] (0,-10) rectangle (10,10)   (20,0) circle (18) ;

\draw (0,-10) -- (0,10) ;
\draw  (20,0) circle (18) ;
\draw (14,0) circle (10) ;
\draw  (-20,0) circle (18) ;
\draw (-10,0) circle (2) ;
\draw (10,0) circle (2) ;

\draw [thick] (0,-0.5) -- (2,-0.5) node [midway,below] {$\tilde{\tau_{2}}$} ;
\draw [dashed, thick] (-8,-0.5) -- (8,-0.5) node [very near end,below] {$\tilde{\tau}$} ;

\end{scope} 
\begin{scope}[xshift=13cm]
\footnotesize
\node at (0,-10) [below] {$\tilde M$} ;
\node at (5,-10) {$\tilde M'$} ;
\node at (-5,-10) {$\tilde N$} ;
\node at (8.8,-8.2) {$\tilde N'$} ;
\node at (-10,2) [above] {$\tilde{h'}^{n_{1}+1} (\tilde L_{i_{0}}^{+})$} ;
\node at (13.5,2) {$\tilde{h'}^{n_{1}+1} (\tilde L_{i_{1}}^{+})$} ;
\end{scope}

\end{tikzpicture}
\end{center}

\caption{Two incompatible pictures in the universal cover ($\tilde B_{1}$ is in grey)}
\label{fig.contradiction}
\end{figure}
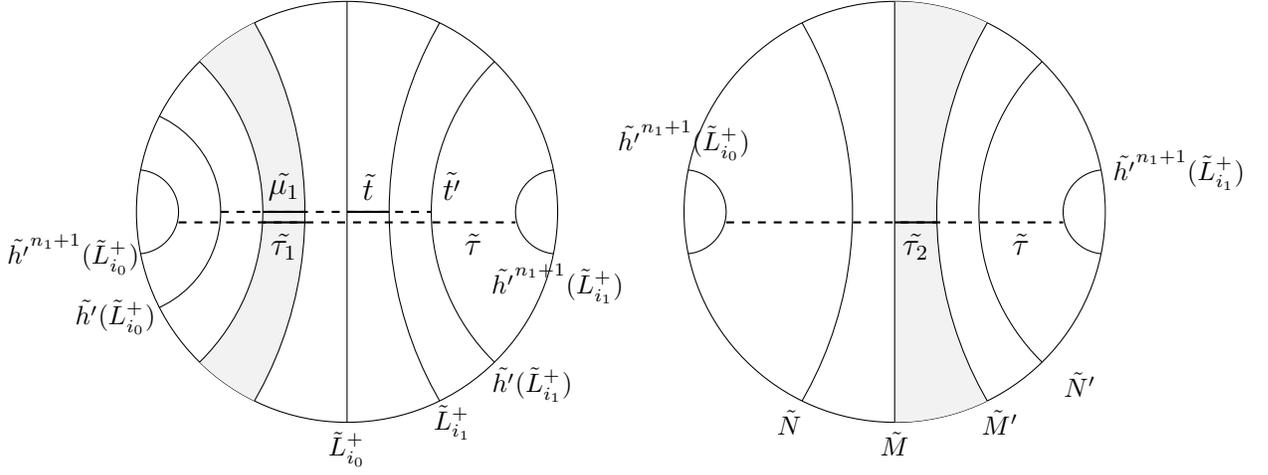

For the second picture, which we draw on the right-hand side of figure~\ref{fig.contradiction}, we let $\tilde s_{1}'$ be the lift of ${h'}^{n_{1}} (s_{1})$ which is included in $\tilde {h'}^{n_{1}} (\tilde t')$. 
By the minimality property of $n_{1}$ we get that  $\tilde B_{1}$ (or the interior of $\tilde B_{1}$ when $n_{1}=1$) separates both end-points of $\tilde s_{1}'$. 
But ${h'}^{n_{1}} (s_{1})$ is homotopic to $s'$ which is disjoint from $L_{i_{0}}^{+}$, thus no lift of $L_{i_{0}}^{+}$ separates both end-points of $\tilde s_{1}'$, and we conclude that there exists two lifts $\tilde M, \tilde {M'}$ of $h'(L_{i_{0}}^{+})$, both in the boundary of the same connected component of $\tilde B_{1}$, and both separating the end-points of  $\tilde s_{1}'$ (or, if $n_{1}=1$, such that each one contains one end-point of $\tilde s_{1}'$). These end-points belong to two lifts $\tilde N,\tilde N'$ of ${h'}^{n_{1}}  L_{i_{0}}^{+}$ which, as a consequence, are also separated by $\tilde M$ and  $\tilde {M'}$ in the case when $n_{1}>1$ (whereas in the case when $n_{1}=1$ we get $\tilde N = \tilde M$ and $\tilde N' = \tilde M'$). Now remember that ${h'}^{n_{1}} (s_{1})$ is included in ${h'}^{n_{1}} (t')$ which has minimal intersection with ${h'}^{n_{1}}  L_{i_{0}}^{+}$. Thus $\tilde N$ and $\tilde N'$ separate both end-points of $\tilde {h'}^{n_{1}} (\tilde t')$. 
From this second picture we draw the conclusion that the geodesic arc $\tilde \tau$ contains an arc $\tilde \tau_{2}$ included in $\tilde B_{1}$ and with end-points on $\tilde M$ and $\tilde {M'}$.

\begin{figure}[h!]
\begin{center}
\begin{tikzpicture} [scale=1]
\colorlet{lightgray}{black!5}
\fill [even odd rule,color=lightgray] (0,0) .. controls +(right:7) and +(right:7) .. (0,4) -- cycle 
(0,1) .. controls +(right:5) and +(right:5) .. (0,3) -- cycle ;

\draw  (0,0) .. controls +(right:7) and +(right:7) .. (0,4) node [inner sep=0,pos=0.4,name=A] {} node [inner sep=0,pos=0.6,name=B] {} ;
\draw  (0,1) .. controls +(right:5) and +(right:5) .. (0,3) node [inner sep=0,pos=0.4,name=C] {} node [inner sep=0,pos=0.6,name=D] {} ;

\node at (4.5,2) {$\times$} ;
\draw [thick,dashed] (A) .. controls +(left:1.4) and + (left:1.4) ..(B) node [right] {$\tau_{1}$} ;
\draw [thick,dashed] (C) .. controls +(right:2) and + (right:2) ..(D) node [left] {$\tau_{2}$};

\node at (1,3.5) {$B_{1}$} ;

\end{tikzpicture}
\end{center}
\caption{Contradiction: the simple curve $\tau$ in  in $B_{1}$}
\label{fig.single-orbit}
\end{figure}
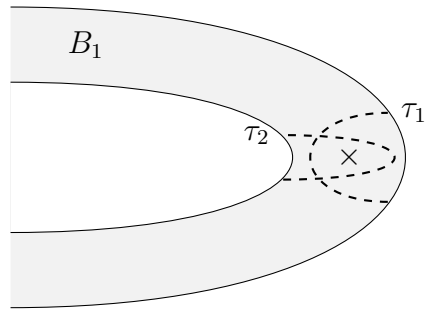

 We project the arcs $\tilde \tau, \tilde \tau_{1},\tilde \tau_{2}$ down in the plane, getting curves $\tau, \tau_{1},\tau_{2}$.
 The curve $\tau_{1}$ is included in $B_{1}$ and has its end-points on $L_{i_{0}}^{+}$ and is not homotopic to an arc included in $L_{i_{0}}^{+}$.
 The curve $\tau_{2}$ is also included in $B_{1}$ and  has its end-points on $h'(L_{i_{0}}^{+})$ and is not homotopic to an arc included in $h'(L_{i_{0}}^{+})$.
 Now we need the crucial hypothesis that  $A_{i_{0}}^{+}$ meets a single orbit $\cO_{j}$: as a consequence, the strip $B_{1}$ contains a single point $x$ of $\cO$ (see figure~\ref{fig.single-orbit}). 
Thus $\tau_{1}$ and $\tau_{2}$ must intersect. The wanted contradiction comes from the fact that both curves are included in $\tau$, combined with the observation that the geodesic curve $\tau$ is a simple curve.\footnote{Note that every curve is homotopic, relative to its end-points and $\cO$, to a simple curve; thus this observation does not follow automatically from the fact that $h'^{n_{1}}t'$ is a simple curve.}
 To prove this last observation, we first recall that  $\tau$ is homotopic, relative to its end-points and $\cO$, to the curve $h'^{n_{1}}t'$, which is  included in ${h'}^{n_{1}+1}(W)$. Since this set has geodesic boundary, it follows that $\tau$ is also included in  ${h'}^{n_{1}+1}(W)$.
If $\tau$ was not a simple curve then $\tilde \tau$ would meet another lift $\tilde \tau'$ of $\tau$; then, denoting by $\tilde W'$ the connected component of $\pi^{-1}({h'}^{n_{1}+1} W)$ that contains $\tilde \tau$ and $\tilde \tau'$,
$\tilde \tau'$ would separate the end-points of $\tilde \tau$ inside $\tilde W'$, which would contradict the fact that ${h'}^{n_{1}} (t')$ is a simple curve included in ${h'}^{n_{1}+1}(W)$.
This completes the proof of the Lemma.
\end{proof}

Owing to the Lemma we know that $t' = \mu_{0} \star t \star \mu'$ where $\mu_{0}$ is an arc crossing $B_{1}$. Since $B_{1}$ contains a single point of $\cO$, one of the two connected component of $B_{1} \setminus \mu_{0}$ contains no point of $\cO$. Using the definition of $\tilde h'$  we see that the geodesics $\tilde L_{i_{0}}^{+}$ and $\tilde{h'}(\tilde L_{i_{0}}^{+})$ have a common end-point on $\partial \bbH^{2}$; in other words one of the end-points of $I_{0}$, say $e_{1}$, is fixed by $\tilde{h'}$. We have seen above that $e_{1}$ is the only fixed point of $\tilde h'$ in $I_{0}$, and thus it is an end-point of $\tilde \lambda$.

In the case when $A_{i_{1}}^{+}$ also meets a single orbit in $\cO$, likewise the other end-point $e_{2}$ of $\tilde \lambda$ is an end-point of $I_{1}$, from which it is easy to deduce that $\tilde \lambda$ is included in $\tilde W$: then $\lambda$ is included in $W$, and in particular it is disjoint from $A_{i_{0}}^{+}$.
Now assume that $e_{2}$ is in the interior of $I_{1}$. 
Since $\tilde W$ has geodesic boundary and $e_{1}$ is an end-point of $\tilde L_{i_{0}}^{+}$, the half-geodesic in $\tilde \lambda$ from  $e_{1}$ to $\tilde L_{i_{1}}^{+}$ is included in $\tilde W$. To see that $\lambda$ is disjoint from $A_{i_{0}^{+}}$ we are left to  prove that it is disjoint from $L_{i_{0}}$. Assume by contradiction that $\tilde \lambda$ meets some lift $\tilde L$ of $L_{i_{0}}$. Then $\tilde L$ is separated from $e_{1}$ by $\tilde L_{i_{1}}^{+}$. Since the sequence of geodesics $\tilde{h'}^{n}(\tilde L_{i_{1}}^{+})$ converges to $e_{2}$ there exists $n>0$ such that $\tilde L$ separates the end-points of $\tilde {h'}^{n} (t)$. Since $\tilde L \neq \tilde L_{i_{0}}^{+}$ there are at least two connected component of $\pi^{-1}(W)$ that separate  the end-points of $\tilde {h'}^{n} (t)$. This contradicts point 2.(c) in Theorem~\ref{theo.fitted}.

\bigskip	

We finally treat the translation case. In this case $r'=1$, and thus point 1 in the conclusion of  Theorem~\ref{theo.fitted} may not hold. We conclude that the homotopy translation arc $\alpha_{1}^-$ is forward proper. It generates a proper homotopy streamline $A$ that contains all the orbits $\cO_{i}$. Using Alexander's trick and the ``straightening principle'' as in the proof of Lemma~\ref{lem.flow-streamlines}, we conclude that $[h;\cO_{1}, \dots , \cO_{r}]$ contains a homeomorphism which is conjugate to a translation. For a translation class relative to at least two orbits we may easily find a reducing line, thus the proof is also complete in this case.
\end{proof}

\begin{proof}[Proof of Lemma~\ref{lem.proper}]
If $\alpha$ was not forward proper, its iterates would give rise to a fitted family relative to the Brouwer subsurface $W$  as in Theorem~\ref{theo.fitted}. In particular, point one of the conclusion of the Theorem says that the iterates of $\alpha$ would not be homotopically disjoint from some  $L_{i_{0}}^+$,  $L_{i_{1}}^+$ with $i_{0} \neq i_{1}$.
But this may not happen since, due to the reducing lines, the iterates of $\alpha$ are homotopically disjoint from all the forward proper homotopy streamlines but $A_{1}^+$.
Details are left to the reader.
\end{proof}

\section{The Schoenflies-Homma theorem}
\label{sec.schoenflies-homma}
Remember that a \emph{topological line} is the image of the real line under a proper injective continuous map $\Phi : \bbR \to \bbR^{2}$. The Schoeflies extension theorem on the sphere can be rephrased by saying that whenever $\Phi_{1}, \Phi_{2}$ are two proper injective continuous maps from the real line to the plane, there exists an orientation preserving homeomorphism $\Phi$ of the plane such that $\Phi \Phi_{1} = \Phi_{2}$. As a consequence, the two connected components of the complement of an oriented topological line may be labeled \emph{right hand side} and \emph{left hand side} in a way which is compatible with the action of orientation preserving homeomorphisms. The following theorem is a generalisation of the Schoenflies theorem due to Homma (\cite{homma1953jordan}).

\begin{theo*}
Let $\cF,\cF'$ be two locally finite families of pairwise disjoint topological oriented lines in the plane.
Assume that for each $F \in \cF$, there exists some $F' \in \cF'$ and some orientation preserving homeomorphism $\Phi_{F}: F \to F'$, in such a way that the map $F \mapsto \Phi_{F}(F)$ is a bijection between $\cF$ and $\cF'$.
Assume that  the correspondance $F \mapsto \Phi_{F}(F)$ \emph{preserves the combinatorics}: for every $F_{1}, F_{2} \in \cF$, if $F_{2}$ is on the right-hand side of $F_{1}$, then $\Phi_{F_{2}} (F_{2})$ is on the right-hand side of $\Phi_{F_{1}} (F_{1})$.

Then there exists an orientation preserving  homeomorphism $\Phi: \bbR^2 \to \bbR^2$ such that for every $F \in \cF$, $\Phi_{\mid F} = \Phi_{F}$.
\end{theo*}
\begin{proof}
We work with the Alexandrov compactification $\bbR^2 \cup \{\infty\}$, which is homeomorphic to the two-sphere. Let $M$ be the union of the elements of $\cF$ with the point $\infty$. Then $M$ is locally connected, compact and connected. Define $M'$ analogously, and let $\Phi :M \to M'$ d\'efined by $\Phi(\infty) = \infty$ and $\Phi(x) = \Phi_{F}(x)$ for every $x \in F$. 
Under the hypotheses of the theorem, the map $\Phi$ preserves the cyclic order at infinity on the germs of topological lines included in $M$. Thus we may apply Homma's theorem (\cite{homma1953jordan}) which provides the desired extension of $\Phi$ to an orientation preserving homeomorphism of the plane.
\end{proof}
\begin{coro*}
Let $\cF$ be as in the previous theorem, and $U$ be a connected component of the complement of $\cup\{F, F \in \cF\}$. Then the closure of $U$ in the plane is homeomorphic to the closed unit disk minus a finite subset of the circle boundary.
\end{coro*}
In other words, the end-compactification of the closure of $U$ is a closed disk.
\begin{proof}
We provide a proof only in the finite case, since the infinite case is not used in the present text.
We may assume that every element of $\cF$ is included in the boundary of $U$. Up to changing the orientation of the $F$'s, we may also assume that $U$ is on the left-hand side of every $F$ in $\cF$. Then every $F_{2}$ in $\cF$ is on the left-hand side of every $F_{1} \neq F_{2}$ in $\cF$. We leave it to the reader to construct an explicit family $\cF'$, having the same combinatorics as $\cF$, and an explicit homeomorphism between the complement in $\bbD^2$ of a finite subset of $\partial \bbD^2$ and the closure of the unique connected component $U'$ whose boundary contains every element of $\cF'$. Now the theorem provides a homeomorphism of the plane that sends $\cF$ to $\cF'$, and in particular that sends the closure of $U$ to the closure of $U'$. We deduce that the closure of $U$ is also homeomorphic to the unit disk minus a finite subset of the boundary.
\end{proof}

\bibliographystyle{alpha}
\bibliography{bibliographie-index}

\end{document}